\documentclass[11pt]{amsart}
\allowdisplaybreaks
\usepackage[english]{babel}
\usepackage{fullpage}
\usepackage{amsfonts}
\usepackage{amsmath}
\usepackage{amsthm}
\usepackage{amssymb}
\usepackage{bbm}
\usepackage{cancel}
\usepackage{color}
\usepackage[all]{xy}
\usepackage{graphicx}
\usepackage{mathabx}
\usepackage{tikz-cd}
\tikzcdset{scale cd/.style={every label/.append style={scale=#1},
    cells={nodes={scale=#1}}}}
\usepackage{epigraph} 
\usepackage{hyperref}
\usepackage{stmaryrd}
\usepackage{enumitem}
\usepackage[mathscr]{eucal}
\usepackage{soul}
\usepackage{mathtools}
\usepackage{xfrac}
\usepackage{kotex}

\newcommand{\defterm}[1]{\textbf{\emph{#1}}}
\usepackage{aliascnt}
\usepackage{cleveref}

\numberwithin{equation}{section} 

\usepackage[norefs,nocites]{refcheck} 

\makeatletter
\newcommand{\refcheckize}[1]{%
  \expandafter\let\csname @@\string#1\endcsname#1%
  \expandafter\DeclareRobustCommand\csname relax\string#1\endcsname[1]{%
    \csname @@\string#1\endcsname{##1}\wrtusdrf{##1}}%
  \expandafter\let\expandafter#1\csname relax\string#1\endcsname
}

\AtBeginDocument{%
  \refcheckize{\cref}%
  \refcheckize{\Cref}%

  \expandafter\let\csname RC@orig@eqref\endcsname\eqref
  \expandafter\def\expandafter\eqref\expandafter#\expandafter1\expandafter{%
    \csname RC@orig@eqref\endcsname{#1}\wrtusdrf{#1}%
  }%

  \expandafter\let\csname RC@orig@ref\endcsname\ref
  \expandafter\def\expandafter\ref\expandafter#\expandafter1\expandafter{%
    \csname RC@orig@ref\endcsname{#1}\wrtusdrf{#1}%
  }%
}
\makeatother

\usepackage{scalerel,stackengine}
\newcommand\reallywidehat[1]{%
\begingroup
\stackMath
\savestack{\tmpbox}{\stretchto{%
  \scaleto{%
    \scalerel*[\widthof{\ensuremath{#1}}]{\kern.1pt\mathchar"0362\kern.1pt}%
    {\rule{0ex}{\textheight}}
  }{\textheight}%
}{2.4ex}}%
\stackon[-6.9pt]{#1}{\tmpbox}%
\endgroup
}

\newcommand\reallywidetilde[1]{%
\begingroup
\stackMath
\savestack{\tmpbox}{\stretchto{%
  \scaleto{%
    \scalerel*[\widthof{\ensuremath{#1}}]{\kern-.8pt\mathchar"307E\kern-.8pt}%
    {\rule{0ex}{\textheight}}
  }{\textheight}%
}{2ex}}%
\stackon[-6.9pt]{#1}{\tmpbox}%
\endgroup
}

\newtheorem{theorem}{Theorem}[section]
\newtheorem{theoremX}{Theorem}
\newtheorem*{theorem*}{Theorem}

\newcommand{\makealiasthm}[2]{%
  \newaliascnt{#1}{theorem}%
  \newtheorem{#1}[#1]{#2}%
  \aliascntresetthe{#1}%
}
\newcommand{\makealiasthmX}[2]{%
  \newaliascnt{#1}{theoremX}%
  \newtheorem{#1}[#1]{#2}%
  \aliascntresetthe{#1}%
}

\makealiasthmX{goalX}{Goal}
\makealiasthm{corollary}{Corollary}
\makealiasthmX{corollaryX}{Corollary}
\newtheorem*{corollary*}{Corollary}
\makealiasthm{proposition}{Proposition}

\newtheorem*{proposition*}{Proposition}
\makealiasthm{lemma}{Lemma}
\makealiasthm{want}{Want}
\makealiasthm{conjecture}{Conjecture}
\makealiasthmX{conjectureX}{Conjecture}
\newtheorem*{conjecture*}{Conjecture}
\makealiasthm{speculation}{Speculation}
\makealiasthmX{speculationX}{Speculation}
\makealiasthm{assumption}{Assumption}
\theoremstyle{definition}
\makealiasthm{definition}{Definition}

\newtheorem*{definition*}{Definition}
\makealiasthm{variant}{Variant}
\newtheorem*{variant*}{Variant}
\makealiasthm{example}{Example}

\newtheorem*{example*}{Example}
\makealiasthm{remark}{Remark}
\newtheorem*{remark*}{Remark}
\makealiasthm{motivation}{Motivation}
\makealiasthm{notation}{Notation}
\newtheorem*{notation*}{Notation}
\makealiasthm{construction}{Construction}
\makealiasthm{question}{Question}
\newtheorem*{question*}{Question}
\newtheorem*{acknowledgements}{Acknowledgements}

\crefname{theorem}{theorem}{theorems}
\Crefname{theorem}{Theorem}{Theorems}
\crefname{theoremX}{theorem}{theorems}
\Crefname{theoremX}{Theorem}{Theorems}
\crefname{goalX}{goal}{goals}
\Crefname{goalX}{Goal}{Goals}
\crefname{corollary}{corollary}{corollaries}
\Crefname{corollary}{Corollary}{Corollaries}
\crefname{corollaryX}{corollary}{corollaries}
\Crefname{corollaryX}{Corollary}{Corollaries}
\crefname{proposition}{proposition}{propositions}
\Crefname{proposition}{Proposition}{Propositions}
\crefname{propositionX}{proposition}{propositions}
\Crefname{propositionX}{Proposition}{Propositions}
\crefname{lemma}{lemma}{lemmas}
\Crefname{lemma}{Lemma}{Lemmas}
\crefname{want}{want}{wants}
\Crefname{want}{Want}{Wants}
\crefname{conjecture}{conjecture}{conjectures}
\Crefname{conjecture}{Conjecture}{Conjectures}
\crefname{conjectureX}{conjecture}{conjectures}
\Crefname{conjectureX}{Conjecture}{Conjectures}
\crefname{speculation}{speculation}{speculations}
\Crefname{speculation}{Speculation}{Speculations}
\crefname{speculationX}{speculation}{speculations}
\Crefname{speculationX}{Speculation}{Speculations}
\crefname{assumption}{assumption}{assumptions}
\Crefname{assumption}{Assumption}{Assumptions}
\crefname{definition}{definition}{definitions}
\Crefname{definition}{Definition}{Definitions}
\crefname{definitionX}{definition}{definitions}
\Crefname{definitionX}{Definition}{Definitions}
\crefname{variant}{variant}{variants}
\Crefname{variant}{Variant}{Variants}
\crefname{example}{example}{examples}
\Crefname{example}{Example}{Examples}
\crefname{exampleX}{example}{examples}
\Crefname{exampleX}{Example}{Examples}
\crefname{remark}{remark}{remarks}
\Crefname{remark}{Remark}{Remarks}
\crefname{motivation}{motivation}{motivations}
\Crefname{motivation}{Motivation}{Motivations}
\crefname{notation}{notation}{notations}
\Crefname{notation}{Notation}{Notations}
\crefname{construction}{construction}{constructions}
\Crefname{construction}{Construction}{Constructions}
\crefname{question}{question}{questions}
\Crefname{question}{Question}{Questions}

\usepackage[colorinlistoftodos]{todonotes}

\let\oldtocsection=\tocsection
\let\oldtocsubsection=\tocsubsection
\let\oldtocsubsubsection=\tocsubsubsection
\renewcommand{\tocsection}[2]{\hspace{0em}\oldtocsection{#1}{#2}}
\renewcommand{\tocsubsection}[2]{\hspace{1em}\oldtocsubsection{#1}{#2}}
\renewcommand{\tocsubsubsection}[2]{\hspace{2em}\oldtocsubsubsection{#1}{#2}}

\def\isotInt{\mathrm{isotInt}}
\def\uExact{\underline{\Exact}}

\def\IsotComp{\mathrm{IsotComp}}
\def\IsotInt{\mathrm{IsotInt}}

\def\all{\mathrm{all}}
\def\surj{\mathrm{surj}}

\def\pullback{\mathrm{pullback}}
\def\tF{\widetilde{F}}

\def\scB{\mathscr{B}}

\def\Tw{\mathrm{Tw}}

\def\Mod{\mathrm{Mod}}

\def\Pic{\mathrm{Pic}}

\def\red{\mathrm{red}}

\def\lfp{\mathrm{lfp}}

\def\ch{\mathrm{ch}}

\def\Hd{\mathrm{Hd}}
\def\triv{\mathrm{triv}}

\def\ev{\mathrm{ev}}

\def\Perf{\mathrm{Perf}}
\def\Aut{\mathrm{Aut}}

\def\fp{\mathrm{fp}}

\def\scD{\mathscr{D}}
\def\diag{\mathrm{diag}}

\def\pr{\mathrm{pr}}
\def\id{\mathrm{id}}

\def\GG{\mathbb{G}}

\def\cart{\ar@{}[rd]|{\Box}}

\def\Tot{\mathrm{Tot}}
\def\Cech{\mathrm{\check{C}ech}}

\def\pSymp{\mathrm{pSymp}}
\def\Symp{\mathrm{Symp}}
\def\Lag{\mathrm{Lag}}

\def\dual{^{\vee}}

\def\isot{\mathrm{isot}}

\def\Z{\mathbb{Z}}

\def\C{\mathbb{C}}

\def\O{\mathcal{O}} 

\def\LL{\mathbb{L}}
\def\TT{\mathbb{T}}
\def\T{\mathrm{T}}

\def\dSt{\mathrm{dSt}}
\def\dAff{\mathrm{dAff}}

\def\QCoh{\mathrm{QCoh}}

\def\Grpd{\mathrm{Grpd}}
\def\Cat{\mathrm{Cat}}
\def\Fun{\mathrm{Fun}}

\def\Map{\mathrm{Map}}
\def\uMap{\underline{\Map}}
\def\fib{\mathrm{fib}}
\def\cof{\mathrm{cof}}

\def\op{\mathrm{op}}
\def\fil{\mathrm{fil}}

\def\forget{\mathrm{forget}}

\def\DR{\mathrm{DR}}

\def\Gr{\mathrm{Gr}}
\def\Fil{\mathrm{Fil}}
\def\Sym{\mathrm{Sym}}

\def\textin{\quad\textup{in}\quad}
\def\and{\quad\textup{and}\quad}

\def\sA{\mathscr{A}}

\def\cl{\mathrm{cl}}
\def\ex{\mathrm{ex}}

\def\cE{\mathcal{E}}
\def\cF{\mathcal{F}}

\def\cL{\mathcal{L}}

\def\pt{\mathrm{pt}}

\def\act{\mathrm{act}}

\def\sm{\mathrm{sm}}

\def\SP{\ar@{}[rd]|{\mathrm{SP}}}

\def\commutes{\ar@{}[rd]|{\circlearrowleft}}

\def\et{\mathrm{et}}

\def\tf{\widetilde{f}}
\def\tg{\widetilde{g}}

\def\dASt{\mathrm{dASt}}

\def\Exact{\mathrm{Exact}}

\def\scC{\mathscr{C}}
\def\scE{\mathscr{E}}

\def\CAlg{\mathrm{CAlg}}

\def\inv{\mathrm{inv}}
\def\equiv{\mathrm{eq}}
\def\colim{\mathrm{colim}}
\def\Grp{\mathrm{Grp}}

\def\Hamil{\mathrm{Ham}}
\def\g{\mathfrak{g}}
\def\h{\mathfrak{h}}
\def\t{\mathfrak{t}}

\def\ind{\mathrm{ind}}

\def\res{\mathrm{res}}

\def\Mon{\mathrm{Mon}}
\def\St{\mathrm{St}}
\def\Sch{\mathrm{Sch}}

\def\grp{\mathrm{grp}}

\def\H{\mathrm{H}}

\def\Ext{\mathrm{Ext}}

\makeatletter
\newcommand*{\da@rightarrow}{\mathchar"0\hexnumber@\symAMSa 4B }
\newcommand*{\da@leftarrow}{\mathchar"0\hexnumber@\symAMSa 4C }
\newcommand*{\xdashrightarrow}[2][]{%
  \mathrel{%
    \mathpalette{\da@xarrow{#1}{#2}{}{\da@rightarrow\mkern-5mu}{\,}{}}{}%
  }%
}
\newcommand{\xdashleftarrow}[2][]{%
  \mathrel{%
    \mathpalette{\da@xarrow{#1}{#2}\da@leftarrow{}{}{\,}}{}%
  }%
}
\newcommand{\xdashleftrightarrow}[2][]{%
  \mathrel{%
    \mathpalette{\da@xarrow{#1}{#2}{\da@leftarrow}{\da@rightarrow\mkern-5mu}{}{}{}}{}%
  }%
}
\newcommand*{\da@xarrow}[7]{%
  \sbox0{$\ifx#7\scriptstyle\scriptscriptstyle\else\scriptstyle\fi#5#1#6\m@th$}%
  \sbox2{$\ifx#7\scriptstyle\scriptscriptstyle\else\scriptstyle\fi#5#2#6\m@th$}%
  \sbox4{$#7\dabar@\m@th$}%
  \dimen@=\wd0 %
  \ifdim\wd2 >\dimen@
    \dimen@=\wd2 %
  \fi
  \count@=2 %
  \def\da@bars{\dabar@\dabar@}%
  \@whiledim\count@\wd4<\dimen@\do{%
    \advance\count@\@ne
    \expandafter\def\expandafter\da@bars\expandafter{%
      \da@bars
      \dabar@ 
    }%
  }%
  \mathrel{#3}%
  \mathrel{%
    \mathop{\da@bars}\limits
    \ifx\\#1\\%
    \else
      _{\copy0}%
    \fi
    \ifx\\#2\\%
    \else
      ^{\copy2}%
    \fi
  }%
  \mathrel{#4}%
}
\makeatother

\usetikzlibrary{decorations.markings} 
\tikzset{double line with arrow/.style args={#1,#2}{decorate,decoration={markings,
mark=at position 0 with {\coordinate (ta-base-1) at (0,1pt);
\coordinate (ta-base-2) at (0,-1pt);},
mark=at position 1 with {\draw[#1] (ta-base-1) -- (0,1pt);
\draw[#2] (ta-base-2) -- (0,-1pt);
}}}}
\tikzset{Equal/.style={-,double line with arrow={-,-}}}

\makeatletter
\newcommand{\namedlabel}[1]{%
  \begingroup
    \refstepcounter{equation}
    \leavevmode
    \hbox to 0pt{\hskip-2\@totalleftmargin \phantomsection\label{#1}\hss}%
  \endgroup
  (\theequation)
}
\makeatother

\def\lra{\longrightarrow}

\def\LagComp{\mathrm{LagComp}}

\def\Corr{\mathrm{Corr}}

\def\Isot{\mathrm{Isot}}

\def\SYMP{\mathrm{SYMP}}
\def\sSymp{\mathscr{S}\mathrm{ymp}}
\def\sHamil{\mathscr{H}\mathrm{am}}
\def\swap{\mathrm{sw}}
\def\LagInt{\mathrm{LagInt}}
\def\TrivLag{\mathrm{TrivLag}}

\def\rig{\mathrm{rigid}}

\def\uPerf{\underline{\mathrm{Perf}}}
\def\uPic{\underline{\mathrm{Pic}}}
\def\uMap{\underline{\mathrm{Map}}}

\def\P{\mathbb{P}}

\def\rmR{\mathrm{R}}

\def\Alg{\mathrm{Alg}}

\def\Trip{\mathrm{Trip}}

\begin{document}

\title{Shifted symplectic rigidification}

\author{Hyeonjun Park}
\author{Jemin You}

\date{\today}

\begin{abstract}
We construct shifted symplectic derived enhancements on rigidified moduli spaces of sheaves on Calabi-Yau varieties of dimension at least two.
More generally, we prove that any $B\mathbb{G}_m$-action on a non-positively-shifted symplectic derived Artin stack is Hamiltonian.
We provide a symplectic rigidification functor as the left adjoint to the trivial action functor in symplectic categories with Lagrangian correspondences.
We also descend the Lagrangian correspondence of short exact sequences of sheaves to rigidified moduli spaces.
\end{abstract}

\maketitle
\tableofcontents
\addtocontents{toc}{\protect\setcounter{tocdepth}{1}}

\section*{Introduction}

Shifted symplectic structures \cite{PTVV} are generalizations of symplectic structures for smooth varieties to derived Artin stacks.
A fundamental example of a shifted symplectic stack is a derived moduli stack of sheaves on a Calabi-Yau variety.
The existence of shifted symplectic structures has many important applications to enumerative geometry,
e.g., cohomological Donaldson-Thomas theory of Calabi-Yau $3$-folds \cite{BBDJS,KL}; Donaldson-Thomas theory of Calabi-Yau $4$-folds \cite{BJ,OT}; Gromov-Witten theory of gauged linear sigma models \cite{CZ,KP2}.

In practice, we are often interested in the {\em rigidified} moduli spaces.
Indeed, moduli stacks of sheaves on algebraic varieties are always Artin stacks due to the constant automorphisms $\GG_m$.
In many perspectives, it is more desirable to have moduli spaces as schemes, or at least generically schemes.
Thus we often remove the constant automorphisms;
this process is called {\em rigidification}.

Unfortunately, the rigidification process does not preserve shifted symplectic structures.
Indeed, we can easily observe that the self-duality of cotangent complexes breaks after rigidification.
Hence the usual derived structures on rigidified moduli spaces are not shifted symplectic.

The purpose of this paper is to construct shifted symplectic derived enhancements on rigidified moduli spaces.
We also study functorial behaviors of this symplectic rigidification. 

\subsection*{Symplectic rigidification}

Let $M$ be a moduli stack of sheaves on a Calabi-Yau $n$-fold (or more generally, moduli of objects in an $n$-Calabi-Yau category).
Then $M$ has a $(2-n)$-shifted symplectic structure by \cite{PTVV,BD}.
In particular, we have a self-duality
\begin{equation}\label{Eq:1}
\TT_M \simeq \LL_M[2-n],
\end{equation}
of the cotangent complex $\LL_M$, where $\TT_M:=\LL_M\dual$ is the tangent complex.

For any point $m \in M$, we have the subgroup $\GG_m \subseteq \Aut_M(m)$ consisting of constant automorphisms.
Equivalently, they form an action of the group stack $B\GG_m$ on $M$.
The {\em rigidification} of $M$ can be defined as the quotient stack
\[M/B\GG_m.\]
The duality \eqref{Eq:1} breaks for the rigidification since we have
$\TT_{M/B\GG_m}|_m \simeq \cof(\C[1] \to \TT_{M,m}).$

This is analogous to a well-known phenomenon in classical symplectic geometry: 
the quotient of a symplectic manifold by an action of a group is mostly not symplectic.
Instead, what we usually do in symplectic geometry is to take the {\em Hamiltonian reduction}.

A shifted symplectic version of Hamiltonian reduction is as follows.
An action of a group stack $G$ on a $d$-shifted symplectic stack $M$ is {\em Hamiltonian} if there exists a $(d+1)$-shifted Lagrangian
\[\mu : M/G \longrightarrow 
\T^*[d+1]BG\]
satisfying a natural compatibility condition (see \cite{Cal,Saf1,AC} or \S\ref{ss:Hamiltonian}).
Here the $(d+1)$-shifted cotangent bundle $\T^*[d+1]BG$ is $(d+1)$-shifted symplectic by \cite{Cal2}.
Then the {\em Hamiltonian reduction} $M/\!/G$ is defined as the Lagrangian intersection
\[\xymatrix{
M/\!/G \ar[r] \ar[d] \cart & BG \ar[d]^0 \\
M/G \ar[r]^-{\mu} & \T^*[d+1]BG,
}\]
which is $d$-shifted symplectic by \cite[Thm.~2.9]{PTVV}.

Thus we want to construct a symplectic version of the rigidification as the Hamiltonian reduction of a $B\GG_m$-action.
Our main result ensures that this is always possible when the shift is non-positive.

\begin{theoremX}[\Cref{Prop:BT-symp}, \Cref{Thm:Ham=Symp}]\label{Thm:Intro-Rig}
Let $M$ be a $d$-shifted symplectic derived Artin stack.
If $d \leq 0$,
then any $B\GG_m$-action on $M$ is Hamiltonian.
\end{theoremX}

Consequently, we can form the {\em symplectic rigidification}  as the Lagrangian intersection
\begin{equation} \label{Eq:81}
  \xymatrix{
M/\!/B\GG_m 
\ar[r] \ar[d] \cart & B^2\GG_m \ar[d]^0 \\
M/B\GG_m \ar[r]^-{\mu} & \T^*[d+1]B^2\GG_m,
}
\end{equation}
where $B^2\GG_m\coloneqq B(B\GG_m)$ is the classifying stack of the group stack $B\GG_m$.
Then the symplectic rigidification $M/\!/B\GG_m$ is $d$-shifted symplectic and has the same classical truncation with the ordinary rigidification $M/B\GG_m$.

We can apply Theorem \ref{Thm:Intro-Rig} to moduli of sheaves on Calabi-Yau varieties.

\begin{corollaryX} 
\label{Cor:Intro-1}
Let $X$ be a smooth projective $n$-fold with $K_X\simeq \O_X$. 
If $n \geq 2$,
the rigidified classical moduli space $M(X)_\cl$ of stable sheaves on $X$ has a $(2-n)$-shifted symplectic derived enhancement.
\end{corollaryX}

For rank $1$ sheaves (e.g. ideal sheaves), there is a shortcut to obtain \Cref{Cor:Intro-1} by considering the fiber over the determinant map in \cite{STV}.
However, for rank $0$ sheaves (i.e. torsion sheaves), \Cref{Cor:Intro-1} is already a new result.
Especially, this result was desired in Donaldson-Thomas theory of Calabi-Yau $4$-folds (e.g. \cite[Claim~3.29]{BJ}, \cite[\S4.4(e)]{GJT}).

Interestingly, Theorem \ref{Thm:Intro-Rig} does not hold for {\em families} in general.
More precisely, for a family of $K3$ surfaces, the $B\GG_m$-action on the family of moduli of sheaves is Hamiltonian if and only if the base is in the {\em Noether-Lefschetz locus} (\Cref{Cor:k3family}).
For instance, the $B\GG_m$-action for the family of quartic surfaces is not Hamiltonian. 

Moreover, the assumption $d\leq 0$ in Theorem \ref{Thm:Intro-Rig} is {\em sharp}.
For an elliptic curve, the $B\GG_m$-action on the moduli of sheaves is not Hamiltonian (\Cref{Cor:non-Ham-elliticcurve}).

The key observation for proving \Cref{Thm:Intro-Rig} is that an action of a group stack on a shifted symplectic stack is Hamiltonian if and only if (1) the symplectic form is invariant, and (2) the associated de Rham form is equivariant (\Cref{Prop:HamilviaEquiv}).
When the group stack is $B\GG_m$, 
\begin{enumerate}
\item the symplectic form is always invariant (\Cref{Prop:BT-symp}), and
\item the associated de Rham form is equivariant if and only if the contraction of the symplectic form with the vector field given by the group action is exact (\Cref{Thm:MomentMapEquation}).
\end{enumerate}
When $d\leq 0$, the contraction $1$-form is a negatively-shifted closed $1$-form and hence exact.

\subsection*{Functoriality}

Let us recall the universal property of the ordinary quotients by group actions.
For a group stack $G$, denote by $\dASt^G$ the $\infty$-category of derived Artin stacks with $G$-actions.
The functor of trivial $G$-actions, $\triv\colon\dASt \to \dASt^G$, has a left adjoint
\[\mathrm{quot} \colon \dASt^G \lra \dASt,\]
which sends $M \in \dASt^G$ to the quotient stack $M/G$.

The symplectic rigidifications have an analogous universal property if we consider the Lagrangian correspondences as morphisms.
Indeed, consider the symplectic category $\Symp_d$ in \cite{Cal,Hau} consisting of $d$-shifted symplectic stacks and their Lagrangian correspondences.
As a variant, we can form a {\em $B\GG_m$-equivariant symplectic category} $\Symp^{B\GG_m}_d$ (\Cref{Def:G-SympCat}) consisting of:
\begin{enumerate}
\item [(Obj)] $d$-shifted symplectic derived Artin stacks with (symplectic) $B\GG_m$-actions; 
\item [(Mor)] $d$-shifted $B\GG_m$-equivariant Lagrangian correspondences.
\end{enumerate}
Consider the functor of trivial $B\GG_m$-actions, $\triv_{\Symp} \colon \Symp_d \to \Symp^{B\GG_m}_d$.

\begin{theoremX}[{\Cref{Cor:rigid-adjunction}}]\label{Intro:Thm-Funct}
Let $d\leq 0$.
Then the above functor $\triv_{\Symp}$ has a left adjoint
\[\rig_{\Symp} : \Symp^{B\GG_m}_{d} \lra \Symp_{d},\]
which sends an object $M \in \Symp^{B\GG_m}_d$ to the symplectic rigidification $M/\!/B\GG_m \in \Symp_{d}$ in \eqref{Eq:81}.
\end{theoremX}

Consequently, for a $d$-shifted $B\GG_m$-equivariant Lagrangian correspondence $M \leftarrow L \rightarrow N$, when $d\leq 0$, we have an induced rigidified Lagrangian correspondence
\[
\xymatrix@-1pc{
& L/\!/B\GG_m \ar[ld] \ar[rd] & \\
M/\!/B\GG_m  && N/\!/B\GG_m  .
}\]

The strategy of proving \Cref{Intro:Thm-Funct} is as follows.
For a general group stack $G$, there is a natural way to form a {\em Hamiltonian category} $\Hamil^G_d$ consisting of $d$-shifted symplectic stacks with Hamiltonian $G$-actions (\Cref{Def:HamCat}).
Moreover, there is a general adjunction
\[\xymatrix{
\red_{\Hamil} \colon 
\Hamil^G_{d}  \ar@<.4ex>[r] & \ar@<.4ex>[l] \Symp_{d} } \colon \triv_{\Hamil},\]
between the Hamiltonian reduction functor $\red_{\Hamil}$ and the trivial Hamiltonian action functor $\triv_{\Hamil}$ (see \S\ref{ss:Hamiltonian} for details).
When $G=B\GG_m$, we can show that the forgetful functor
\[\forget_{\Hamil} \colon \Hamil_d^{B\GG_m} \xrightarrow{\simeq} \Symp^{B\GG_m}_d\]
is an equivalence of $\infty$-categories (\Cref{Thm:Ham=Symp}), which implies \Cref{Intro:Thm-Funct}.

As an example, we provide a rigidified version of the extension Lagrangian correspondence \cite{BD}.
For a smooth projective Calabi-Yau $n$-fold $X$, we have a $(2-n)$-shifted Lagrangian correspondence
\[\xymatrix@-1pc{
& \uExact(X) \ar[ld]_{(\ev_1,\ev_3)} \ar[rd]^{\ev_2} \\
\uPerf(X)\times \uPerf(X) && \uPerf(X)
} \]   
by \cite[Cor.~6.5]{BD},
where $\uPerf(X)$ is the derived moduli stack of perfect complexes on $X$ 
and $\uExact(X)$ is the derived moduli stack of exact triangles of perfect complexes on $X$.

\begin{corollaryX}[\Cref{Prop:RidigifiedExtCorr}] \label{Intro-Cor:RidExtCorr}
Let $X$ be a smooth projective Calabi-Yau $n$-fold with $n \geq 2$.
Then there exists a canonical $(2-n)$-shifted Lagrangian correspondence
\[
\xymatrix@-1pc{
&\underline{\Exact}(X)^{\rig} \ar[ld] \ar[rd] & \\
\uPerf(X)/\!/B\GG_m \times \uPerf(X)/\!/B\GG_m && \uPerf(X)/\!/B\GG_m 
}\]
for some $\underline{\Exact}(X)^{\rig}$, which has the same classical truncation with $\uExact(X)/\!/B\GG_m$.
\end{corollaryX}

The Lagrangian correspondence in \Cref{Intro-Cor:RidExtCorr} is obtained as the composition of the rigidified Lagrangian correspondence $\uExact(X)/\!/B\GG_m$ and the canonical Lagrangian correspondence 
\[\Big(\uPerf(X)/\!/B\GG_m \times \uPerf(X)/\!/B\GG_m\Big) \dashrightarrow \Big(\uPerf(X) \times \uPerf(X)\Big)/\!/B\GG_m \]
induced by the functoriality of Hamiltonian reductions along change of groups (\Cref{Prop:Ind/Res-Functoriality}).

\subsection*{Sequel}

In a follow-up paper \cite{You}, the second-named author will apply our results on symplectic rigidification to enumerative geometry of Calabi-Yau $4$-folds \cite{BJ,OT}.
The virtual pullback formula for stable pairs and stable sheaves in \cite[Thm.~0.5]{Park1} will be lifted to a $(-2)$-shifted Lagrangian correspondence.
Since this derived geometric lift is naturally compatible with the {\em reduction} operation in \cite{KP,BKP},
we can obtain a virtual pullback formula between the reduced DT4 virtual classes for hyper-k{\"a}hler $4$-folds, studied in \cite{COT1,COT2}.
This will give an affirmative answer to the question raised in \cite[Conj.~A.1]{COT2}.

\subsection*{Related work}

Theorem \ref{Intro:Thm-Funct} is motivated by the symplectic pushforwards, introduced in \cite{Park2}.
However, our symplectic rigidification functor is not a special case of the symplectic pushforward functors.
The symplectic pushforwards exist for the {\em exact} version of symplectic categories, while our symplectic rigidification functor exists for the ordinary symplectic categories, where the symplectic forms are closed but not necessarily exact.
Especially, for $0$-shifted symplectic stacks, the symplectic forms are usually not exact (e.g. moduli of stable sheaves on $K3$-surfaces with coprime Chern character), so the symplectic pushforward does not give us a symplectic rigidification.

\begin{acknowledgements}
We thank Tasuki Kinjo for the suggestion to include $0$-shifted symplectic stacks in \Cref{Thm:Intro-Rig}.
We thank Arkadij Bojko and Yalong Cao for discussions on applications to Donaldson-Thomas theory of Calabi-Yau $4$-folds.
J.Y. would like to thank his advisor Young-Hoon Kiem for his generous support and valuable comments during this project.

H.P. was supported by Korea Institute for Advanced Study (SG089201).
J.Y. was supported by Korea Institute for Advanced Study (MG108701).
\end{acknowledgements}

\subsection*{Notation and conventions}

We use the language of $\infty$-categories in \cite{LurHTT, LurHA}.
Throughout, all derived stacks are assumed to be over the field $\C$ of complex numbers.
We denote by
\begin{itemize}
\item $\Cat$ the $\infty$-category of $\infty$-categories,
\item $\Grpd \subseteq \Cat$ the full subcategory of $\infty$-groupoids, and
\item $\dSt$ 
the $\infty$-category of derived stacks.
\end{itemize}
We say that a derived stack is geometric if it is $n$-geometric in the sense of \cite[Def.~1.3.3.1]{TV1}.
We say that a derived stack is locally geometric if it is the union of open substacks that are geometric.

\section{Symplectic categories}  \label{Sec:SympCat}

In this section, we review the theory of {\em shifted symplectic structures} in \cite{PTVV}, using the symplectic categories in \cite{Cal,Hau}.

\subsection{Differential forms}

We first fix basic notations and terminologies for differential forms on derived stacks. 
All the spaces and maps that we use for differential forms are encoded in the de Rham complexes with their Hodge filtrations.

Let $M$ be a derived stack.
Denote by $\QCoh(M)$ the symmetric monoidal $\infty$-category of quasi-coherent sheaves on $M$.
The $\infty$-category of \defterm{filtered complexes} on $M$ is the functor category
\[\Fil\big(\QCoh(M)\big) \coloneqq \Fun((\Z,\geq), \QCoh(M)),\]
where the ordered set $(\Z,\geq)$ is identified to the associated $\infty$-category.
The $\infty$-category $\Fil\big(\QCoh(M)\big)$ has a canonical symmetric monoidal structure given by the Day convolution (see \cite[\S2.2.6]{LurHA}).
We will use the following notations for complexes induced by a filtered complex $E$ on $M$.
\begin{itemize}
\item  $\Fil^pE\in \QCoh(M)$ is the image of $p\in \Z$ under the functor $E\colon (\Z,\geq) \to \QCoh(M)$,
\item $\Fil^{p/q}E \in \QCoh(M)$ is the cofiber of the induced map $\Fil^qE \to \Fil^pE$ for any integers $q \geq p$,
\item $\Gr^p E\in \QCoh(M)$ is the  cofiber of the induced map $\Fil^{p+1}E \to \Fil^pE$ for an integer $p$,
\item $\Fil^{-\infty}E \in \QCoh(M)$ is the colimit of the functor $E\colon (\Z,\geq) \to \QCoh(M)$.
\end{itemize}

The $\infty$-category of \defterm{complete} filtered complexes on $M$ is the full subcategory
\[\widehat{\Fil}\big(\QCoh(M)\big) \subseteq \Fil\big(\QCoh(M)\big)\]
consisting of filtered complexes $E$ such that $\Fil^{\infty}E \coloneqq \varprojlim_{p \to \infty}\Fil^pE\simeq 0$.
Since $\widehat{\Fil}\big(\QCoh(M)\big)$ is a reflective subcategory, it has an induced symmetric monoidal structure (see \cite[Thm.~2.25]{GP}).

Let $ g\colon M \to B$ be a morphism of derived stacks.
We have the {\em de Rham complex}
\[\DR(M \to B) \in \CAlg\left(\widehat{\Fil}\big(\QCoh(B)\big)\right)\footnote{The standard notation is $\DR(M/B)$. We use $\DR(M \to B)$ to avoid confusion with quotient-stack notations.}\]
as a commutative algebra object in the symmetric monoidal $\infty$-category $\widehat{\Fil}\big(\QCoh(B)\big)$.
We refer to \cite{CPTVV} for the construction, under identifying complete filtered complexes with graded mixed complexes (see also \cite[\S1.1]{PY}).
We note that our de Rham complex is the completion of the non-completed de Rham complex studied in \cite[\S1]{Park2}.

When $g:M \to B$ is locally geometric, it admits a cotangent complex $\LL_{M/B} \in \QCoh(M)$, and 
\begin{equation}\label{Eq:GrDR}
\Gr^p\DR(M\to B) \simeq g_*\left(\Lambda^p \LL_{M/B}[-p]\right) \textin \QCoh(B).
\end{equation}
See \cite[Thm.~2.6]{CalSaf} for a proof of this equivalence.

\begin{notation}
Let $M \to B$ be a morphism of derived stacks.
\begin{itemize}
\item The space of $d$-shifted \defterm{closed $p$-forms} on $M \to B$ is the mapping space
\[\sA^{p,\cl}(M \to B,d)\coloneqq \Map_{\QCoh(B)}(\O_B, \Fil^p\DR(M \to B)[p+d]).\]
\item The space of $d$-shifted \defterm{(ordinary) $p$-forms} on $M \to B$ is the mapping space
\[\sA^{p}(M \to B,d)\coloneqq \Map_{\QCoh(B)}(\O_B, \Gr^p\DR(M \to B)[p+d]).\]
\item The space of $d$-shifted \defterm{exact $p$-forms} on $M \to B$ is the mapping space
\[\sA^{p,\ex}(M \to B,d)\coloneqq \Map_{\QCoh(B)}(\O_B, \Fil^{0/p}\DR(M \to B)[p+d-1]).\]
\item The space of $d$-shifted \defterm{de Rham forms} on $M \to B$ is the mapping space
\[\sA^{\DR}(M \to B,d)\coloneqq \Map_{\QCoh(B)}(\O_B, \Fil^0\DR(M \to B)[d]).\]
\end{itemize}
\end{notation}

As a variant, we also consider the space of \defterm{$\cF$-twisted closed $p$-forms} for $\cF \in \QCoh(B)$, 
\[\sA^{p,\cl}(M \to B,\cF)\coloneqq \Map_{\QCoh(B)}(\O_B, \Fil^p\DR(M \to B)\otimes \cF[p]).\]
The spaces of twisted ordinary, exact, de Rham forms are defined analogously.

\subsection{Symplectic categories}
We then review the {\em categories} of shifted symplectic stacks.
The morphisms in these categories are Lagrangian correspondences between shifted symplectic stacks.

Let $B$ be a derived stack.
By \cite[Def.~3.6]{Bar}, we have the {\em correspondence category}
\[\Corr_B\coloneq \Corr(\dSt_B),\]
where $\dSt_{B}\coloneq \dSt_{/B}$ is the overcategory.
The $\infty$-category $\Corr_B$ has the following description:
\begin{itemize}
\item An object in $\Corr_B$ is a derived stack $M$ over $B$.
\item A morphism from $M$ to $N$ in $\Corr_B$ is a diagram
\[{\xymatrix@-1pc{
& L\ar[ld] \ar[rd] & \\ M && N
}}\]
of derived stacks over $B$.
\end{itemize}
We note that the assignment $B \mapsto \Corr_B$ upgrades to an $\infty$-functor
\[\Corr_{(-)}\colon\Corr(\dSt) \lra \Cat.\]
We refer to \S\ref{Appendix:Corr} for a review of the theory of correspondence categories.
See \Cref{Const:Corr-Funct} for the precise construction of the above functor $\Corr_{(-)}$.

Let $\sA^{2,\cl}_B[d] \in \dSt_B$ be the {\em derived stack} of $d$-shifted closed $2$-forms over $B$.
This derived stack is uniquely characterized by the following universal property:
for any derived stack $M$ over $B$,
\[\Map_{\dSt_B}(M,\sA^{2,\cl}_B[d]) \simeq \sA^{2,\cl}(M\to B,d).\]
The $d$-shifted \defterm{presymplectic category} over $B$ is the correspondence category
\[\pSymp_{B,d}\coloneqq \Corr_{\sA^{2,\cl}_B[d]}.\] 
\begin{itemize}
\item An object $\pSymp_{B,d}$ is a derived stack $M$ over $B$ together with $\theta_M \in \sA^{2,\cl}(M \to B,d)$.
\item A morphism from $(M,\theta_M)$ to $(N,\theta_N)$ in $\pSymp_{B,d}$ is a correspondence $M \xleftarrow{} L \xrightarrow{} N$ of derived stacks over $B$ together with equivalences $\gamma_L\colon\theta_M|_L \simeq \theta_N|_L$ in $\sA^{2,\cl}(L \to B,d)$.
\end{itemize}
A $d$-shifted \defterm{symplectic stack} $M$ is an object in $\pSymp_{B,d}$ such that
\begin{enumerate}
\item the projection map $M \to B$ is locally geometric and locally of finite presentation,
\item the image of $\theta_M$ under the map $\sA^{2,\cl}(M \to B,d) \to \sA^{2}(M \to B,d)$ induces an equivalence \[\TT_{M/B} \xrightarrow{\simeq} \LL_{M/B}[d].\]
\end{enumerate}
A $d$-shifted \defterm{Lagrangian correspondence} $L \colon M \dashrightarrow N$ is a morphism in $\pSymp_{B,d}$ such that
\begin{enumerate}
\item the projection map $L \to B$ is locally geometric and locally of finite presentation,
\item the following commutative square induced by $\gamma_L$ is a cartesian square of perfect complexes,
\[\xymatrix@-1pc{
\TT_{L/B} \ar[r] \ar[d] \cart & \TT_{N/B}|_L \simeq \LL_{N/B}|_L[d] \ar[d] \\ \TT_{M/B}|_L \simeq \LL_{M/B}|_L[d] \ar[r] & \LL_{L/B}[d].
}\]
\end{enumerate}

\begin{definition}
The $d$-shifted \defterm{symplectic category} over $B$ is the subcategory
\[\Symp_{B,d} \subseteq \pSymp_{B,d}\]
consisting of $d$-shifted symplectic stacks over $B$ and $d$-shifted Lagrangian correspondences over $B$.
\end{definition}

It is straightforward to check that the compositions of Lagrangian correspondences in $\pSymp_{B,d}$ are again Lagrangian correspondences, see \cite[Thm.~4.4]{Cal}, \cite[Thm.~1.2]{Saf1} or \cite[Prop.~14.12]{Hau}.
Since a subcategory of an $\infty$-category is determined by the underlying homotopy $1$-category, the above description gives us a well-defined $\infty$-category $\Symp_{B,d}$.

Let $p\colon U \to B$ be a morphism of derived stacks.
Then we have the \defterm{pullback functor}
\begin{equation}\label{Eq:pullback-SympCat}
\mathrm{pullback}_{U/B} \colon \Symp_{B,d} \lra \Symp_{U,d}.
\end{equation}
Indeed, we have a canonical morphism
\begin{equation}\label{Eq:pullback-SympCat-1}
{\xymatrix@-1pc{
& \sA^{2,\cl}_B[d]\times_B U\ar[ld]_{\pr_1} \ar[rd] & \\ \sA^{2,\cl}_B[d] && \sA^{2,\cl}_U[d]
}}\end{equation}
in $\Corr(\dSt)$. See \cite[Rem.~B.12.6]{CHS} for the precise construction of $\sA^{2,\cl}_B[d]\times_B U \to \sA^{2,\cl}_U[d]$.
Applying the functor $\Corr_{(-)} \colon \Corr(\dSt) \to \Cat$ in \Cref{Const:Corr-Funct}, we obtain a functor
\[\pSymp_{B,d} \lra \pSymp_{U,d}.\]
It is also straightforward to check that this presymplectic pullback functor preserves shifted symplectic stacks and Lagrangian correspondences. Equivalently, we have the pullback functor \eqref{Eq:pullback-SympCat}.

\subsection{Lagrangian categories}

We then consider the categories of Lagrangians, where the morphisms are $2$-fold Lagrangian correspondences.

Let $B$ be a derived stack, $d$ be an integer,
and $M$ be a $d$-shifted symplectic stack over $B$.
Denote by $M^{\isot} \in \dSt_B$ the derived stack of isotropic morphisms to $M$,
which is the fiber product
\begin{equation*} 
\xymatrix@-.5pc{
M^{\isot} \ar[r] \ar[d] \cart & M \ar[d]^-{\theta_M} \\
B \ar[r]^-{0} & \sA^{2,\cl}_B[d].
}\end{equation*}
Note that there exists a canonical equivalence of derived stacks,
\begin{equation}\label{Eq:M^isotxM^isot}
M^{\isot}\times_M M^{\isot} \simeq M^{\isot} \times_B \sA^{2,\cl}_B[d-1].
\end{equation}
The $d$-shifted \defterm{isotropic category} of $M$ is the correspondence category
\[\Isot_{M,d} \coloneqq\Corr_{M^{\isot}}.\]
\begin{itemize}
\item An object $L$ in $\Isot_{M,d}$ is equivalent to a morphism $B \dashrightarrow M$ in $\pSymp_{B,d}$.
\item A morphism $C\colon L_1 \dashrightarrow L_2$ in $\Isot_{M,d}$ is equivalent to a morphism $B \dashrightarrow {L_1}\times_M L_2$ in $\pSymp_{B,d-1}$,
where ${L_1}\times_M L_2 \in \pSymp_{B,d-1}$ is given by the image of $L_1\times_M L_2$ under \eqref{Eq:M^isotxM^isot}.
\end{itemize}
A $d$-shifted \defterm{Lagrangian} $L$ on $M$ is an object in $\Isot_{M,d}$ such that 
\begin{enumerate}
\item []
the induced morphism $B \dashrightarrow M$ in $\pSymp_{B,d}$ is a $d$-shifted Lagrangian correspondence.
\end{enumerate}
A $d$-shifted \defterm{$2$-fold Lagrangian correspondence} $L_1 \dashrightarrow L_2$ is a morphism in $\Isot_{M,d}$ such that 
\begin{enumerate}
\item [] the induced morphism $B\dashrightarrow L_1\times_M L_2$ in $\pSymp_{B,d-1}$ is a $(d-1)$-shifted Lagrangian correspondence.
\end{enumerate}

\begin{definition}
The $d$-shifted \defterm{Lagrangian category} of $M$ is the subcategory
\[\Lag_{M,d} \subseteq \Isot_{M,d} \]
consisting of $d$-shifted Lagrangians on $M$ and their $2$-fold Lagrangian correspondences.
\end{definition}

The $\infty$-category $\Lag_{M,d}$ is well-defined 
since the compositions of $2$-fold Lagrangian correspondences are again $2$-fold Lagrangian correspondences,
by 
\cite[Thm.~3.1]{Ben} (see \cite[\S4.2]{AB} for details).


\begin{remark}
An $(\infty,2)$-categorical enhancement $\mathrm{SYMP}_{B,d}$ of the symplectic category $\Symp_{B,d}$ is constructed in \cite{CHS}.
The Lagrangian category $\Lag_{M,d}$ can be described as the mapping category
\[\Lag_{M,d} \simeq \Map_{\mathrm{SYMP}_{B,d}} (M,B).\]
However, we will only use the underlying $(\infty,1)$-category in this paper.
\end{remark}

Let $L \colon M \dashrightarrow N$ be a $d$-shifted Lagrangian correspondence over $B$.
We have the \defterm{Lagrangian composition functor}
\begin{equation}\label{Eq:LagComp}
\LagComp_L\colon \Lag_{M,d} \lra \Lag_{N,d}.
\end{equation}
Indeed, we have a canonical morphism
\begin{equation}\label{Eq:S1-1}
{\xymatrix@-1pc{
& L^{\isot}\ar[ld] \ar[rd] & \\ M^{\isot} && N^{\isot}
}}\end{equation}
in $\Corr(\dSt)$, where $L^{\isot} \coloneq \fib(L \to \sA^{2,\cl}_B[d])$.
Applying the functor $\Corr_{(-)} \colon \Corr(\dSt) \to \Cat$ in \Cref{Const:Corr-Funct}, we obtain an  $\infty$-functor
\[\IsotComp_L \colon \Isot_{M,d} \lra \Isot_{N,d}.\]
This functor $\IsotComp_L$ preserves Lagrangians and $2$-fold Lagrangian correspondences (see e.g. \cite[\S4.2]{AB} or \cite{CHS}) and hence we have a well-defined functor $\LagComp_L$

These Lagrangian composition functors are functorial along compositions of Lagrangian correspondences in the following sense.

\begin{proposition}\label{Prop:LagComp-Functoriality}
There exists a canonical $\infty$-functor
\[\Lag_{(-),d} \colon \Symp_{B,d} \lra \Cat\]
which sends 
\begin{itemize}
\item 
a $d$-shifted symplectic stack $M$ to the Lagrangian category $\Lag_{M,d}$ and
\item a Lagrangian correspondence $L \colon M \dashrightarrow N$ to the Lagrangian composition functor $\LagComp_L$.
\end{itemize}
\end{proposition}

\Cref{Prop:LagComp-Functoriality} is an immediate consequence of having an $(\infty,2)$-category $\SYMP_{B,d}$.
Here we provide a direct $(\infty,1)$-categorical proof.

\begin{proof}[Proof of \Cref{Prop:LagComp-Functoriality}]
The morphism $\sA^{2,\cl}_B[d] \xleftarrow{0} B \xrightarrow{\id}B$ in $\Corr(\dSt)$ induces a functor
\[(-)^{\isot} \coloneq \Corr\left((-)\times_{\sA^{2,\cl}_B[d],0}B\right)\colon  \pSymp_{B,d} \coloneq \Corr_{\sA^{2,\cl}_B[d]} \lra \Corr_B 
\]
by \Cref{Const:Corr-Funct} (for $\scC=\dSt$).
We also have a functor
\[\Corr\left((\dSt_B)_{/(-)}\right) \colon \Corr_B\coloneq \Corr(\dSt_B) \lra \Cat 
\]
by \Cref{Const:Corr-Funct} (for $\scC=\dSt_B$).
Consider the composition
\[\Isot_{(-),d}\colon\Symp_{B,d} \hookrightarrow \pSymp_{B,d} \xrightarrow{(-)^{\isot}} \Corr_B \xrightarrow{\Corr\left((\dSt_B)_{/(-)}\right)}\Cat.\]
\begin{itemize}
\item On objects, the functor $\Isot_{(-),d}$ sends a $d$-shifted symplectic stack $M$ to the isotropic category $\Isot_{M,d}\coloneq \Corr_{M^{\isot}}$.
\item On morphisms, the functor $\Isot_{(-),d}$ sends a Lagrangian correspondence $L \colon M \dashrightarrow N$ to the functor $\IsotComp_L$.
\end{itemize}
Since $\IsotComp_L$ sends $\Lag_{M,d}$ to $\Lag_{N,d}$, the functor $\Isot_{(-),d}$ induces the desired functor $\Lag_{(-),d}$. 
(More precisely, the functor $\Lag_{(-),d}$ is the straightening of the subcategory of the unstraightening of the functor $\Isot_{(-),d}$ which is fiberwise $\Lag_{M,d}$.).
\end{proof}

\subsection{Adjunction}

We now provide a canonical adjunction between the symplectic category and the Lagrangian category, associated to a Lagrangian.
This is a general form of the adjunction for the Hamiltonian reduction functor, that we will use later.

Let $B$ be a derived stack and $d$ be an integer.

\begin{proposition}\label{Prop:LagCat-Adjunction}
Let $L \colon M\dashrightarrow N$ be a $d$-shifted Lagrangian correpondence over $B$.
Then there exists a canonical adjunction
\[\xymatrix{
\LagComp_L \colon 
\Lag_{M,d}  \ar@<.4ex>[r] & \ar@<.4ex>[l] \Lag_{N,d} } \colon \LagComp_{L^{\swap}}
\]
where $L^{\swap} \colon N \dashrightarrow M$ is the Lagrangian correspondence given by swapping the role of $M$ and $N$.
\end{proposition}

\Cref{Prop:LagCat-Adjunction} is an immediate consequence of \cite[Prop.~2.11.1]{CHS}, that is, $1$-morphisms in the $(\infty,2)$-category $\SYMP_{B,d}$ admit adjoints.
Here we provide an $(\infty,1)$-categorical proof. 

\begin{proof}[Proof of \Cref{Prop:LagCat-Adjunction}]
By \Cref{Cor:Corr-Adj}, the correspondence \eqref{Eq:S1-1} gives us an adjunction
\[\xymatrix{
\IsotComp_L \colon 
\Isot_{M,d}  \ar@<.4ex>[r] & \ar@<.4ex>[l] \Isot_{N,d} } \colon \IsotComp_{L^{\swap}}
\]
To show that this adjunction factors through the desired adjunction of Lagrangian categories,
it suffices to show the following property:
for $d$-shifted Lagrangians $L_1 \to M$, $L_2 \to N$, and a morphism $C \colon \LagComp_L(L_1) \dashrightarrow L_2$ in $\Isot_{N,d}$, with the adjunct $C^{\#}\colon L_1 \dashrightarrow \LagComp_{L^{\swap}}(L_2)$ in $\Isot_{M,d}$,
\[\text{$C \colon \LagComp_{L}(L_1)\dashrightarrow L_2$ lies in $\Lag_{N,d}$} \iff \text{$C^{\#}\colon L_1 \dashrightarrow\LagComp_{L^{\swap}}(L_2)$ lies in $\Lag_{M,d}$}.\]
This follows from the canonical equivalence
\[(L_1\times_M L)\times_N L_2 \simeq L_1\times_M (L \times_N L_2)\]
of $(d-1)$-shifted symplectic stacks over $B$. 
\end{proof}

We particularly focus on the following special case:
Let $L \to M$ be a $(d+1)$-shifted Lagrangian.
\begin{itemize}
\item The \defterm{Lagrangian intersection functor}
\begin{equation}\label{Eq:LagInt}
\LagInt_{L/M} \colon \Lag_{M,d+1} \lra \Symp_{B,d}
\end{equation}
is the Lagrangian composition functor \eqref{Eq:LagComp} along the Lagrangian correspondence $M \dashrightarrow B$ induced by the Lagrangian $L \to M$.
This functor sends a Lagrangian $L_1 \to M$ to the Lagrangian intersection
$L_1\times_M L.$
\item The \defterm{trivial Lagrangian functor}
\begin{equation}\label{Eq:TrivLag}
\TrivLag_{L/M} \colon \Symp_{B,d}\lra \Lag_{M,d+1}
\end{equation}
is the Lagrangian composition functor \eqref{Eq:LagComp} along the Lagrangian correspondence $B \dashrightarrow M$ induced by the Lagrangian $L \to M$.
This functor sends a $d$-shifted symplectic stack $N$ over $B$ to the product Lagrangian
$(L \to M) \times_B (N \to B).$
\end{itemize}

\begin{corollary}\label{Prop:SympAdj}
Let $M$ be a $(d+1)$-shifted symplectic stack over $B$
and $L$ be a $(d+1)$-shifted Lagrangian on $M$.
Then there exists a canonical adjunction
\[\xymatrix{
\LagInt_{L/M} \colon 
\Lag_{M,d+1}  \ar@<.4ex>[r] & \ar@<.4ex>[l] \Symp_{B,d} } \colon \TrivLag_{L/M}
.\]
\end{corollary}

\subsection{Lagrangian fibrations}

We review the notion of Lagrangian fibrations.
We will provide a technically useful lemma for a Lagrangian fibration: the non-degeneracy of isotropic morphisms to the total space can be detected by the non-degeneracy of the induced closed $2$-forms over the base.

Let $B$ be a derived stack, $d$ be an integer,
and $M$ be a $(d+1)$-shifted symplectic stack over $B$.

\begin{definition}
A $(d+1)$-shifted \defterm{Lagrangian fibration} $M \to U$ over $B$ consists of:
\begin{itemize}
\item a morphism $\pi\colon M \to U$ of derived stacks over $B$, that are locally geometric and locally of finite presentation,
\item a commutative square of derived stacks
\begin{equation}\label{Eq:Lagfib}
\xymatrix@R-1pc{
M \ar[r]^-{(\theta_M,\pi)} \ar[d] & \sA^{2,\cl}_B[d+1] \times_B U \ar[d] \\
U \ar[r]^-{0} & \sA^{2,\cl}_U[d+1]
}\end{equation}
such that the following induced null-homotopic sequence is a cofiber sequence
\[\TT_{M/U} \to \TT_{M/B} \simeq \LL_{M/B}[d+1] \to \LL_{M/U}[d+1].\]
\end{itemize}
\end{definition}

Let $\pi\colon M \to U$ be a $(d+1)$-shifted Lagrangian fibration.
Then the graph
$(\id,\pi) \colon M \to M\times_B U$
is a $(d+1)$-shifted Lagrangian over $U$.
In particular, for any section $s\colon B \to U$ of $U \to B$, the fiber
\[M_s\coloneq M\times_{\pi,U,s} B \to M\]
is a $(d+1)$-shifted Lagrangian over $B$.

We have the \defterm{fiberwise Lagrangian intersection functor}
\begin{equation}\label{Eq:Funct-LagFib}
\LagInt_{M/U}^{\fib} \colon\Lag_{M,d+1} \lra \Symp_{U,d}.
\end{equation}
Indeed, we have a canonical map of derived stacks
\begin{equation}\label{Eq:Funct-LagFib-1}
\mathrm{isotInt}^{\fib}_{M/U} \colon M^{\isot} 
\simeq \fib\left(M \xrightarrow{(\theta_M,\pi)}\sA^{2,\cl}_B[d+1]\times_B U\right) 
\lra \sA^{2,\cl}_U[d],
\end{equation}
induced by the commutative square \eqref{Eq:Lagfib}. 
Applying the functor $\Corr_{(-)} \colon \Corr(\dSt) \to \Cat$ in \Cref{Const:Corr-Funct}, we obtain a functor
\[\IsotInt_{M/U}^{\fib}\colon \Isot_{M,d+1} \lra \pSymp_{U,d}.\]
This functor not only preserves the non-degeneracy, but also {\em detects} the non-degeneracy.

\begin{lemma}\label{Lem:LagFib-Non-degeneracy}
Let $M \to U$ be a $(d+1)$-shifted Lagrangian fibration.
We have a cartesian square
\[\xymatrix{
\Lag_{M,d+1} \ar@{^{(}->}[r] \ar[d]_{\LagInt_{M/U}^{\fib}} \cart &  \Isot_{M,d+1} \ar[d]^{\IsotInt_{M/U}^{\fib}} \\
\Symp_{U,d} \ar@{^{(}->}[r] &  \pSymp_{U,d} .
}\]
\end{lemma}

\begin{proof} 
For a morphism $L \to M$ of derived stacks over $B$, the morphism $L \to U$ is locally geometric and locally finitely presented if and only if the composition $L \to B$ is locally geometric and locally finitely presented,
since $U \to B$ is locally geometric and locally finitely presented.

For a $(d+1)$-shifted isotropic morphism $L \to M$, it is a Lagrangian if and only if the induced $d$-shifted closed $2$-form on $L \to M \to U$ is symplectic, by \cite[Lem.~2.4.4]{Park2}.

For $(d+1)$-shifted Lagrangians $L_1 \to M$ and $L_2 \to M$, the map $L_1\times_M L_2 \to L_1\times_U L_2$ has a canonical $(d+1)$-shifted Lagrangian correspondence fibration structure (in the sense of \cite[Def.~2.4.1]{Park2}) by \cite[Lem.~2.4.3]{Park2}.
Applying \cite[Lem.~2.4.4]{Park2} again, a morphism $C\colon L_1 \dashrightarrow L_2$ in $\Isot_{M,d+1}$, identified to an isotropic morphism $C \to L_1\times_M L_2$, is a Lagrangian if and only if the induced isotropic morphism $C \to L_1\times_U L_2$ is Lagrangian.
\end{proof}

\subsection{Cotangent bundles} 

We finally provide an important example in symplectic geometry: the {\em cotangent bundles}.
We will also explain how to understand the symplectic pushforwards in \cite{Park2}. 

For a perfect complex $E \in \QCoh(U)$ on a derived stack $U$, 
the \defterm{total space} is the derived stack
\[\Tot_U(E) \in \dSt_U \]
uniquely characterized by the following universal property:
for any $(S \xrightarrow{s}U)\in \dSt_U$,
\[\Map_{\dSt_U}(S,\Tot_U(E)) \simeq \Map_{\QCoh(S)}(\O_S,s^*E).\]
By abuse of notation, we sometimes use the same letter $E$ to denote the associated total space.

\begin{example}\label{Ex:Cotangent}
Let $p\colon U \to B$ be a morphism of derived stacks that is locally geometric and locally of finite presentation.
The $d$-shifted \defterm{cotangent bundle} of $p\colon U \to B$ is the total space
\[\T^*[d](U/B) \coloneqq \Tot_U(\LL_{U/B}[d]).\]
By \cite[Thm.~2.4]{Cal2}, we have the following canonical symplectic/Lagrangian structures:
\begin{enumerate}
\item the cotangent bundle $\T^*[d](U/B)$ has a canonical $d$-shifted symplectic structure over $B$,
\item the projection $\T^*[d](U/B) \to U$ has a canonical $d$-shifted Lagrangian fibration structure,
\item the zero section $0\colon U \to \T^*[d](U/B) $ has a canonical $d$-shifted Lagrangian structure.
\end{enumerate}
\end{example}

The pullback functor between symplectic categories has the following canonical decomposition:

\begin{lemma}\label{Lem:DecompositionofPullback}
Let $p\colon U \to B$ be a morphism of derived stacks that is locally geometric and locally of finite presentation..
Then we have a canonical commutative triangle
\[\xymatrix@R-.2pc{
\Symp_{B,d} \ar[rd]_-{\TrivLag_{U/\T^*[d+1](U/B)}\qquad} \ar[rr]^-{\pullback_{U/B}} && \Symp_{U,d}  \\
&\Lag_{\T^*[d+1](U/B),d+1} \ar[ru]_-{\qquad \LagInt^{\fib}_{\T^*[d+1](U/B)/U}} .& 
}\]
\end{lemma}
\begin{proof}
Let $M \coloneq \T^*[d+1](U/B)$.
Then we have a canonical commutative triangle
\[
\xymatrix@-1pc{
\sA^{2,\cl}_B[d] \ar@{.>}[rr]^-{\eqref{Eq:pullback-SympCat-1}} \ar@{.>}[rd]_-{a} & & \sA^{2,\cl}_U[d]  \\
& M^{\isot} \ar@{.>}[ru]_-{\isotInt_{M/U}} &
}
\]
in $\Corr(\dSt)$, 
where $a$ is the morphism in \eqref{Eq:S1-1} for the Lagrangian correspondence $B \dashrightarrow M$ induced by the zero section Lagrangian.
Applying the functor $\Corr_{(-)} \colon \Corr(\dSt) \to \Cat$, we obtain the presymplectic version of \Cref{Lem:DecompositionofPullback} without the non-degeneracy.
This presymplectic commutative triangle preserves the non-degeneracy and we have the desired commutative triangle.
\end{proof}

We can describe the symplectic pushforwards in \cite{Park2}
as follows:

\begin{remark}
Everything in this section works in the exact setting. 
We add $(-)^{\ex}$ for the exact versions of categories and functors used in this section.
What changes in the exact setting is:
for a locally finitely presented geometric morphism $U \to B$ of derived stacks, the exact version of \eqref{Eq:Funct-LagFib-1},
\[\isotInt^{\fib,\ex}_{\T^*[d+1](U/B)/U} \colon\T^*[d+1](U/B)^{\isot,\ex}\lra{} \sA^{2,\ex}_U[d]\]
is an {\em equivalence} \cite[Prop.~2.3.1]{Park2}.
Hence 
the exact version of \Cref{Lem:LagFib-Non-degeneracy} gives us an equivalence
\[\LagInt^{\fib,\ex}_{\T^*[d+1](U/B)/U} \colon \Lag^{\ex}_{\T^*[d+1](U/B),d+1} \xrightarrow{\simeq} \Symp_{U,d}^{\ex}.\]
Therefore, by the exact versions of \Cref{Prop:SympAdj} and \Cref{Lem:DecompositionofPullback}, the composition
\[\Symp_{U,d}^{\ex} \xleftarrow[\simeq]{\LagInt^{\fib,\ex}_{\T^*[d+1](U/B)/U}} \Lag^{\ex}_{\T^*[d+1](U/B),d+1} \xrightarrow{\LagInt_{U/\T^*[d+1](U/B)}^{\ex}} \Symp_{B,d}^{\ex} \]
is the left adjoint of 
the pullback functor, 
$\pullback_{U/B}^{\ex} \colon \Symp_{B,d}^{\ex} \to \Symp_{U,d}^{\ex}$.
\end{remark}

\section{Hamiltonian reduction}\label{Sec:HamRed}

In this section, we provide a general theory of {\em Hamiltonian reduction} in shifted symplectic geometry,
following \cite{Cal,Saf1,AC}.
The following points will be crucial in the next section \S\ref{Sec:SympRig}.
\begin{enumerate}
\item By definition, a group action on a shifted symplectic stack is {\em symplectic} if the symplectic form is {\em invariant} under the group action.
\item A useful observation (Proposition \ref{Prop:HamilviaEquiv}) is that a symplectic action is {\em Hamiltonian} if and only if the de Rham form associated to the symplectic form is {\em equivariant}.
\end{enumerate}
We will present the theory using the symplectic categories of the previous section \S\ref{Sec:SympCat}.

Throughout this section, we work with the relative setting.
We fix an arbitrary base derived stack $B$.
All group stacks, group actions, and shifted symplectic stacks will be considered over $B$.
We sometimes omit $B$ in the notations when it is clear from the context.

\subsection{Group actions}

We first review the ordinary group actions on derived stacks.

We call a {\em group stack} for a group object (in the sense of \cite[Def.~7.2.2.1]{LurHTT}) in the category of derived stacks $\dSt_B$.
More precisely, the $\infty$-category of \defterm{group stacks} is the full subcategory
\[\Grp_B \subseteq \Fun(\Delta^{\op},\dSt_B)\]
consisting of functors $G \colon \Delta^{\op} \to \dSt_B$ 
such that $G_0\simeq B$ and for any partition $[n] = S_1\cup S_2$ of $[n]\in \Delta$ with $\#(S_1 \cap S_2)=1$, the induced map $G_n \to G(S_1)\times_B G(S_2)$ is an equivalence.

We have the functor of \defterm{classifying stacks}
\[\Grp_B \longrightarrow \dSt_B \colon G\mapsto BG:={\colim}_{n \in \Delta^{\op}}G_{n}.\]

\begin{example}\label{Ex:ClassifyingGroupstack}
The main group stack in this paper is the classifying stack \[BT\in \Grp_B\] of the $r$-dimensional torus $T\coloneq \GG_{m,B}^{\times r}$ over $B$.
The group structure can be given as follows:

A \defterm{commutative group stack} is a group object in $\Grp_B$,\footnote{To be precise, we should call this an $E_2$-group stack. For a {\em classical} group scheme (e.g. the torus $\GG_m^{\times r}$), this $E_2$-commutativity is equivalent to the commutativity in the classical sense.
}  i.e., a functor
	\[G : \Delta^{\op} \times \Delta^{\op} \to \dSt_B\]
such that
$G_{\bullet, n} \colon \Delta^{\op}  \to \dSt_B$ and $G_{m, \bullet} \colon \Delta^{\op}  \to \dSt_B$ are group stacks for all $m$ and $n$.

For a commutative group stack $G$, we can form the classifying stack as a {\em group} stack:
\[BG:={\colim}_{n \in \Delta^{\op}}G_{\bullet,n} \textin \Grp_B.\]
In particular, we also have the {\em second classifying stack}
$B^2G:=B(BG) \in \dSt_B$.

Since $\GG_m^{\times r}$ is a (classical) commutative group scheme, the pullback $T=\GG_{m,B}^{\times r}\coloneqq \GG_m^{\times r}\times B$ is also a commutative group stack, and hence we have the classifying stack $BT$ as a group stack.
\end{example}

Let $G$ be a group stack.
Following \cite[Def.~4.2.2.2]{LurHA},
we define the $\infty$-category of \defterm{derived stacks with $G$-actions} as the full subcategory
\[\dSt_B^G \subseteq \Fun(\Delta^{\op},\dSt_B)_{/G}\]
consisting of morphisms $M \to G$ of functors $\Delta^{\op} \to \dSt_B$ such that for any $[n]\in \Delta$, the map $M_n \to M_0 \times_B G_n$, induced by the map $[0] \to [n] \colon 0 \mapsto n$ in $\Delta$, is an equivalence.

We have the functor of \defterm{quotient stacks}
\[\dSt_B^G \xrightarrow{\simeq} \dSt_{BG} \colon M \mapsto M/G\coloneq \colim_{n\in \Delta^{\op}} M_n,\]
which is an equivalence of $\infty$-categories, since $B \to BG$ is an effective epimorphism in the sense of \cite[\S6.2.3]{LurHTT}.

\subsection{Symplectic actions}

In classical symplectic geometry, a group action on a symplectic manifold is called {\em symplectic} if the symplectic form is invariant under the group action.
We will consider an analogous definition in shifted symplectic setting.

We first fix some notations and terminologies for differential forms on derived stacks with group actions.
It will be important to distinguish {\em invariant} forms and {\em equivariant} forms.

\begin{definition}
Let $M$ be a derived stack with an action of a group stack $G$.
\begin{enumerate}
\item The space of \defterm{$G$-invariant} $d$-shifted closed $p$-forms on $M$ is
\[\sA^{p,\cl}_{G\textnormal{-}\inv}(M,d)\coloneq \sA^{p,\cl}(M/G \to BG,d).\]
\item The space of \defterm{$G$-equivariant} $d$-shifted closed $p$-forms on $M$ is
\[\sA^{p,\cl}_{G\textnormal{-}\equiv}(M,d)\coloneq \sA^{p,\cl}(M/G\to B,d).\]
\end{enumerate}
The spaces of $G$-invariant/equivariant ordinary/exact/de Rham forms are defined analogously.
\end{definition}

We have a canonical commutative triangle of forgetful maps
\begin{equation*}
\xymatrix{
 \sA^{p,\cl}_{G\textnormal{-}\equiv}(M,d) \ar[rr]^-{\equiv\to\inv} \ar[rd]_{\cancel{\equiv}} &&  \sA^{p,\cl}_{G\textnormal{-}\inv}(M,d) \ar[ld]^{\cancel{\inv}} \\& \sA^{p,\cl}(M,d) , &
}\end{equation*}
where we used an abbreviation $\sA^{p,\cl}(M,d)\coloneq \sA^{p,\cl}(M \to B,d)$.

A \defterm{symplectic action} of a group stack $G$ on a $d$-shifted symplectic stack $M$ is an action of $G$ on the underlying derived stack of $M$ together with the following additional data:
\begin{enumerate}
\item [(D1)] a $G$-invariant closed $2$-form $\theta_M^{\inv} \in \sA^{2,\cl}_{G\textnormal{-}\inv}(M,d)$,
\item [(D2)] an equivalence of $d$-shifted symplectic structures $\cancel{\inv}(\theta_M^{\inv})\simeq \theta_M $ in $\sA^{2,\cl}(M,d)$.
\end{enumerate}
More precisely, the {\em space} of symplectic structures for a $G$-action on $M$ is the fiber
\begin{equation}\label{Eq:Space-sympactions}
\xymatrix{
\sSymp(G\circlearrowleft M) \ar[r] \ar[d] \cart & \sA^{2,\cl}_{G\textnormal{-}\inv} (M,d) \ar[d]^-{\cancel{\inv}} \\
\star \ar[r]^-{\theta_M} & \sA^{2,\cl}(M,d).
}\end{equation}

We can moreover form a {\em category} of symplectic actions for a group stack $G$.

\begin{definition}\label{Def:G-SympCat}
The \defterm{$G$-equivariant symplectic category} is the symplectic category
\[\Symp^G_{B,d} \coloneq \Symp_{BG,d}.\]
\end{definition}

Then objects in $\Symp_{B,d}^G$ correspond to $d$-shifted symplectic stacks $M$ with symplectic $G$-actions.

\begin{notation}
\begin{enumerate}
\item 
The \defterm{forgetful functor}
\[\forget_{\Symp} 
\colon \Symp^G_{B,d} \lra \Symp_{B,d},\]
is the pullback functor \eqref{Eq:pullback-SympCat} along the section $s\colon B \to BG$.
\item The \defterm{functor of trivial symplectic actions}
\[\triv_{\Symp} 
\colon \Symp_{B,d} \longrightarrow \Symp_{B,d}^G,\]
is the pullback functor \eqref{Eq:pullback-SympCat} along the projection map $\pi\colon BG \to B$.
\end{enumerate}
\end{notation}

\subsection{Hamiltonian actions}\label{ss:Hamiltonian}

In classical symplectic geometry, a group action on a symplectic manifold is called {\em Hamiltonian} if it admits a {\em moment map}, i.e. a map from the manifold to the (dual) Lie algebra satisfying certain natural properties.
For a Hamiltonian action, we can construct a symplectic version of a quotient, called the {\em Hamiltonian reduction}, as the quotient of the zero locus of the moment map.
In shifted symplectic geometry, an elegant idea due to \cite{Cal,Saf1} is to consider the moment map as a shifted Lagrangian.

We first recall the symplectic geometry of the Lie algebra of a group stack.
We say that a group stack $G \in \Grp_B$ is \defterm{smooth} if the underlying derived stack $G_1\in \dSt_B$ is geometric and smooth.

\begin{notation}
Let $G$ be a smooth group stack.
The \defterm{Lie algebra}\footnote{Even though we are calling $\g$ as the {\em Lie algebra}, we will not use its Lie algebra structure.
We only consider $\g$ as a mere ($G$-equivariant) perfect complex.} of $G$ is the tangent complex
\[\g \coloneq \TT_{G/B,\id}\]
at the identity section $\id \colon B \to G$.
Then we have a canonical equivalence of derived stacks
\[\T^*[d+1](BG/B) \simeq \g\dual[d]/G,\]
for the canonical $G$-action on $\g$.
By \cite{Cal2} (see \Cref{Ex:Cotangent}), we have the following structures:
\begin{itemize}
\item the quotient stack $\g\dual[d]/G$ has a canonical $(d+1)$-shifted symplectic structure over $B$,
\item the projection $\g\dual[d]/G \to BG$ has a canonical $(d+1)$-shifted Lagrangian fibration structure,
and the projection $\g\dual[d] \to \g\dual[d]/G$ has an induced $(d+1)$-shifted Lagrangian structure,
\item the zero section $0\colon BG \to \g\dual[d]/G$ has a canonical $(d+1)$-shifted Lagrangian structure.
\end{itemize}

\end{notation}

A \defterm{Hamiltonian action} of a smooth group stack $G$ on a $d$-shifted symplectic stack $M$ is an action of $G$ on the underlying derived stack of $M$ together with the following additional data:
\begin{enumerate}
\item [(D1)] a $G$-equivariant morphism $\mu \colon M \to \g\dual[d]$,
\item [(D2)] a $(d+1)$-shifted Lagrangian structure on $\mu/G \colon M/G \to \g\dual[d]/G \simeq \T^*[d+1](BG/B)$,
\item [(D3)] an equivalence of the $d$-shifted symplectic structure $\theta_M$ on $M$ and the induced symplectic structure via the Lagrangian intersection diagram
\[\xymatrix@R-.5pc{
M \ar[r]^-{\mu} \ar[d] \cart & \g\dual[d] \ar[d] \\
M/G \ar[r]^-{\mu/G} & \g\dual[d]/G.
}\]
\end{enumerate}

More precisely, the {\em space} of Hamiltonian structure for a $G$-action on $M$ is the fiber product
\begin{equation}\label{Eq:Space-HamActions}
\xymatrix{
\sHamil(G\circlearrowleft M) \ar[r] \ar[d] \cart & \Map_{\dSt^G_B}(M, \g\dual[d])^{\isot}\ar[d] \\
\sSymp(G\circlearrowleft M) \ar[r] & \sA^{2,\cl}_{G\textnormal{-}\inv} (M,d),
}
\end{equation}
where $\Map_{\dSt^G_B}(M, \g\dual[d])^{\isot}\coloneq\Map_{\dSt_{BG}}(M/G, (\g\dual[d]/G)^{\isot})$ and
the right vertical map is given by the fiberwise isotropic intersection map \eqref{Eq:Funct-LagFib-1} along the Lagrangian fibration $\g\dual[d]/G \to BG$.

For a Hamiltonian action of $G$ on $M$, the \defterm{Hamiltonian reduction} $M/\!/G$ is the $d$-shifted symplectic stack, defined as the Lagrangian intersection
\begin{equation*} 
    \xymatrix@R-.5pc{M/\!/G \ar[r] \ar[d] \cart & BG \ar[d]^0 \\ M/G \ar[r]^-{\mu/G} & \g\dual[d]/G.}  
\end{equation*}
We have a canonical Lagrangian correspondence for Hamiltonian reduction,
\begin{equation}\label{Eq:HamRed-LagCor}
\xymatrix@-1pc{
& \mu^{-1}(0)\ar[ld] \ar[rd]& \\ M/\!/G && M,
}\end{equation}
by the triple Lagrangian intersection theorem \cite{Ben} for the Lagrangians $M/G,BG,\g\dual[d] \to \g\dual[d]/G$.

We can upgrade the above Hamiltonian actions for a smooth group stack $G$ to {\em categories}.

\begin{definition}\label{Def:HamCat}
The $G$-equivariant \defterm{Hamiltonian category} is the Lagrangian category
\[\Hamil^G_{B,d} \coloneq \Lag_{\g\dual[d]/G,d+1}.\]
\end{definition}

Objects in $\Hamil^G_{B,d}$ correspond to $d$-shifted symplectic stacks with Hamiltonian $G$-actions.

\begin{notation}\label{Notation:Functors-Forget/HamRed/TrivHam} \
\begin{enumerate}
\item The \defterm{forgetful functor}
\[\forget_{\Hamil} 
\colon \Hamil^G_{B,d} \longrightarrow \Symp^G_{B,d},\]
is the fiberwise Lagrangian intersection functor \eqref{Eq:Funct-LagFib} along the projection $\g\dual[d]/G \to BG$.

\item The \defterm{Hamiltonian reduction functor} 
\[\red_{\Hamil} 
\colon \Hamil^G_{B,d} \longrightarrow \Symp_{B,d},\]
is the Lagrangian intersection functor \eqref{Eq:LagInt} along the zero section $0 \colon BG \to \g\dual[d]/G$.

\item The \defterm{functor of trivial Hamiltonian actions}  
\[\triv_{\Hamil} 
\colon \Symp_{B,d}\longrightarrow \Hamil^G_{B,d} ,\]
is the trivial Lagrangian functor \eqref{Eq:TrivLag} along the zero section $0 \colon BG \to \g\dual[d]/G$.
\end{enumerate}
\end{notation}

Immediately from the results in \S\ref{Sec:SympCat}, we have the following:
\begin{itemize}
\item By \Cref{Prop:SympAdj}, we have a canonical adjunction
\[\xymatrix{
\red_{\Hamil} \colon 
\Hamil^G_{B,d}  \ar@<.4ex>[r] & \ar@<.4ex>[l] \Symp_{B,d} } \colon \triv_{\Hamil}.\]
\item By \Cref{Lem:DecompositionofPullback}, we have a canonical commutative triangle
\[\xymatrix@R-1pc{\Symp_{B,d}\ar[rr]^-{\triv_{\Hamil}} \ar[rd]_-{\triv_{\Symp}} && \Hamil^G_{B,d} \ar[ld]^{\quad\forget_{\Hamil}} \\ &  \Symp^G_{B,d} .& }\]
\end{itemize}
The canonical Lagrangian correspondence \eqref{Eq:HamRed-LagCor} can be extended to the natural transformation
\begin{equation}\label{Eq:HamRed-LagCor-funtorial}
\red_{\Hamil} \simeq \forget_{\Symp}\circ \forget_{\Hamil} \circ \triv_{\Hamil} \circ  \red_{\Hamil} \xrightarrow{\mathrm{counit}}\forget_{\Symp}\circ \forget_{\Hamil}  
\end{equation}
of functor $\Hamil^G_{B,d} \to \Symp_{B,d}$.

We provide a useful description of Hamiltonian actions:
a symplectic action is Hamiltonian if and only if the $G$-invariant de Rham form associated to the symplectic form is $G$-equivariant.

\begin{proposition}
\label{Prop:HamilviaEquiv}
Let $M$ be a $d$-shifted symplectic stack with a symplectic action of a smooth group stack $G$.
Then there exists a canonical equivalence  of spaces
\[\sHamil(G\circlearrowleft M)  \simeq \fib\left(\sA^{\DR}_{G\textnormal{-}\equiv}(M,d+2) \xrightarrow{\equiv \to\inv} \sA^{\DR}_{G\textnormal{-}\inv}(M,d+2), [\theta^{\inv}_M]\right).\]
\end{proposition}

\begin{proof}
We first claim that there is a canonical cartesian diagram of spaces
\[\text{\small\xymatrix@C+1pc{
\sA^{2,\ex}_{G\textnormal{-}\equiv}(M,d+1)  \ar[r] \ar[d]_{\equiv\to\inv} \ar@{}[rd]|{\textnormal{(A)}} & \sA^{2,\cl}_{G\textnormal{-}\equiv}(M,d+1) \ar@{.>}[d]^-{(\equiv\to\inv,[-])} \\
\sA^{2,\ex}_{G\textnormal{-}\inv}(M,d+1) \ar[d] \ar@{.>}[r]^-{(\equiv\to \inv,0)} \ar@{}[rd]|{\textnormal{(B)}} & \sA^{2,\cl}_{G\textnormal{-}\inv}(M,d+1) \bigtimes_{\sA^{\DR}_{G\textnormal{-}\inv}(M,d+3)}\sA^{\DR}_{G\textnormal{-}\equiv}(M,d+3) \ar[r]^-{\pr_1} \ar[d]^{\pr_2} \ar@{}[rd]|{\textnormal{(C)}} & \sA^{2,\cl}_{G\textnormal{-}\inv}(M,d+1) \ar[d]^-{[-]} \\
\star \ar[r]^-{0} & \sA^{\DR}_{G\textnormal{-}\equiv}(M,d+3)  \ar[r]^-{\equiv\to \inv} & \sA^{\DR}_{G\textnormal{-}\inv}(M,d+3) .
}}\]
Indeed, we can form the above commutative diagram from the canonical map of filtered complexes $\DR(M/G \to B) \to \Gamma(BG, \DR(M/G \to BG))$.
Since the squares (C), (B)$+$(C), and (A)$+$(B) are cartesian, the square (B) is also cartesian, and hence the square (A) is also cartesian.

By \cite[Lem.~2.3.2]{Park2}, we have a canonical fiber square of spaces
\[\xymatrix@R-.5pc{
\Map_{\dSt_B^G}\left(M,\mathfrak{g}\dual[d]\right) \ar[r]^{} \ar[d] \cart & \sA^{2,\ex}_{G\textnormal{-}\equiv}(M,d+1) \ar[d]^{\equiv \to \inv} \\
\star \ar[r]^-0 & \sA^{2,\ex}_{G\textnormal{-}\inv}(M,d+1) .
}\]
From the above cartesian square (A), we obtain a canonical equivalence of spaces
\begin{equation} \label{Eq:IsotMapSpace}
    \Map_{\dSt_{B}^G}(M,\g\dual[d])^{\isot} \simeq \sA^{2,\cl}_{G\textnormal{-}\inv}(M,d) \times_{\sA^{\DR}_{G\textnormal{-}\inv}(M,d+2)}\sA^{\DR}_{G\textnormal{-}\equiv}(M,d+2).
\end{equation}
By taking the fiber over $\theta_M^\inv \in \sA^{2,\cl}_{G\textnormal{-}\inv}(M,d)$, we obtain the desired equivalence.
\end{proof}

One immediate corollary of Proposition \ref{Prop:HamilviaEquiv} is that exact symplectic actions are Hamiltonian.\footnote{This was already observed in \cite[\S3.3]{Park2} using the exact version of symplectic categories.}
Since negatively-shifted closed $2$-forms are exact in the absolute case ($B=\pt$) \cite[Prop.~3.2]{KPS}, we can deduce the following result:

\begin{corollary}\label{Cor:ExactSympAct}
Let $M$ be a $d$-shifted symplectic stack over $B=\pt$ and $G$ be a linearly reductive group.
If $M_{\cl}$ is quasi-compact with affine stabilizers and $d<0$, then any symplectic action of $G$ on $M$ is Hamiltonian.
\end{corollary}

\begin{proof}
By \Cref{Prop:HamilviaEquiv}, it suffices to show that $[\theta_M^{\inv}] \simeq 0$ in $\sA^{\DR}_{G\textnormal{-}\inv}(M,d+2)$.
Since $[\theta_M] \simeq 0$ in $\sA^{\DR}(M,d+2)$ by \cite[Prop.~3.2]{KPS}, it suffices to show that the forgetful map
\[{\cancel{\inv}}:\sA^{\DR}_{G\textnormal{-}\inv}(M,d+2) \to \sA^{\DR}(M,d+2)\]
is injective on $\pi_0$.

We have a canonical equivalence
\[ (B \to BG)^*  \Fil^0\DR(M/G\to BG) \simeq  \Fil^0\DR(M) \textin \QCoh(B), \]
since the pushforward along $M/G \to BG$ has the base-change by \cite[Cor.~1.4.5]{DG}.
Since $G$ is linearly reductive, the counit map 
\[\pi^* \circ \pi_* \xrightarrow{} \id \textin \Fun(\QCoh(BG), \QCoh(BG)) \]
for the projection map $\pi \colon BG \to B$, admits a retract (see e.g. \cite[Lem.~A.8]{AKLPR}).
Hence 
\[ \Gamma(BG, \Fil^0\DR(M/G\to BG)) \to \Fil^0\DR(M)\]
also admits a retract and hence injective on $H^0$.
\end{proof}

Corollary \ref{Cor:ExactSympAct} is a useful result, but we will not use it in this paper.
We focus on the classifying stack $BT$ of a torus $T\simeq\GG_m^{\times r}$, instead of reductive groups.

\subsection{Change of groups} 

We provide a functorial property of the Hamiltonian reduction.

Let $G \to H$ be a morphism of smooth group stacks.
Then we have a canonical $(d+1)$-shifted Lagrangian correspondence
\[\xymatrix@-1pc{ 
& \h\dual[d]/G \ar[ld] \ar[rd] & \\
\g\dual[d]/G & & \h\dual[d]/H,
}\]
by \cite[Thm.~2.13]{Cal2}, where $\h\coloneq \TT_{H/B,\id}$ is the Lie algebra of $H$.

\begin{itemize}
\item The \defterm{induction functor}
\[\ind_{G/H} 
\colon \Hamil^G_{B,d} \longrightarrow \Hamil_{B,d}^H\]
is the Lagrangian composition functor \eqref{Eq:LagComp} along $\h\dual[d]/G\colon \g\dual[d]/G \dashrightarrow \h\dual[d]/H$.
\item The \textbf{\em restriction functor} 
\[ \res_{G/H} 
\colon \Hamil^H_{B,d} \longrightarrow \Hamil^G_{B,d},\]
is the Lagrangian composition functor \eqref{Eq:LagComp} along the swap $
\h\dual[d]/H\dashrightarrow  \g\dual[d]/G $.
\end{itemize}
By \Cref{Prop:LagCat-Adjunction}, we have a canonical adjunction $\ind_{G/H} \dashv \res_{G/H}.$

These induction functors and restriction functors are functorial along morphisms of smooth group stacks.
Let $\Grp_B^{\sm} \subseteq \Grp_B$ be the full subcategory consisting of smooth group stacks.
Let 
$\Corr(\Grp_B^{\sm})_{\all/\sm.\surj} \subseteq \Corr(\Grp_B^{\sm})$
be the subcategory consisting of the same objects and morphisms $G \xleftarrow{s} K \xrightarrow{t} H$ 
such that $t$ is smooth surjective.

\begin{proposition}\label{Prop:Ind/Res-Functoriality}
There exists a canonical $\infty$-functor
\[\Hamil^{(-)}_{B,d}\colon \Corr(\Grp_B^{\sm})_{\all/\sm.\surj} \lra \Cat,\]
which sends 
\begin{itemize}
\item a smooth group stack $G$ to the Hamiltonian category $\Hamil^G_{B,d}$,
\item a morphism $G \xleftarrow{\id} G \to H$ in $\Corr(\Grp_B^{\sm})_{\all/\sm.\surj}$ to the induction functor $\ind_{G/H}$, and
\item a morphism $G \xleftarrow{} H \xrightarrow{\id} H$ in $\Corr(\Grp_B^{\sm})_{\all/\sm.\surj}$ to the restriction functor $\res_{H/G}$.
\end{itemize}
\end{proposition}

\begin{proof}
Consider the $\infty$-functor of classifying stacks
\[B|_{\sm}\colon\Grp_B^{\sm} \lra \dSt_B^{\lfp},\quad G \mapsto BG\]
where $\dSt_B^{\lfp}  \subseteq \dSt_B$ is the full subcategory consisting of derived stacks that are locally geometric and locally of finite presentation over $B$.
By \Cref{Lem:classifyingstacks-cartesian} below, applying the functor $\Corr \colon \Trip \to \Cat$ of correspondence categories (\Cref{Def:Corr}), we obtain an $\infty$-functor
\[
\Corr(B|_{\sm})\colon\Corr(\Grp_B^{\sm})_{\all/\sm.\surj} \lra \Corr(\dSt_B^{\lfp}).
\]
By \cite[Thm.~3.13]{CalSaf}, we have an $\infty$-functor
\[\T^*[d+1] \colon \Corr(\dSt^{\lfp}_B) \lra \Symp_{B,d+1}\]
which sends $M\in \dSt_B^{\lfp}$ to the cotangent bundle $\T^*[d+1](M/B)$
and a morphism $M \xleftarrow{} L \xrightarrow{} N$ in $\Corr(\dSt_{B}^{\lfp})$ to the cotangent correspondence $\T^*[d+1](M/B) \leftarrow \T^*[d+2](L/M\times N) \rightarrow \T^*[d+1](N/B)$.
We define the desired functor as the composition
\[\Hamil_{B,d}^{(-)}\colon \Corr(\Grp_B^{\sm})_{\all/\sm.\surj} \xrightarrow{\Corr(B|_{\sm})}  \Corr(\dSt_B^{\lfp} )\xrightarrow{\T^*[d+1]} \Symp_{B,d+1} \xrightarrow{\Lag_{(-),d+1}} \Cat,\]
where $\Lag_{(-),d+1}$ is the functor of Lagrangian categories in \Cref{Prop:LagComp-Functoriality}.
\end{proof}

We need the following basic lemma on group stacks to complete the proof of \Cref{Prop:Ind/Res-Functoriality}.

\begin{lemma}\label{Lem:classifyingstacks-cartesian}
Consider a cartesian square
\[\xymatrix@-1pc{
G_1 \ar[r] \ar[d] \cart & G_2 \ar[d] \\
H_1 \ar[r] & H_2
} \]   
in $\Grp_B^{\sm}$.
If $G_2 \to H_2$ is smooth surjective, then the induced commutative square
\[\xymatrix@-1pc{
BG_1 \ar[r] \ar[d]  & BG_2 \ar[d] \\
BH_1 \ar[r] & BH_2
} \]   
is a cartesian square in $\dSt_B$.
\end{lemma}
\begin{proof}
Form a cartesian diagram
\[\xymatrix@-1pc{
K \ar[r] \ar[d] \cart & G_1 \ar[r] \ar[d] \cart & G_2 \ar[d] \\
B\ar[r] & H_1 \ar[r] & H_2
} \]   
in $\Grp_B^{\sm}$, where $B$ is the trivial group stack.
Consider the induced commutative diagram
\[\xymatrix@-1pc{
BK \ar[r] \ar[d] \ar@{}[rd]|{(S_1)}  & BG_1 \ar[r] \ar[d]  \ar@{}[rd]|{(S_2)} & BG_2 \ar[d] \\
B\ar[r] & BH_1 \ar[r] & BH_2
} \]   
in $\dSt_B$.
Since $B \to BH_1$ is smooth surjective, 
to prove that the right square $(S_2)$ is cartesian,
it suffices to show that the left square $(S_1)$ and the total square $(S_1)+(S_2)$ are cartesian.

By \cite[Prop.~6.1.2.11]{LurHTT}, the (augmented) groupoid object in $\dSt_B$ associated to the $K$-action on $G_1$ is equivalent to the Cech nerve of $G_1 \to H_1$.
Since $G_1 \to H_1$ is smooth surjective, taking the colimits of the two groupoid objects gives us an equivalence
$G_1/K \simeq H_1.$
Form a commutative diagram
\[\xymatrix@-1pc{
H_1 \ar[r] \ar[d]  \ar@{}[rd]|{(S_3)} & BK \ar[r] \ar[d]  \ar@{}[rd]|{(S_1)} & B \ar[d] \\
B\ar[r] & BG_1 \ar[r] & BH_1
} \]   
in $\dSt_B$. 
Since the left square $(S_3)$ is cartesian by the equivalence $G_1/K \simeq H_1$, the total square $(S_3)+(S_1)$ is cartesian, and the map $B \to BG_1$ is smooth surjective, the right square $(S_1)$ is cartesian.
An analogous argument for $G_2 \to H_2$ shows that the square $(S_1)+(S_2)$ is also cartesian.
\end{proof}

A \defterm{short exact sequence} $K \to G \to H$ of smooth group stacks is a cartesian square
\[
\xymatrix@-1pc{
K \ar[r] \ar[d] \cart & G \ar[d] \\
B \ar[r] & H 
}    
\]
in $\Grp_B^{\sm}$, such that $G \to H$ is smooth surjective and $B$ is the trivial group stack.
By \Cref{Prop:Ind/Res-Functoriality}, a short exact sequence $K \to G \to H$ induces commutative diagrams of $\infty$-categories,
\[\xymatrix@R-1pc{\Hamil_{B,d}^G\ar[rd]_-{\red_{\Hamil}} \ar[rr]^-{\ind_{G/H}} && \Hamil^H_{B,d} \ar[ld]^{\quad\red_{\Hamil}} \\ &  \Symp_{B,d} & }
\raisebox{-1.5em}{\and}  \xymatrix@C+2pc@R-1pc{
\Hamil^G_{B,d} \ar[d]_-{\ind_{G/H}} \ar[r]^{\res_{K/G}} & \Hamil^K_{B,d} \ar[d]^-{\red_{\Hamil}}\\
\Hamil_{B,d}^H \ar[r]^{\forget_{\Hamil}}& \Symp_{B,d}.}\]
More explicitly, given a $G$-Hamiltonian action on a $d$-shifted symplectic stack $M$, the Hamiltonian reduction $M/\!/K$ has an induced $H$-Hamiltonian action, and we have a canonical equivalence
\begin{equation}\label{Eq:thirdisomorphismthm}
M/\!/G \simeq (M/\!/K)/\!/H \textin \Symp_{B,d}.
\end{equation}

\section{Symplectic rigidification}\label{Sec:SympRig}

In this section, we prove our main results:
\begin{enumerate}
\item (\Cref{Thm:Intro-Rig}) all $BT$-actions on non-positively shifted symplectic stacks are Hamiltonian;
\item (\Cref{Intro:Thm-Funct}) existence of the symplectic rigidification functor as the left adjoint of the trivial action functor.
\end{enumerate}

Throughout this section, we work with the relative setting as in \S\ref{Sec:HamRed}.
Let $B$ be an arbitrary base derived stack.
All derived stacks and shifted symplectic stacks are assumed to be over $B$.

Let $T\coloneq \GG_{m,B}^{\times r}$ be the $r$-dimensional torus over $B$ for some integer $r$.
Let $BT$ be the classifying stack of $T$, which we regard as a group stack over $B$ (as in \Cref{Ex:ClassifyingGroupstack}).

\subsection{Results}

We first summarize the results that we will prove in this section.
We postpone the proofs to the subsequent subsections.

Firstly, we will show that all $BT$-actions on shifted symplectic stacks are symplectic.
More precisely, we will show that the spaces of symplectic structures for $BT$-actions on shifted symplectic stacks are contractible.
Moreover, we will upgrade this to the following categorical statement.

\begin{proposition}\label{Prop:BT-symp}
We have a canonical cartesian square of $\infty$-categories
\[\xymatrix@C+3pc@R-1pc
{
\Symp_{B,d}^{BT} \ar[r]^-{\forget_{\Symp}} \ar[d] \cart &\Symp_{B,d}  \ar[d] \\
\Corr_{B}^{BT} \ar[r] & \Corr_B,
}\]
where $\Corr_B^{BT}\coloneq \Corr(\dSt_B^{BT})$.
\end{proposition}

Secondly, we will prove that $BT$-actions on $d$-shifted symplectic stacks are Hamiltonian under the following assumption on the base $B$ and the shift $d$.
As the above results on symplectic actions,
we will show that the spaces of Hamiltonian structures are contractible, and also upgrade the result to equivalences of Hamiltonian categories and symplectic categories.

\begin{assumption}\label{Assumption}
Assume one of the following two conditions on the base $B$ and the shift $d$:
\begin{enumerate}
\item [(A1)] $B=\pt$ and $d\leq 0$,
\item [(A2)] $B$ is a classical Artin stack\footnote{In this paper, the term ``classical Artin stack'' includes classical {\em higher} Artin stacks.}, locally of finite type, and $d \leq -1$.
\end{enumerate}		
\end{assumption}

\begin{theorem}\label{Thm:Ham=Symp}	
Under \Cref{Assumption},
we have a canonical equivalence of $\infty$-categories,
\[\forget_{\Hamil}\colon \Hamil^{BT}_{B,d} \xrightarrow{\simeq} \Symp^{BT}_{B,d}.\]
\end{theorem}

\begin{definition}\label{Def:SympRig}
Under \Cref{Assumption}, the \defterm{symplectic rigidification} functor is the composition
\[\rig_{\Symp} \colon \Symp_{B,d}^{BT} \xleftarrow[\simeq]{\forget_{\Hamil}} \Hamil_{B,d}^{BT} \xrightarrow{\red_{\Hamil}} \Symp_{B,d} . \]
\end{definition}

More explicitly, for a $d$-shifted symplectic stack $M$ with a $BT$-action, the action is symplectic by \Cref{Prop:BT-symp} and is Hamiltonian by \Cref{Thm:Ham=Symp} under \Cref{Assumption}.
The symplectic rigidification $M/\!/BT$ is the Hamiltonian reduction of the $BT$-action, which is the Lagrangian intersection
\[\xymatrix@C-1pc{
M/\!/BT \ar[rr] \ar[d] \ar@{}[rrd]|-{\Box} && B^2T \ar[d]^0 \\
M/BT \ar[r] & \T^*[d+1](B^2T/B) \ar@{}[r]|-{\simeq} & \t\dual[d-1]/BT .
}\]
Here, $\t : = \TT_{T/B,\id}$ is the Lie algebra of $T$.

\begin{corollary}\label{Cor:rigid-adjunction}
Under \Cref{Assumption}, we have a canonical adjunction of $\infty$-categories
\[\xymatrix{
\rig_{\Symp} \colon \Symp_{B,d}^{BT}  \ar@<.4ex>[r] & \ar@<.4ex>[l] \Symp_{B,d} } \colon \triv_{\Symp}.\]
\end{corollary}

\begin{proof}
We have a canonical adjunction $\red_{\Hamil} \dashv \triv_{\Hamil}$ by \Cref{Prop:SympAdj}.
Since $\forget_{\Hamil}$ is an equivalence by \Cref{Thm:Ham=Symp}, 
we have an induced adjunction \[\rig_{\Symp} \coloneq \red_{\Hamil}\circ (\forget_{\Hamil})^{-1} \dashv \forget_{\Hamil} \circ \triv_{\Hamil} \simeq \triv_{\Symp},\]
where $\triv_{\Symp} \simeq \forget_{\Hamil} \circ \triv_{\Hamil}$ by \Cref{Lem:DecompositionofPullback}.
\end{proof}

The key ingredient for proving \Cref{Thm:Ham=Symp} is the {\em contraction map} (\Cref{Def:Contraction})
\begin{equation}\label{Eq:ContractionMap}
\iota \colon \sA^{p,\cl}_{BT\textnormal{-}\inv}(M,d) \longrightarrow \sA^{p-1,\cl}(M,\t\dual[d-1]),
\end{equation}
for a $BT$-action on a derived stack $M$.
Using the contraction map, we can actually show that both the conditions (A1) and (A2) in \Cref{Assumption} are {\em sharp} for \Cref{Thm:Ham=Symp}.
In general, without \Cref{Assumption}, there exists an {\em obstruction} for a $BT$-action being Hamiltonian.

\begin{theorem}\label{Thm:MomentMapEquation}
Let $M$ be a $d$-shifted symplectic stack with $\theta\in \sA^{2,\cl}(M,d)$.
For a $BT$-action on $M$, the space of Hamiltonian structures fits into a canonical commutative square of spaces
\begin{equation}\label{Square:MomentMapEquation}
\xymatrix{
\sHamil(BT \circlearrowleft M) \ar[r] \ar[d]  & \Map(M, \t\dual[d-1]) \ar[d]^{d} \\
\star \ar[r]^-{\iota \theta} & \sA^{1,\cl}(M,\t\dual[d-1]).
}
\end{equation}
Moreover, if we assume that
\begin{enumerate}
\item [(A3)] $B$ is a classical Artin stack, locally of finite type, and $d \leq 2$,
\end{enumerate}
then the above commutative square \eqref{Square:MomentMapEquation} is cartesian.
\end{theorem}

More explicitly, \Cref{Thm:MomentMapEquation} says that for a Hamiltonian $BT$-action on a $d$-shifted symplectic $M$,
we have the \defterm{moment map equation}:
\begin{equation}\label{Eq:MomentMapEquation}
\iota\theta \simeq d\mu \textin \sA^{1,\cl}(M,\t\dual[d-1]).
\end{equation}
Under the assumption (A3), a $BT$-action on a $d$-shifted symplectic stack $M$ is Hamiltonian if and only if there exists a map $\mu \colon M \to \t\dual[d-1]$ satisfying the moment map equation \eqref{Eq:MomentMapEquation}.
Hence we may say that the associated de Rham class
\[[\iota \theta] \in \sA^{\DR}(M,\t\dual[d])\]
is the obstruction for the $BT$-action being Hamiltonian.
In \S\ref{Sec:Moduli}, we will provide non-Hamiltonian examples when \Cref{Assumption} is not satisfied, by computing the Hamiltonian obstructions.

\begin{remark}
If we don't even assume (A3) in \Cref{Thm:MomentMapEquation}, then the commutative square \eqref{Square:MomentMapEquation} is not necessarily cartesian.
There are additional higher obstructions in $\sA^{\DR}(M,\Lambda^k\t\dual[d-3k+3])$, for a $BT$-action being Hamiltonian.
\end{remark}

\subsection{Invariance}\label{ss:invariance}

We will show that all $BT$-actions on shifted symplectic stacks are symplectic (\Cref{Prop:BT-symp}).
More generally, all closed forms on derived stacks with $BT$-actions are invariant.

\begin{proposition} 
\label{Prop:Invariance}
Let $M$ be a derived stack with a $BT$-action.
Then the forgetful map
\[\cancel{\inv} :\sA^{p,\cl}_{BT\textnormal{-}\inv}(M,d) \xrightarrow{\simeq} \sA^{p,\cl}(M,d)\]
is an equivalence of spaces for all $p$ and $d$.
\end{proposition}

\Cref{Prop:Invariance} will follow from the following observation.

\begin{lemma}\label{Lem:FullyFaithful}
Let $M$ be a derived stack with a $BT$-action.
Then the pullback functor
\[\QCoh(M/BT) \xrightarrow{} \QCoh(M) \]
is fully faithful.
\end{lemma}

\begin{proof}
Recall that $M/BT \coloneq \varinjlim_{n \in \Delta^{\op}}M_n$, where $M_n \in \dSt_B$ is the image of $n\in \Delta$ under the functor $M \colon \Delta^{\op} \to \dSt$, defining the $BT$-action on $M$.
Equivalently, we have a canonical equivalence
\[M \simeq \varinjlim_{n \in \Delta^{\op}} M_n^{\triv} \times BT \textin \dSt_B^{BT},\]
where $M_n^{\triv} \in \dSt_B^{BT}$ is the derived stack $M_n$ with the trivial $BT$-action.
Since the functor $\QCoh \colon \dSt^{\op} \to \Cat$ preserves limits, it suffices to prove the statement for $M_n^{\triv} \times BT$.
Then we have a canonical equivalence of $\infty$-categories,
\[\QCoh(M_n \times BT) \simeq \prod_{w\in \Z^{ r}}\QCoh(M_n)\]
and the pullback $\pr_1^* \colon\QCoh(M_n) \to \QCoh(M_n \times BT)$ factors through the $w=0$ component, see \cite[Thm~.4.1]{Mou} and \cite[Lem.~5.7]{BP}.
In particular, the pullback functor $\pr_1^*$ is fully faithful.
\end{proof}

\begin{proof}[Proof of \Cref{Prop:Invariance}]
Denote by $\Gamma\DR(X \to Y) \coloneq (Y \to \pt)_* \DR(X \to Y) \in \Fil \big(\QCoh(\pt)\big) $ for any morphism $X \to Y$ of derived stacks.
It suffices to show that the pullback map
\begin{equation}\label{Eq:F2}
\Gamma\DR(M/BT \to B^2T) \to \Gamma\DR(M \to B) \textin  \widehat{\Fil}\big(\QCoh(\pt)\big)
\end{equation}
is an equivalence.
The functor $\Gamma\DR \colon \dSt_B^{\op} \to \widehat{\Fil}\big(\QCoh(\pt)\big)$ preserves limits since $\DR$ is defined as a right Kan extension.
Since the derived stack $M/BT$ is a colimit of derived affine schemes, we may assume that $M/BT$ is a derived affine scheme.
Since $M/BT$ is geometric, $\Gr^p$ of the map \eqref{Eq:F2} can be identified to the pullback map
\[ \Gamma(M/BT,\Lambda^p\LL_{(M/BT)/B^2T}[-p]) \to \Gamma(M,\Lambda^p\LL_{M/B}[-p]) \]
by \eqref{Eq:GrDR}.
This map is an equivalence since composing it with $\Map_{\QCoh(\pt)}(\O_{\pt}[d]-)$ gives
\[\Map_{\QCoh(M/BT)}(\O_{M/BT},\Lambda^p\LL_{(M/BT)/B^2T}[-p-d])\xrightarrow{\simeq}\Map_{\QCoh(M)}(\O_M,\Lambda^p\LL_{M/B}[-p-d])\]
which is an equivalence by \Cref{Lem:FullyFaithful}.
Since \eqref{Eq:F2} is a map of complete filtered complexes,
the associated graded equivalence implies the desired equivalence.
\end{proof}

\begin{corollary}
Let $M$ be a $d$-shifted symplectic stack with a $BT$-action.
Then the space of symplectic structures $\sSymp(BT\circlearrowleft M)$ on the action is contractible.
\end{corollary}

\begin{proof}
This follows from the definition of $\sSymp(BT\circlearrowleft M)$ in \eqref{Eq:Space-sympactions} and \Cref{Prop:BT-symp}.
\end{proof}

\begin{proof}[Proof of \Cref{Prop:BT-symp}]
Let $s\colon B \to B^2T$ be the section.
Denote by $s_* \colon \dSt_{B} \to \dSt_{B^2T}$ the {\em Weil restriction}, i.e., the right adjoint of the pullback $s^* \colon \dSt_{B^2T} \to \dSt_B$.
By \Cref{Prop:Invariance}, the map
\[\sA^{2,\cl}_{B^2T}[d] \xrightarrow{\simeq} s_*(\sA^{2,\cl}_{B}[d]),\]
induced from the canonical map $s^* \sA^{2,\cl}_{B^2T}[d] \to \sA^{2,\cl}_B[d]$, is an equivalence.
Consider the following canonical commutative diagram of $\infty$-categories
\[\xymatrix@-.5pc{
\Symp_{B^2T,d} \ar[r] \ar[d] \ar@{}[rd]|-{(S_1)} & \pSymp_{B^2T,d} \ar[r] \ar[d] \ar@{}[rd]|-{(S_2)} & \Corr_{B^2T} \ar[d] \\
\Symp_B \ar[r] & \pSymp_{B,d} \ar[r] & \Corr_B.
}\]
The left square $(S_1)$ is cartesian since $B \to B^2T$ is smooth surjective and the non-degeneracy of $2$-forms is a smooth local property.
The right square $(S_2)$ is cartesian since the canonical map $\sA^{2,\cl}_{B^2T}[d] \xrightarrow{} s_*(\sA^{2,\cl}_{B}[d])$ is an equivalence, see \Cref{Lem:Corr-Cart}.
Hence the total square $(S_1)+(S_2)$ is also cartesian.
\end{proof}

\subsection{Contraction map}

We will construct the {\em contraction map} \eqref{Eq:ContractionMap} used in \Cref{Thm:MomentMapEquation}.
More generally, we will construct the contraction map as a morphism of de Rham complexes.
Classically, the Cartan magic formula
implies that the de Rham differentials commute with contraction maps for {\em invariant} forms.
The heuristic reason for having the contraction map for a $BT$-action as a morphism of de Rham complexes is that 
all differential forms are invariant by \Cref{Prop:Invariance}.

We first provide a {\em K{\"u}nneth formula} for de Rham complexes.
Denote by $\widehat{\otimes}$ the symmetric monoidal structure on $\widehat{\Fil}\big(\QCoh(B)\big)$.

\begin{lemma}\label{Lem:Kunneth}
Let $N$ be a derived stack, satisfying the following properties:
\begin{itemize}
\item the projection $p_N\colon N \to B$ is locally geometric and locally of finite presentation,
\item $p_{N,*} \colon \QCoh(N) \to \QCoh(B)$ universally preserves colimits and perfect complexes.
\end{itemize}	
Then for any derived stack $M$, there exists a canonical equivalence
\[\DR(M) \widehat{\otimes} \DR(N) \simeq \DR(M\times N) \textin \Fil\big(\QCoh(B)\big).\]
\end{lemma}

\begin{proof}
Since $\widehat{\otimes}$ is the coproduct in $\CAlg\left( \widehat{\Fil}\big(\QCoh(B)\big)\right)$ by \cite[Prop.~3.2.4.7]{LurHA},
we have a canonical map
\[\boxtimes \colon \DR(M) \widehat{\otimes} \DR(N) \to \DR(M\times N) \textin \CAlg\left( \widehat{\Fil}\big(\QCoh(B)\big)\right).\]
Since $\Gr \colon \widehat{\Fil}\big(\QCoh(B)\big) \to \Gr\big(\QCoh(B)\big)$ is conservative \cite[Lem.~2.15]{GP} and monoidal \cite[Prop.~2.26]{GP},
it suffices to show that the induced map 
\[\bigoplus_{a+b=p}\Gr^a\DR(M) \otimes \Gr^b\DR(N) \to \Gr^p\DR(M\times N) \textin \QCoh(B)\]
is an equivalence.
Since $\Gr^{a}\DR(M)=0$ for $a<0$, there are only finitely many non-zero complexes appearing in the left-hand side.
We may assume that the base $B$ is affine
since the de Rham complex $\DR(M \to B)$ is defined as the limit of $\DR(M \times_B A \to A)$ under $\QCoh(B) \simeq \lim_{A \in \dAff_{/B}} \QCoh(A)$.

Then the statement follows from the arguments in \cite[Lem.~A.7]{BKP} (which proved \Cref{Lem:Kunneth} when $B=\pt$ and $N$ is a smooth projective variety).
Here we briefly recall the proof:

We may also assume that $M$ is affine.
Indeed, since $\Gr^p\DR(N) \simeq p_{N,*}(\Lambda^p\LL_{N/B})$ is a perfect complex on $B$,
the functor $-\otimes\Gr^p\DR(N) \colon \QCoh(B) \to \QCoh(B)$ preserves limits.
The functor
$\DR \colon \dSt_B^{\op} \to \Fil\big(\QCoh(B)\big)$ preserves limits since it is defined as a right Kan extension.
The functor $(-)\times_B N \colon \dSt_B\to\dSt_B$ preserves colimits since $\dSt_B$ is an $\infty$-topos.
Hence if we write $M\simeq \varinjlim_{i} M_i$ for derived affine schemes $M_i$, then it suffices to show the desired equivalence for $M_i$.

Since $M$ and $N$ are now locally geometric, it suffices to show that the canonical map
\[\bigoplus_{a+b=p}\Gamma(M,\Lambda^a\LL_M) \otimes \Gamma(N,\Lambda^b\LL_N) \to \Gamma(M\times N, \Lambda^p \LL_{M\times N})\]
is an equivalence.
Since $\Lambda^p \LL_{M\times N} \simeq \bigoplus_{a+b=p} \Lambda^a\LL_M\boxtimes \Lambda^b \LL_N$,
the equivalence follows from the base change formula and the projection formula in \cite[Prop.~3.10]{BZFN} (see also \cite[Thm.~B.8.12]{CHS}).
\end{proof}

In particular,  we can apply \Cref{Lem:Kunneth} to $N=BT$.

\begin{corollary} \label{Cor:KunnethBT}
Let $M$ be a derived stack. 
Then we have a canonical equivalence
\[\DR(M) \widehat{\otimes} \DR(BT) \simeq \DR(M\times BT) \textin \Fil\big(\QCoh(B)\big).\]
\end{corollary}

\begin{proof}
By \Cref{Lem:Kunneth}, it suffices to check that $p_* \colon \QCoh(BT) \to \QCoh(B)$ preserves colimits and perfect complexes,
where $p \colon BT \to B$ is the projection map.
By \cite[Cor.~1.4.5]{DG}, $p_*$ preserves colimits and satisfies the base change formula.
Hence it suffices to check that $p_*$ preserves perfect complexes when $B$ is affine.
By \cite[Prop.~3.9 and Cor.~3.22]{BZFN}, the perfect complexes in $\QCoh(BT)$ and $\QCoh(B)$ are exactly the compact objects.
By \cite[Lem.~5.7]{BP}, $p_*$ preserves compact objects.
\end{proof}

We then compute the de Rham complex of $BT$.
We will use the following notations:
\begin{itemize}
\item $\widehat{(-)} \colon \Fil\big(\QCoh(B)\big) \to \widehat{\Fil}\big(\QCoh(B)\big)$ is the completion functor, which is the left adjoint of the inclusion functor.
\item $\widehat{\oplus}$ is the coproduct in $\widehat{\Fil}\big(\QCoh(B)\big)$, which is the completion of the coproduct in $\Fil\big(\QCoh(B)\big)$.
\item $(-)(d) \colon \Fil\big(\QCoh(B)\big) \to \Fil\big(\QCoh(B)\big)$ is the shift functor, where $\Fil^pE(d) \coloneq \Fil^{p+d}E$.
\item $(-)^{\fil} \colon \QCoh(B) \to \Fil\big(\QCoh(B)\big)$ is the left adjoint of the functor $E \mapsto \Fil^0E$,
which factors through $\widehat{\Fil}\big(\QCoh(B)\big) \subseteq \Fil\big(\QCoh(B)\big)$.
\end{itemize}

\begin{lemma}\label{Lem:Splitting}
There exists a canonical equivalence
\begin{equation*} 
    \DR(BT) \simeq \widehat{\bigoplus}_{p \geq 0} \Sym^p(\t\dual)^{\fil}(-p)[-2p] \textin \Fil\big(\QCoh(B)\big).
\end{equation*}
\end{lemma}

\begin{proof}
A \defterm{splitting} of a filtered complex $E \in \widehat{\Fil}\big(\QCoh(B)\big)$ is a collection $s\coloneq\{s^p\}_{p\in \Z}$ of sections 
\[s^p \colon \Gr^pE \to \Fil^pE \textin \QCoh(B)\]
of the canonical maps $\Fil^pE \to \Gr^pE$.
Any splitting $s$ of $E$ induces a map
\[s_{\#}\colon\widehat{\bigoplus}_{p\in \Z}(\Gr^pE)^{\fil}(-p) \to E \textin \widehat{\Fil}\big(\QCoh(B)\big),\]
 by the universal property of coproducts and the adjunction $(-)^{\fil} \dashv \Fil^0(-)$.
The above map $s_{\#}$ is an equivalence since $\Gr(s_{\#})$ is an equivalence.

Firstly, consider the case when $B=\pt$.
Then the space of splittings of $\DR(BT) \in \widehat{\Fil}\big(\QCoh(B)\big)$,
\[\prod_{p\in \Z} \Map_{\QCoh(B)}(\Gr^p \DR(BT), \Fil^p\DR(BT)),\]
is contractible,
since $\Gr^p\DR(BT)\simeq \Sym^p(\t\dual[-2]) \in \QCoh(B)$ is concentrated in degree $2p$,
and hence $\Fil^p\DR(BT) \simeq \varprojlim_{q \to \infty} \Fil^{p/q}\DR(BT) \in \QCoh(B)$ lies in $\QCoh(B)^{\geq 2p}$.
Hence the unique splitting gives us the desired equivalence.

Now consider a general derived stack $B$.
Then we have a canonical map
\[\widehat{(-)} \circ (B \to \pt)^* \big(\DR(B\GG_m^{\times r} \to \pt)\big) \to \DR(BT \to B) \textin \widehat{\Fil}\big( \QCoh(B)\big).\]
This map is an equivalence since $\Gr \colon \widehat{\Fil}\big(\QCoh(B)\big) \to \Gr\big(\QCoh(B)\big)$ is conservative and the pushforward $(BT \to B)_* \colon \QCoh(BT) \to \QCoh(B)$ has the base change formula by \cite[Prop.~3.10]{BZFN}.
Hence the desired equivalence for the case $B=\pt$ implies the desired equivalence for general $B$.
\end{proof}

We can now construct the contraction map as follows:

\begin{construction}\label{Def:Contraction}
Let $M$ be a derived stack with a $BT$-action.
The \defterm{contraction map} 
\[\iota \colon \DR(M) \longrightarrow \DR(M) \otimes  \t\dual(-1)[-2]	 \]
is the composition
\begin{multline*}
\DR(M) \xrightarrow{\sigma^*} \DR(M \times BT) \xleftarrow[\mathrm{Cor.}\ref{Cor:KunnethBT}]{\simeq} \DR(M) \widehat{\otimes }\DR(BT)  \\
\xleftarrow[\mathrm{Lem.}\ref{Lem:Splitting}]{\simeq} \widehat{\bigoplus}_{p\in \Z} \DR(M) \otimes \Sym^p(\t\dual)(-p)[-2p] \xrightarrow{\pr_1} \DR(M) \otimes  \t\dual(-1)[-2]	
\end{multline*}
where $\sigma \colon M \times BT \to M$ is the action map.
\end{construction}

In particular, we have a contraction map between spaces of differential forms,
\begin{equation*} 
\sA^{p,\cl}_{BT\textnormal{-}\inv}(M,d) \xrightarrow[\simeq]{\cancel{\inv}}\sA^{p,\cl}(M,d) \xrightarrow{\iota} \sA^{p-1,\cl}(M,\t\dual[d-1])
\end{equation*}
used in \eqref{Eq:ContractionMap} and \Cref{Thm:MomentMapEquation}.

\subsection{Equivariance}

We show that $BT$-actions on shifted symplectic stacks are Hamiltonian under \Cref{Assumption}.
In particular, we prove  \Cref{Thm:Ham=Symp} and \Cref{Thm:MomentMapEquation}.
Recall from \Cref{Prop:HamilviaEquiv} that being Hamiltonian is equivalent to the $BT$-equivariance
of the associated de Rham form.

We provide obstructions to $BT$-equivariance of de Rham forms via the contraction maps.

\begin{proposition} 
\label{Prop:Equiv}
Assume one of the following conditions:
\begin{enumerate}
\item [(A3)] $d\leq 2$ and $B$ is classical Artin stack, locally of finite type, or
\item [(A4)] $r=1$ (and hence $T=\GG_{m,B}$).
\end{enumerate}
Let $M$ be a derived stack with a $BT$-action.
Then there exists a canonical fiber square of spaces
\[\xymatrix{
\sA^{\DR}_{BT\textnormal{-}\equiv} (M,d+2) \ar[r] \ar[d]^-{\cancel{\equiv}} \cart  & \star \ar[d]^-{0} \\
\sA^{\DR}(M,d+2) \ar[r]^-{\iota} & \sA^{\DR}(M,\t\dual[d]) .
}\]
\end{proposition}

\begin{proof} 
Note that there exists a canonical null-homotopic sequence
\begin{equation}\label{Eq:2}
\DR(M/BT) \xrightarrow{{\cancel{\equiv}}} \DR(M) \xrightarrow{\iota} \DR(M)\otimes \t\dual(-1)[-2]\end{equation}
induced from the coequalizing diagram $M \times BT \rightrightarrows M \to M/BT$ associated to the $BT$-action.

Assume (A4), hence $T\simeq \GG_{m,B}$.
We claim that the above sequence \eqref{Eq:2} is a cofiber sequence.
Indeed, as in the proof of \Cref{Lem:FullyFaithful}, since $M \simeq \colim_{n \in \Delta^{\op}} M_n^{\triv} \times BT \in \dSt^{BT}_B$, we may assume that $M \simeq N \times BT$, for a derived stack $N$ with the trivial $BT$-action.
Then by the K\"{u}nneth formula (Lemma \ref{Lem:Kunneth}), we may assume that $N \simeq B$.
By Lemma \ref{Lem:Splitting}, we may identify the contraction map $\iota \colon \DR(BT) \to \DR(BT)\otimes \t\dual(-1)[-2]$ with the composition
{\footnotesize \[
\bigoplus_{p\in \Z}\left(\Sym^p(\t\dual) \xrightarrow{\Sym^p(\id,\id)} \Sym^p(\t\dual \oplus \t\dual) \simeq \bigoplus_{a+b=p} \Sym^{a}(\t\dual)\otimes \Sym^b(\t\dual)
\xrightarrow{\pr_1}\Sym^{p-1}(\t\dual)\otimes \Sym^1(\t\dual) \right)(-p)[-2p],\] }

\noindent
where $\Delta\colon \t\dual \to \t\dual \oplus \t\dual$ is the diagonal map.
Since $\t \simeq \O_B$, the $p\neq 0$ component of the above map is an equivalence.
Hence the fiber of the above map is $\O_B^{\fil} \simeq \DR(B \to B)$, as desired.

Assume (A3).
We will use the induction on $r$.
We may write $T\simeq T_1 \times T_2$, where $T_2 \simeq \GG_{m,B}$ is the $1$-dimensional torus.
From the induced $BT_2$-action on $M/BT_1$, 
we have a coequalizing diagram $M \times BT_1 \rightrightarrows M \to M/BT_1$ in $\dSt_B^{BT_2}$, and we can form a morphism of null-homotopic sequences
\[\xymatrix{
\sA^{\DR} (M/BT_1,d+2) \ar[d]^{\iota_{BT_2}} \ar[r] & \sA^{\DR} (M,d+2) \ar[r]^-{\iota_{BT_1}} \ar[d]^{\iota_{BT_2}} & \sA^{\DR} (M, \t_1\dual[d]) \ar[d]^{\iota_{BT_2}}\\
\sA^{\DR} (M/BT_1,\t_2\dual[d]) \ar[r] & \sA^{\DR} (M,\t_2\dual[d] ) \ar[r]^-{\iota_{BT_1}} & \sA^{\DR} (M,\t_1\dual \otimes \t_2\dual [d-2]).
}\]
By the induction hypothesis, the two rows are fiber sequences.
Since $d \leq2$ and $B$ is classical, the space $\sA^{\DR} (M,\t_1\dual \otimes \t_1\dual[ d-2])$ is discrete by \cite[Lem.~6.1.2(1)]{Park2},
and the above commutative diagram and \Cref{Prop:Equiv} under assumption (A4) for the $BT_2$-action on $M/BT_1$ give us
\begin{align*}
\sA^{\DR} (M/BT,d+2) &\simeq \fib\left(\sA^{\DR} (M/BT_1,d+2) \xrightarrow{\iota_{BT_2}} \sA^{\DR} (M/BT_1,\t_2\dual[d])\right)	\\
&\simeq \fib\left(  \sA^{\DR} (M,d+2) \xrightarrow{(\iota_{BT_1},\iota_{BT_2})} \sA^{\DR} (M,\t_1\dual[d] ) \times \sA^{\DR} (M,\t_2\dual[d])  \right) \\
&\simeq \fib\left(  \sA^{\DR} (M,d+2) \xrightarrow{\iota_{BT}} \sA^{\DR} (M,\t\dual[d])\right),
\end{align*}
as desired.
\end{proof}

\begin{corollary} \label{Cor:preHamilMomentMapEquation}
Let $M$ be a derived stack with a $BT$-action.
Then there exists a canonical commutative square of spaces
\[
\xymatrix{
\Map_{\dSt^{BT}_B}(M,\t\dual[d-1])^{\isot}\ar[r] \ar[d]   & \sA^0(M,\t\dual[d-1]) \ar[d]^{d} \\
\sA^{2,\cl}(M,d) \ar[r]^-{\iota } & \sA^{1,\cl}(M,\t\dual[d-1]).
}\]
Moreover, when the assumption (A3) of \Cref{Thm:MomentMapEquation} is satisfied, the above square is cartesian.
\end{corollary}

\begin{proof}
Consider the canonical commutative diagram of spaces
\[\xymatrix@C+2pc{
\Map_{\dSt^{BT}_B}(M,\t\dual[d-1])^{\isot} \ar[r] \ar[d] \ar@{}[rd]|-{(S_1)} & \sA^{\DR}_{BT\textnormal{-}\equiv}(M,d+2) \ar[r] \ar[d]^-{\equiv\to\inv}    \ar@{}[rdd]|-{(S_2)}& \star \ar[dd]^0 \\
\sA^{2,\cl}_{BT\textnormal{-}\inv}(M,d) \ar[r]^-{[-]} \ar[d]^{\cancel{\inv}}_{\simeq}  \ar@{}[rd]|-{(S_3)}& \sA^{\DR}_{BT\textnormal{-}\inv}(M,d+2) \ar[d]^{\cancel{\inv}}_{\simeq} &  \\
\sA^{2,\cl}(M,d) \ar[r]^-{[-]} & \sA^{\DR}(M,d+2)  \ar[r]^-{\iota} & \sA^{\DR}(M,\t\dual[d]) ,
}\]
where $(S_1)$ is the cartesian square induced by \Cref{Prop:HamilviaEquiv} (more precisely, the equivalence \eqref{Eq:IsotMapSpace}),
$(S_2)$ is the commutative square in \Cref{Prop:Equiv},
and the canonical commutative square $(S_3)$ is cartesian since $\cancel{\inv}$ is an equivalence by \Cref{Prop:Invariance}.

The total square $(S_1)+(S_2)+(S_3)$ has the following alternative factorization
\[\xymatrix{
\Map_{\dSt^{BT}_B}(M,\t\dual[d-1])^{\isot}\ar@{.>}[r]^-{} \ar[d] \ar@{}[rd]|-{(S_4)}  & \sA^{0}(M,\t\dual[d-1]) \ar[r] \ar[d]^-{d}  \ar@{}[rd]|-{(S_5)}&    \star \ar[d]^0 \\
\sA^{2,\cl}(M,d)  \ar[r]^-{\iota} & \sA^{1,\cl}(M,\t\dual[d-1]) \ar[r]^-{[-]} & \sA^{\DR}(M,\t\dual[d]),
}\]
since the contraction map $\iota$ commutes with $[-]$ and the right square $(S_5)$ is cartesian.

Now assume (A3) in \Cref{Thm:MomentMapEquation}.
Since the square $(S_2)$ is cartesian by \Cref{Prop:Equiv},
the total square $(S_1)+(S_2)+(S_3)\simeq(S_4)+(S_5)$ is cartesian, 
and hence the square $(S_4)$ is also cartesian.
\end{proof}

\begin{remark}
In \Cref{Cor:preHamilMomentMapEquation},
the top horizontal arrow \[\Map_{\dSt^{BT}_B}(M,\t\dual[d-1])^{\isot} \to\sA^0(M,\t\dual[d-1]) \]  is equivalent to the canonical forgetful map.
However, the precise comparison of the two maps requires a lengthy diagram chasing argument.
Since this comparison is not used in the main results of this paper, we omit the proof.
\end{remark}

\begin{proof}[Proof of \Cref{Thm:MomentMapEquation}]
We have a canonical commutative diagram of spaces
\[
\xymatrix{
\sHamil(BT \circlearrowleft M) \ar[r] \ar[d] \cart & \Map_{\dSt^{BT}_B}(M,\t\dual[d-1])^{\isot}\ar[r] \ar[d]   & \sA^0(M,\t\dual[d-1]) \ar[d]^{d} \\
\star \ar[r]^-{\theta} & \sA^{2,\cl}(M,d) \ar[r]^-{\iota } & \sA^{1,\cl}(M,\t\dual[d-1]),
}\]
where the left cartesian square is given by the definition of the space $\sHamil(BT \circlearrowleft M)$ in \eqref{Eq:Space-HamActions},
and the right commutative square is the one in \Cref{Cor:preHamilMomentMapEquation}.
If we assume (A3) in \Cref{Thm:MomentMapEquation}, the right square is also cartesian by \Cref{Cor:preHamilMomentMapEquation}.
\end{proof}

To prove \Cref{Thm:Ham=Symp}, we recall the following well-known fact: negatively shifted closed $1$-forms are exact.
For the completeness, we provide a proof here.

\begin{lemma}\label{Lem:Exactness}
Let $M$ be a derived stack over $B$.
Assume that $M$ is locally geometric and locally of finite presentation over $B$.
Under \Cref{Assumption}, the de Rham differential
\[ d \colon \sA^0(M,d-1) \to \sA^{1,\cl}(M,d-1) \]
is an equivalence of spaces.
\end{lemma}

\begin{proof}
Consider the canonical cartesian square of spaces
\[\xymatrix@-.5pc{
\sA^0(M,d-1) \ar[r]^-{d} \ar[d] \cart & \sA^{1,\cl}(M,d-1)  \ar[d]^-{[-]} \\
\star \ar[r]^-{0} & \sA^{\DR}(M,d).
}\]

\noindent
(A2) If $d \leq -1$ and $B$ is classical, then the space $\sA^{\DR}(M,d)$ is contractible by \cite[Lem.~6.1.2]{Park2}, and the statement follows.

\smallskip

\noindent
(A1) If $d = 0$ and $B = \pt$, then the space $\sA^{\DR}(M,d)$ is discrete by \cite[Lem.~6.1.2]{Park2}.
Hence it suffices to show that the map $[-] \colon \sA^{1,\cl}(M,-1) \to \sA^{\DR}(M,0)$ is zero.
Since the map \[ \sA^{\DR}(M,0) \lra \prod_{m\colon \pt \to M} \sA^{\DR}(\pt,0),\]
induced by the pullbacks $m^*$, is injective by \cite[Lem.~6.1.2]{Park2}, we may assume that $M=\pt$.
Since $\LL_{\pt} \simeq 0$, the map $[-]$ is zero.
\end{proof}

\begin{proof}[Proof of \Cref{Thm:Ham=Symp}]
Let $\dSt^{\lfp}_B \subseteq \dSt_B$ be the full subcategory consisting of derived stacks that are locally geometric and locally of finite presentation over $B$.
Form a cartesian diagram
\[\xymatrix@C+2pc{
\dSt_{(\t\dual[d-1]/BT)^{\isot}}^{\lfp} \ar[r]^-{\isotInt^{\fib,\lfp}} \ar[d] \cart & \dSt_{\sA^{2,\cl}_{B^2T}[d]}^{\lfp}\ar[d] \ar[r] \cart & \dSt_B^{\lfp} \ar[d] \\
\dSt_{(\t\dual[d-1]/BT)^{\isot}} \ar[r]^-{\isotInt^{\fib}} & \dSt_{\sA^{2,\cl}_{B^2T}[d]} \ar[r] & \dSt_B.
}\]
Then the map $\isotInt^{\fib,\lfp}$ is an equivalence of $\infty$-categories by \Cref{Cor:preHamilMomentMapEquation} and \Cref{Lem:Exactness}.

Consider the canonical commutative diagram of $\infty$-categories,
\[\xymatrix@C+3pc{
\Lag_{\t\dual[d-1]/BT,d+1} \ar[r]^-{} \ar[d]_{\LagInt^{\fib}} \ar@{}[rd]|-{(A)} &  \Corr\left(\dSt_{(\t\dual[d-1]/BT)^{\isot}}^{\lfp}\right)\ar[d]_{\simeq}^-{\Corr(\isotInt^{\fib,\lfp})} \ar[r] \ar@{}[rd]|-{(B)} & \Isot_{\t\dual[d-1]/BT,d+1}\ar[d]^-{\IsotInt^{\fib}} \\
\Symp_{B^2T,d} \ar[r] & \Corr\left(\dSt_{\sA^{2,\cl}_{B^2T}[d]}^{\lfp}\right) \ar[r] & \mathrm{p}\Symp_{B^2T,d}
}\]
The right square $(B)$ is cartesian since the functor $\Corr \colon \Cat^{\fp} \to \Cat$ (in \Cref{Def:Corr}) preserves limits.
Since the total square $(A)+(B)$ is cartesian by \Cref{Lem:LagFib-Non-degeneracy},
the left square $(A)$ is also cartesian.
Since the middle vertical arrow $\Corr(\isotInt^{\fib,\lfp})$ is an equivalence, so is the left vertical arrow $\LagInt^{\fib}$, which is $\forget_{\Hamil} \colon \Hamil^{BT}_{B,d} \to \Symp^{BT}_{B,d}$ by definition (\Cref{Notation:Functors-Forget/HamRed/TrivHam}).
\end{proof}

\section{Moduli of perfect complexes}\label{Sec:Moduli}

In this section, we apply our general results on symplectic rigidifications to derived moduli stacks of perfect complexes on Calabi-Yau varieties.

Throughout this section, we let $T\coloneq \GG_{m,B} \in \Grp_B$ be the $1$-dimensional torus over a base $B$.

\subsection{Results}

We first summarize the results that we will prove in this section.

Let $p\colon X \to B$ be a smooth projective morphism of classical schemes, of relative dimension $n$.
A \defterm{Calabi-Yau structure} on $p\colon X \to B$ is an isomorphism of line bundles
\[\omega \colon \O_X \xrightarrow{\simeq} K_{X/B},\]
where $K_{X/B} \coloneq \Lambda^n\Omega_{X/B}$ is the canonical line bundle.
Denote by
\[\uPerf(X/B) \in \dSt_B\]
the derived moduli stack of perfect complexes on the fibers of $p \colon X \to B$.
Similarly, we let $\uPic(X/B)\in \dSt_B$ be the derived moduli stack of line bundles on the fibers of $p\colon X \to B$.

The derived stack $\uPerf(X/B)$ has the following additional structures:
\begin{itemize}
\item $\uPerf(X/B)$ has a canonical $(2-n)$-shifted symplectic structure
over $B$, when $p\colon X \to B$ has a Calabi-Yau structure, by \cite{PTVV};
\item $\uPerf(X/B)$ has a canonical action of $BT$, given by the constant automorphisms of perfect complexes, see \S\ref{ss:BT-action} for the precise construction.
\end{itemize}
By \Cref{Prop:BT-symp}, the $BT$-action on $\uPerf(X/B)$ is always symplectic (and such symplectic structure on the action is unique).
Our main result is about the Hamiltonian structures on the $BT$-action.

\begin{theorem}\label{Thm:Perf-Ham}
Let $p\colon X \to B$ be a smooth projective Calabi-Yau morphism of relative dimension $n$ between classical schemes of finite type.
Then the canonical $BT$-action on $\uPerf(X/B)$ is Hamiltonian if and only if either
\begin{enumerate}
\item $n \geq 3$, or
\item $n=2$ and the projection map $\uPic(X/B)_\cl \to B$ is flat.
\end{enumerate}
\end{theorem}

By \Cref{Thm:Ham=Symp}, we already know that the $BT$-action on $\uPerf(X/B)$ is Hamiltonian when $n \geq 3$.
Hence to prove \Cref{Thm:Perf-Ham}, it suffices to show that
\begin{itemize}
\item [(a)] the $BT$-actions on $\uPerf(E)$ for an elliptic curve $E$ and $\uPerf$ are not Hamiltonian;
\item [(b)] the $BT$-action on $\uPerf(S/B)$ for a family of Calabi-Yau surfaces $S \to B$ is Hamiltonian if and only if $\uPic(S/B)_\cl \to B$ is flat.
\end{itemize}
Both statements follow from the computations of the Hamiltonian obstructions in \Cref{Thm:MomentMapEquation}.

Roughly speaking, the projection $\uPic(S/B)_\cl \to B$ is flat if and only if the {\em Noether-Lefschetz divisors} are trivial (see \Cref{Cor:k3family}).
In particular, the $BT$-action on $\uPerf(S/B)$ is Hamiltonian for a single Calabi-Yau surface (i.e. when $B=\pt$), but there are non-Hamiltonian examples when we consider families (i.e. when $B\neq\pt$).
We provide one example here.

\begin{corollary} 
Let $|\O_{\P^3}(4)|$ be the complete linear system of quartic surfaces in the projective space $\P^3$.
Let $B\coloneq |\O_{\P^3}(4)|^{\sm} \subseteq |\O_{\P^3}(4)|$ be the open subscheme consisting of smooth quartic surfaces.
Let $S \to B$ be the universal family.
Then the canonical $BT$-action on $\uPerf(S/B)$ is not Hamiltonian.
\end{corollary}

\begin{proof}
    By \Cref{Thm:Ham=Symp}, it suffices to show that $\uPic(S/B)_\cl \to B$ is not flat.
    We prove this by contradiction.
    Assume that $\uPic(S/B)_\cl \to B$ is flat.
    Consider the open and closed subscheme
    \[\uPic(S/B,0)\coloneq\left\{ \big(b\in B, L\in \Pic(S_b)\big) \colon c_1(L) \in \ker\big(H^2(S_b,\Z) \xrightarrow{i_{b,*}} H^4(\P^3,\Z)\big) \right\} \subseteq \uPic(S/B), \]
    of degree $0$ line bundles $L$ on the fibers $S_b$ of $S \to B$.
    Here $c_1(L)\in H^2(S_b,\Z)$ is the first Chern class of $L$ in the singular cohomology and $i_b\colon S_b \hookrightarrow \P^3$ is the inclusion.

    For a very general $g \in B$, the fiber $\uPic(S/B,0)_{\cl} \times_B \{g\}$ is connected:
    indeed, every degree $0$ line bundle is trivial on $S_g$ by the Noether-Lefschetz theorem (see e.g. \cite[\S3.3.2]{Voi}).
    Since $\uPic(S/B,0)_{\cl} \to B$ is an open morphism (since it is flat) and $B$ is irreducible, it follows that $\uPic(S/B,0)_{\cl}$ is connected.

    Since degrees and self-intersection numbers of line bundles are locally constant, we therefore have $L^2 = 0$ for every degree $0$ line bundle on the fibers $S_b$.
    However, this is false:
    for example, take the Fermat quartic surface $S_b = \{ X_0^4+X_1^4+X_2^4+X_3^4 = 0 \}$ and consider the lines $l$, $l'$ defined by $X_0=\lambda X_1$ and $X_2=\lambda X_3$ for two different values of $\lambda$ with $\lambda^4=-1$;
    then $L = \O_{S_b}(l-l')$ has degree $0$ and $L^2=-4 <0$ by the adjunction formula.
    This leads to contradiction.
\end{proof}

We will also provide a rigidified version of the extension correspondence.
Let $\uExact(X/B)$ be the derived stack of exact triangles $E_1 \to E_2 \to E_3$ of perfect complexes on the fibers of $X \to B$
(see \S\ref{ss:ExtensionCorrespondence} for the precise definition).
When $X \to B$ has a Calabi-Yau structure,
\[\xymatrix@R-.5pc{
& \uExact(X/B) \ar[ld]_{(\ev_1,\ev_3)} \ar[rd]^{\ev_2} \\
\uPerf(X/B)\times_B \uPerf(X/B) && \uPerf(X/B)
} \]   
has a canonical Lagrangian correspondence structure, by \cite[Cor.~6.5]{BD}.

\begin{proposition}\label{Prop:RidigifiedExtCorr}
Let $X \to B$ be a smooth projective Calabi-Yau morphism of dimension $n$.
If $n\geq 3$ (or $n=2$ and $B=\pt$), there exists a canonical $(2-n)$-shifted Lagrangian correspondence
\[\xymatrix@R-.5pc{
& \uExact(X/B)^{\rig} \ar[ld]_{} \ar[rd]^{} \\
\uPerf(X/B)/\!/BT\times_B \uPerf(X/B)/\!/BT && \uPerf(X/B)/\!/BT.
}\]
\end{proposition}

\Cref{Prop:RidigifiedExtCorr} will follow from the functoriality of symplectic rigidification (\Cref{Def:SympRig}) together with the functoriality of Hamiltonian reductions along change of groups (\Cref{Prop:Ind/Res-Functoriality}).

\subsection{$BT$-actions}\label{ss:BT-action}

We first review the canonical $BT$-actions on moduli of perfect complexes.
We will use the monoid structures on moduli of perfect complexes to explain the $BT$-actions.

The $\infty$-category of \defterm{monoid stacks} is the full subcategory
\[\Mon \subseteq \Fun(\Delta^{\op},\dSt)\]
consisting of functors $M \colon \Delta^{\op} \to \dSt$ such that $M_0\simeq \pt$ and the map $M_n \to M_1^{\times n}$, induced by the maps $\{i-1,i\} \hookrightarrow [n]$ in $\Delta$ for $1\leq i \leq n$, is an equivalence.
These monoid stacks are the monoid objects in $\dSt$ in the sense of \cite[Def.~4.1.2.5]{LurHA}.
We collect some basic facts on the monoid stacks:
\begin{itemize}
\item The monoid stacks are the associative algebra objects in $\dSt$, with the cartesian symmetric monoidal structure, i.e. we have a canonical equivalence of $\infty$-categories
\[\Mon \simeq \Alg(\dSt)\]
by \cite[Prop.~4.1.2.10, Prop.~2.4.2.5]{LurHA}.
\item The group stacks are monoid stacks, i.e., we have
\[\Grp \subseteq \Mon\] 
as full subcategories of $\Fun(\Delta^{\op},\dSt)$, by \cite[Ex.~4.1.2.7]{LurHA}.
\end{itemize}
\begin{lemma}\label{Lem:BT-action-1}
There exists a canonical equivalence of $\infty$-categories
\[\Mon \simeq \Fun^{\rmR} \big(\dSt^{\op}, \Alg(\Grpd)\big),\]
where $\Fun^{\rmR}(-,-)\subseteq \Fun(-,-)$ denotes the full subcategory of right adjoint functors,
and $\Grpd$ is given the cartesian symmetric monoidal structure.
\end{lemma}

\begin{proof}
We have canonical equivalences of $\infty$-categories
{\small \[\Mon \simeq \Alg(\dSt) \simeq \Alg\big(\Fun^{\et}(\CAlg_{\C}^{\leq0},\Grpd)\big) \simeq \Fun^{\et}\big(\CAlg_{\C}^{\leq0},\Alg(\Grpd)\big) \simeq \Fun^{\rmR} \big(\dSt^{\op}, \Alg(\Grpd)\big),\]}

\noindent
by \cite[Rem.~2.1.3.4]{LurHA} and \cite[Lem.~B.3.12]{CHS}, where $\Fun^{\et}(\CAlg^{\leq0}_{\C},-)\subseteq \Fun(\CAlg^{\leq0}_{\C},-)$ denotes the full subcategory of functors satisfying {\'e}tale descent.
\end{proof}

The notion of group actions extends to monoids.
Let $M$ be a monoid stack. The $\infty$-category of \defterm{derived stacks with $M$-actions} is the full subcategory
\[\dSt^M \subseteq \Fun(\Delta^{\op},\dSt)_{/M}\]
consisting of natural transformations $X \to M$ of functors $X,M \colon\Delta^{\op} \to \dSt$ such that for any $[n]\in \Delta$, the map $X_n \to X_0 \times M_n$, induced by the map $[0] \to [n] \colon 0 \mapsto n$ in $\Delta$, is an equivalence, see \cite[Def.~4.2.2.2]{LurHA}.
We also collect basic facts on monoid actions:
\begin{itemize}
    \item[\namedlabel{itm:second}{}] For a monoid stack $M$, there exists a canonical $M$-action on the underlying derived stack $M_1$ of $M$, by \cite[Ex.~4.2.2.4]{LurHA}.
    \item[\namedlabel{itm:third}{}] For a morphism $M \to N$ of monoid stacks and a $N$-action on a derived stack $X_0$ given by $X \to   N$, there exists an induced $M$-action on $X_0$, given by $X\times_N M \to M$.
\end{itemize}

Recall that the functor of \defterm{quasi-coherent sheaves}
\[\QCoh \colon \dSt^{\op} \longrightarrow \CAlg(\Cat)\]
is the right Kan extension of the functor of modules
$\CAlg_{\C}^{\leq0} \subseteq\CAlg_{\C} \longrightarrow \CAlg(\Cat) \colon A \mapsto \Mod_A.$
\begin{itemize}
\item The $\infty$-category of \defterm{perfect complexes} on $A\in \CAlg^{\leq0}_{\C}$ is the smallest full subcategory \[\Perf(A) \subseteq \Mod_A\]
containing $A\in \Mod_A$ and is stable under fibers, cofibers and retracts.
\item The $\infty$-category of \defterm{line bundles} on $A\in \CAlg^{\leq0}_{\C}$ is the full subcategory \[\Pic(A) \subseteq \Perf(A)\]
consisting of perfect complexes whose restrictions to $H^0(A)$ is a line bundle (in the classical sense).
\end{itemize}
Since the full subcategories $\Perf(A),\Pic(A) \subseteq \Mod_A$ are stable under tensor products and extensions along $A \to B$, we have induced functors
\[\CAlg_{\C}^{\leq0}\longrightarrow \CAlg(\Cat) \colon A \mapsto \Perf(A) \and \CAlg_{\C}^{\leq0}\longrightarrow \CAlg(\Cat) \colon A \mapsto \Pic(A).\]
By taking the right Kan extension, we obtain the functors of perfect complexes and line bundles
\[\Perf \colon \dSt^{\op} \longrightarrow \CAlg(\Cat) \and \Pic \colon \dSt^{\op} \longrightarrow \CAlg(\Cat).\]
Consider the following functor induced by \Cref{Lem:BT-action-1}:
\begin{equation}\label{Eq:BT-action-1}
\Fun^{\rmR}\big(\dSt^{\op}, \CAlg(\Cat)\big) \to \Fun^{\rmR}\big(\dSt^{\op}, \CAlg(\Grpd)\big) \to \Fun^{\rmR}\big(\dSt^{\op}, \Alg(\Grpd)\big) \simeq \Mon. 
\end{equation}

\begin{definition}\
\begin{enumerate}
\item The derived moduli stack of \defterm{perfect complexes} is the monoid stack
\[\uPerf \in \Mon,\]
defined as the image of the functor $\Perf \colon \dSt^{\op} \to \CAlg(\Cat)$ under the functor \eqref{Eq:BT-action-1}.
\item The \defterm{Picard stack} is the monoid stack
\[\uPic \in \Mon,\]
defined as the image of the functor $\Pic \colon \dSt^{\op} \to \CAlg(\Cat)$ under the functor \eqref{Eq:BT-action-1}.
The Picard stack $\uPic$ is a group stack since all line bundle are invertible by \cite[Prop.~2.9.4.2]{LurSAG}.
\end{enumerate}
\end{definition}

Let $B$ be a derived stack. The functor of \defterm{mapping stacks}
\[\uMap_B \colon \dSt_B^{\op} \times \dSt_B \longrightarrow \dSt_B\]
is the functor of internal hom, i.e., for $X,Y \in \dSt_B$, the mapping stack $\uMap_B(X,Y)\in \dSt_B$ is uniquely characterized by the universal property:
\begin{equation} \label{Eq:MapUniv}
    \Map_{\dSt_B}(S, \uMap_B(X,Y)) \simeq \Map_{\dSt_B} (S \times_B X, Y).
\end{equation}
We note that the mapping stacks are compatible with group stacks and their actions:
\begin{itemize}
\item [\namedlabel{itm:3}{}] For a derived stack $X$ over $B$ and a group stack $G$ over $B$, the mapping stack $\uMap_B(X,G)$ is a group stack. More precisely, the functor $\uMap_B$ lifts to a functor
\[\uMap_B^{\grp} \colon \dSt_B^{\op} \times \Grp_B \longrightarrow \Grp_B.\]
\item [\namedlabel{itm:4}{}]  For a derived stack $X$ and an action of a group stack $G$ on a derived stack $Y$, we have an induced action of $\uMap_B(X,G)$ on $\uMap_B(X,Y)$.
More precisely, let $\dSt_B^{\act} \subseteq \Fun(\Delta^{1} \times \Delta^{\op},\dSt_B)$ be the full subcategory consisting of group stack actions on derived stacks, i.e., morphisms $M \to G$ of functors $\Delta^{\op} \to \dSt_B$ such that $G$ is a group stack and $M \in \dSt^G_B$.
Then there exists a functor
\[\uMap_B^{\mathrm{act}} \colon \dSt_B^{\op} \times \dSt_B^{\act} \longrightarrow \dSt_B^{\act}\]
which lifts $\uMap_B$ and $\uMap_B^{\grp}$.
\end{itemize}

\begin{definition}
Let $p\colon X \to B$ be a morphism of derived stacks.
\begin{enumerate}
\item The derived stack of \defterm{perfect complexes} on the fibers of $p\colon X \to B$ is the mapping stack
\[\uPerf(X/B)\coloneq \uMap_B(X,\uPerf \times B) \in \dSt_B.\]
\item The derived stack of \defterm{line bundles} on the fibers of $p\colon X \to B$ is the mapping stack
\[\uPic(X/B)\coloneq \uMap_B(X,\uPic \times B) \in \dSt_B.\]
\end{enumerate}    
\end{definition}

\begin{construction} \label{Construction:BGmAction}
Let $p\colon X \to B$ be a morphism of derived stacks.
The canonical $BT$-action on the derived stack $\uPerf(X/B)$ is defined through the following steps:
\begin{enumerate}
\item 
The monoid stack $\uPerf$ has a canonical action on its underlying derived stack $\uPerf$ by \eqref{itm:second}.
\item 
The morphism $\uPic \to \uPerf$ of monoid stacks induces an action of $\uPic$ on $\uPerf$ by \eqref{itm:third}.
\item 
Applying $\uMap_B(X, (-)\times B)$, we get an induced action of $\uPic(X/B)$ on $\uPerf(X/B)$ by \eqref{itm:4}.
\item 
The group homomorphism $\uPic(B/B) \to \uPic(X/B)$ from \eqref{itm:3} induces an action of $\uPic\times B \simeq \uPic(B/B)$ on $\uPerf(X/B)$ by \eqref{itm:third}.
\item 
The equivalence $B\GG_m\simeq \uPic$ of group stacks in \Cref{Lem:Pic=BGm} below induces an action of $BT$ on $\uPerf(X/B)$, where $T\coloneq \GG_m\times B$.
\end{enumerate}
\end{construction}

We need the following lemma to complete \Cref{Construction:BGmAction}.

\begin{lemma}\label{Lem:Pic=BGm}
There exists a canonical equivalence of group stacks,
\[\uPic \simeq B\GG_m.\]
\end{lemma}

It is well-known that $\uPic\simeq B\GG_m$ as plain derived stacks.
Here we provide a proof of \Cref{Lem:Pic=BGm} since we haven't found a proof of comparing the group structures in the literature.

\begin{proof}[Proof of \Cref{Lem:Pic=BGm}]
Let $\St$ be the $\infty$-category of classical stacks.
There exists a fully faithful functor $\St \hookrightarrow \dSt$ with a right adjoint $(-)_{\cl} \colon \dSt \to \St$.
Since the base $\C$ is a field, the inclusion $\St \hookrightarrow \dSt$ preserves products.
Hence we have a functor $\Grp(\St) \to \Grp(\dSt)$ between group objects.

We first observe that both the derived stacks $\uPic$ and $B\GG_m$ are classical.
Indeed, $\uPic$ is classical by \cite[2.9.6.2]{LurSAG} and $B\GG_m$ is classical since the inclusion $\St \hookrightarrow \dSt$ preserves colimits.
Hence it suffices to show that $\uPic \simeq B\GG_m$ as classical group stacks, i.e. objects in $\Grp(\St)$.

Recall from \cite[\S6.2.3]{LurHTT} that we have a canonical equivalence of $\infty$-categories
\begin{equation}\label{Eq:BG1}
B\colon \Grp(\St) \simeq \St_* \colon \Cech,
\end{equation}
where $ \St_*$ is the full subcategory of the undercategory $\St_{\pt/}$, consisting of effective epimorphisms $\pt \to M$ of classical stacks. 
Here the functor $B\colon \Grp(\St) \to \St_*$ sends a group stack $G$ to the canonical morphism $\pt \to BG$ for the classifying stack $BG$
and the functor $\Cech \colon \St_* \to \Grp(\St)$ sends a morphism $\pt \to M$ to the Cech nerve $\Cech(\pt \to M)$.

By considering the group objects in the $\infty$-categories in \eqref{Eq:BG1}, we obtain canonical equivalences 
\begin{equation}\label{Eq:BG2}
\Grp(\Grp(\St)) \simeq \Grp(\St_*) \simeq \Grp(\St)_*
\end{equation}
where $\Grp(\St)_*$ is the full subcategory of the undercategory $\Grp(\St)_{\pt/}$ consisting of group homomorphisms $\pt \to G$ that are effective epimorphisms.
Here the second equivalence follows from the definition of group objects.

We note that both $\uPic$ and $B\GG_m$ can be considered as objects in $\Grp(\St)_*$.
Indeed, since the group structure on $B\GG_m$ is defined via the commutative group structure on $\GG_m$ and the equivalence \eqref{Eq:BG2}, we have a canonical group homomorphism $\pt \to B\GG_m$, which is an effective epimorphism.
On the other hand, the morphism $\pt \to \uPic$ given by the trivial line bundle $\C$ on $\pt$ is the unit of the group structure of $\uPic$, and hence it can be regarded as a group homomorphism.
Since we already know $\uPic\simeq B\GG_m$ as classical stacks, without the group structures (see e.g. \cite[Ex.~14.4.8]{LM}), and there is a unique map $\pt \to \uPic\simeq B\GG_m$, the map $\pt \to \uPic$ is an effective epimorphism.

We will show that $\uPic\simeq B\GG_m$ as objects in $\Grp(\St)_*$.
By the equivalence \eqref{Eq:BG2}, it suffices to show that the Cech nerves are equivalent,
\[\Cech(\pt \to \uPic) \simeq \Cech (\pt \to B\GG_m) \textin \Grp(\Grp(\St)).\]
Note that $\Cech (\pt \to B\GG_m) \in \Grp(\Grp(\Sch)) \subseteq \Grp(\Grp(\St))$, where $\Sch \subseteq \St$ is the full subcategory of classical schemes.
Since we know that $\uPic \simeq B\GG_m$ as objects in $\St_*$, we also have $\Cech (\pt \to \uPic) \in \Grp(\Grp(\Sch))$.
Since $\Sch$ is an $1$-category, the forgetful functor
\[\Grp(\Grp(\Sch)) \to \Grp(\Sch)\]
is fully faithful.
Hence it suffices to show that
\[\Cech(\pt \to \uPic) \simeq \Cech (\pt \to B\GG_m) \textin \Grp(\Sch).\]
By the equivalence \eqref{Eq:BG1}, this follows from the equivalence $\uPic \simeq B\GG_m$ in $\St_{*}$, which we already know.
\end{proof}

\subsection{Non-Hamiltonian actions}

We compute the Hamiltonian obstructions (\Cref{Thm:MomentMapEquation}) for moduli of perfect complexes and prove our main result (\Cref{Thm:Perf-Ham}).

We first review the canonical shifted symplectic structures on the moduli of perfect complexes.
Let $p\colon X \to B$ be a smooth projective Calabi-Yau morphism of dimension $n$ between classical schemes of finite type.
\begin{itemize}
\item For any derived stack $M$ over $B$,
we have the \defterm{integration map}
\[\int_{X/B} \colon \sA^{p,\cl}(X\times_BM\to B,d) \lra \sA^{p,\cl}(M \to B,d-n),\]
induced from the following composition of morphisms of filtered complexes on $B$,
\begin{multline*}
\quad\qquad\DR(X\times_BM \to B) \xleftarrow[\boxtimes]{\simeq} \DR(X \to B) \widehat{\otimes}\DR(M \to B) \\ 
\rightarrow p_*\DR(X \to X) \widehat{\otimes} \DR(M \to B) \simeq p_*\O_X \otimes \DR(M\to B) \xrightarrow{\omega} \DR(M\to B)[-n],
\end{multline*}
where the first left arrow is an isomorphism by \Cref{Lem:Kunneth}
and the last arrow is given by the counit map $p_* \O_X \to \O_B[-n]$ of the adjunction $p_* \dashv p^*[n]$.
\item We have the {\em Chern character} \cite{TV2} 
\[\ch_p \colon \uPerf \lra  \sA^{p,\cl}[p] \textin \dSt.\]
\end{itemize}
The $(2-n)$-shifted symplectic structure on the moduli stack $\uPerf(X/B) \coloneq \uMap_B(X,\uPerf\times B)$ is
\[\int_{X/B} \ch_2(\cE_{X/B}) \in \sA^{2,\cl}(\uPerf(X/B) \to B,2-n),\]
where $\cE_{X/B}$ is the universal perfect complex on $X\times_B \uPerf(X/B)$.

Our key result for proving \Cref{Thm:Perf-Ham} is the following.
We continue letting $T : = \GG_{m,B} \in \Grp_B$.

\begin{proposition}\label{Prop:4.13}
The canonical $BT$-action on $\uPerf(X/B)$ is Hamiltonian if and only if
\[\int_{X/B}\ch_1(\cE_{X/B}) \in \sA^{1,\cl}(\uPerf(X/B)\to B, 1-n)\]
is exact.    
More generally,
for any open substack $M \subseteq \uPerf(X/B)$, the induced $BT$-action on $M$ is Hamiltonian if and only if $\int_{X/B} \ch_1(\cE_{X/B}|_{X\times_B M})$ is exact.
\end{proposition}

\Cref{Prop:4.13} will follow from \Cref{Thm:MomentMapEquation} together with the following lemma.
Recall from \Cref{Def:Contraction} the contraction map
\[\iota \colon \sA^{p,\cl}(\uPerf,p) \longrightarrow \sA^{p-1,\cl}(\uPerf,p-1),\]
with respect to the canonical $B\GG_m$-action on $\uPerf$ (in \S\ref{ss:BT-action}).
Here we fixed a trivialization
\[\C \xrightarrow{\simeq} \t\dual \simeq \Gamma(B\GG_m,\LL_{B\GG_m}[1]),\]
of the dual Lie algebra $\t\dual$ of $\GG_m$,
via the underlying $1$-shifted $1$-form $ \ch_1(\cL)^{\cancel{\cl}} \in \sA^1(B\GG_m,1)$, 
where $\cL$ is the universal line bundle on $B\GG_m\simeq \uPic$.
Let $\cE$ be the universal perfect complex on $\uPerf$.

\begin{lemma}\label{Lem:Contraction-Chern}
There exists a canonical equivalence
\[\iota \ch_p(\cE) \simeq \ch_{p-1}(\cE) \textin \sA^{p-1,\cl}(\uPerf,p-1).\]
\end{lemma}

\begin{proof}
Let $\sigma \colon \uPerf \times B\GG_m \to \uPerf$ be the action map.
Then we have canonical equivalences
\[\sigma^*\ch_p(\cE) \simeq \ch_p(\cE \boxtimes \cL) \simeq \sum_{a+b=p}\ch_a(\cE) \boxtimes \ch_b(\cL) \textin \sA^{p,\cl}( \uPerf \times B\GG_m ,p),\]
by the multiplicativity of the Chern characters (see \cite[Thm.~6.5]{HSS}).
Then the statement follows from the definition of contraction map in \Cref{Def:Contraction} and the trivialization $\C\simeq \t\dual$ that we chose.
\end{proof}

\begin{proof}[Proof of \Cref{Prop:4.13}]
By the construction of $BT$-action on the mapping stack $\uPerf(X/B) \coloneq \uMap_B(X,\uPerf \times B)$ in \eqref{itm:3}, the diagram
\[\xymatrix{
X\times_B\uPerf(X/B) \ar[r]^-{\ev} \ar[d]^{\pr_2} & \uPerf \times B \\ \uPerf(X/B)
}\]
is $BT$-equivariant.
Since the contraction $\iota$ and the integration $\int_{X/B}$ are defined for the de Rham complexes and $\iota\ch_2(\cE) \simeq \ch_1(\cE)$ by \Cref{Lem:Contraction-Chern}, we have canonical equivalences
\[ \iota \int_{X/B} \ev^*(\ch_2(\cE))  \simeq  \int_{X/B} \iota  \ev^*(\ch_2(\cE)) \simeq  \int_{X/B} \ev^*(\iota\ch_2(\cE)) \simeq  \int_{X/B} \ev^*(\ch_1(\cE)) \]
of $(1-n)$-shifted closed $1$-forms on $\uPerf(X/B) \to B$.
Since $\cE_{X/B}\simeq \ev^*\cE$, we have
\[\iota \int_{X/B} \ch_2(\cE_{X/B}) \simeq \int_{X/B}\ch_1(\cE_{X/B}).\]
Then the statement follows from \Cref{Thm:MomentMapEquation}.
\end{proof}

\begin{notation}
Denote by
\begin{itemize}
\item $\H_{\DR}^k(X/B)\coloneq \pi_0 \sA^{\DR}(X \to B,k)$,
\item $\Fil_{\Hd}^p \H_{\DR}^k(X/B)\coloneq \pi_0 \sA^{p,\cl}(X \to B, k-p)$.
\end{itemize}
Then the canonical map $\Fil_{\Hd}^p \H_{\DR}^k(X/B) \hookrightarrow \H_{\DR}^k(X/B)$ is injective, since the Hodge to de Rham spectral sequence degenerates, by \cite{Del1}.    
\end{notation}

\begin{corollary}\label{Cor:non-Ham-Perf}
The canonical $B\GG_m$-action on $\uPerf$ is not Hamiltonian.
\end{corollary}

\begin{proof}
It suffices to show that the canonical $B\GG_m$-action on the open substack $\uPic \subseteq \uPerf$
is not Hamiltonian.
Since the de Rham cohomology of $\uPic \simeq B\GG_m$ is the polynomial ring generated by $[\ch_1(\cL)] \in \H^2_{\DR}(\uPic)$ (\Cref{Lem:Splitting}), we have $[\ch_1(\cL)] \neq 0$ and hence $\ch_1(\cL)\in \sA^{1,\cl}(\uPic,1)$ is not exact.
Hence the statement follows from \Cref{Prop:4.13}.
\end{proof}

\begin{corollary}\label{Cor:non-Ham-elliticcurve}
Let $E$ be a smooth projective curve of genus $1$.
Then the canonical $B\GG_m$-action on $\uPerf(E)$ is not Hamiltonian.
\end{corollary}

\begin{proof}
    Let $\Delta \subseteq E \times E$ be the diagonal divisor and consider its associated line bundle on $E \times E$,
    \[ L : = \O_{E \times E}(\Delta). \]
    Let $\cE$ be the universal perfect complex on $E \times \uPerf(E)$.
    By the universal property of mapping stacks \eqref{Eq:MapUniv}, there exists a morphism $f \colon E \to \uPerf(E)$ such that $L \simeq (\id_E \times f)^* \cE$ on $E \times E$.
    Therefore,
    \[ \int_E [\ch_1(L)] \simeq \int_E (\id_E \times f)^*[\ch_1(\cE)] \simeq f^*\int_E[\ch_1(\cE)] \textin \sA^{\DR}(E,1), \]
    since the integration map $\int_E$ forms a functor $\DR(E \times (-)) \to \DR(-)[-1]$ by construction.
    By \Cref{Prop:4.13}, it suffices to show that $\int_E [\ch_1(\cE)] \neq 0$, which follows from the non-vanishing
    \[ \int_E [\ch_1(L)] \neq 0 \textin \H_{\DR}^1(E). \]

    Let $\omega \in \H^0(E,\Omega_E^1)$ be the Calabi-Yau form, which can be seen as an element of $\H^1_{\DR}(E)$ since $\Gr^{\geq 2}\DR(E) = 0$ (by \eqref{Eq:GrDR}).
    Let $p \colon E \to \pt$ be the projection.
    The map $\DR(E) \to \C[-1]$ used to define the integration is the composition of the cupping with $\omega$ and the top degree projection map $\DR(E) \to \C(-1)[-2]$.
    Indeed, we implicitly used $\omega$ to identify the right adjoint $p^*(-) \otimes \Omega_E^1[1]$ of $p_*$ to $p^*[1]$.
    For smooth varieties, the de Rham cohomology $\H^*_{\DR}$ and the singular cohomology $\H^*$ with $\C$-coefficients are isomorphic by \cite{Gro}.
    Since the top degree projection is equivalent to pushforward in singular cohomology, we therefore have that
    \[ \int_E \simeq (\pr_2)_*((-) \cup \pr_1^*\omega) \colon \H^*(E \times E) \to \H^*(E) \]
    where $\pr_i \colon E \times E \to E$ are the projections.
    
    The Chern character in de Rham cohomology coincides with the classical one on smooth varieties by \cite[Thm.~5.4]{TV2}.
    In particular, the first Chern class of a line bundle associated to a divisor is equal to the cycle class of the divisor.
    Hence,
    \[ [\ch_1(L)] = [\Delta]  \textin \H^2(E \times E). \]
    
    The compositions $\pr_i \circ \delta = \id_E $ and $[\Delta] = \delta_*1$ where $\delta \colon E \to E \times E$ is the diagonal map.
    Therefore, we have
    \[ \int_E[\ch_1(\cL)] =(\pr_2)_*([\Delta] \cup \pr_1^*\omega) = (\pr_2)_*(\delta_*1 \cup \pr_1^*\omega) = (\pr_2)_*\delta_*\delta^*\pr_1^*\omega = \omega \textin \H_{\DR}^1(E) \]
    where we used the projection formula in the second identity.
    This implies the desired non-vanishing $\int_E[\ch_1(\cL)] \neq 0$, and completes the proof.
\end{proof}

\begin{definition}
Assume that the base $B$ is a local Artinian scheme. Let $X_0\coloneq X\times_B \{0\}$ be the fiber of $p\colon X \to B$ over the unique closed point $0 \in B$.
\begin{enumerate}
\item The space of \defterm{horizontal de Rham classes} is
\[\H_{\DR}^k(X/B)^{\nabla}\coloneq\mathrm{image}\left(\H_{\DR}^k(X) \to \H_{\DR}^k(X/B)\right) 
.\]
\item The map of \defterm{horizontal lifts}
\[\widetilde{(-)} \colon \H_{\DR}^k(X_0) \xrightarrow{\simeq} \H^k_{\DR}(X/B)^{\nabla}\]
is the inverse of the composition
\[\H^k_{\DR}(X/B)^{\nabla} \hookrightarrow \H^k_{\DR}(X/B) \xrightarrow{0^*} \H^k_{\DR}(X_0), \]
which is an isomorphism:
we have $\Fil^0\DR(X/B) \simeq \Fil^0\DR(X \times B/B) \simeq \Fil^0\DR(X) \otimes \Gamma(B,\O_B)$ by K\"{u}nneth formula (\Cref{Lem:Kunneth}) and $\Fil^0\DR(X) \simeq \Fil^0\DR(X_0)$ by nil-invariance of $\Fil^0\DR$ \cite[pp.~62]{Park2}.
\end{enumerate}
\end{definition}

\begin{corollary}\label{Cor:k3family}
Assume that $n=2$.
Then the canonical $BT$-action on $\uPerf(X/B)$ is Hamiltonian if and only if the following condition is satisfied:
\begin{itemize}
\item[\namedlabel{itm:5}{}] For any local Artinian ring $A$, a pointed morphism of schemes $(A,0) \to (B,b)$, and a perfect complex $E\in \uPerf(X_b)$, we have
\[\widetilde{[\ch_1(E)]} \in \Fil_{\Hd}^1 \H_{\DR}^2(X_A/A) ,\]
where $[\ch_1(E)] \in \H^2_{\DR}(X_b)$ is the image of $\ch_1(E)$ under the composition $\sA^{1,\cl}(X_b,1) \to \sA^{\DR}(X_b,2) \to \pi_0\sA^{\DR}(X_b,2)$ 
and $X_A\coloneq X\times_BA$ is the pullback of $X \to B$ along $A \to B$.
\end{itemize}
\end{corollary}

\begin{proof}
This follows from \Cref{Prop:4.13} together with \cite[Prop.~6.2.1]{Park2}.
\end{proof}

\begin{proof}[Proof of \Cref{Thm:Perf-Ham}]
If $n=0$ (resp. $n=1$), the statement follows from \Cref{Cor:non-Ham-Perf} (resp. \Cref{Cor:non-Ham-elliticcurve}).
If $n\geq3$, the statement follows from \Cref{Thm:Ham=Symp}.
It remains to show that when $n=2$, the property \eqref{itm:5} is equivalent to the flatness of $\Pic(X/B)_\cl \to B$.

Assume \eqref{itm:5}.
Following \cite[\S{5}]{BuFl}, call a coherent sheaf $E$ on a smooth projective variety $Y$ {\em $1$-semiregular} if the trace map $\Ext^2_Y(E,E) \to \H^2(Y,\O_Y)$ is injective.
Every line bundle is $1$-semiregular since the trace map is an isomorphism.
By \cite[Prop.~5.9]{BuFl}, flat deformation of $1$-semiregular sheaves over local Artinian rings satisfying \eqref{itm:5} extends to infinitesmial thickenings of the base.
Therefore, the morphism $\Pic(X/B)_{\cl} \to B$ satisfies the infinitesimal lifting property for local Artinian rings.
Therefore, it is smooth by \cite[Lem.~37.12.2]{Sta}.

Conversely, assume that the morphism $\Pic(X/B)_{\cl} \to B$ is flat.
Then, since its fibers are smooth, it is smooth by \cite[Lem.~29.34.3]{Sta}.
Hence, it satisfies the infinitesimal lifting property for local Artinian rings by \cite[Lem.~37.12.2]{Sta}.
Therefore, in the situation of \eqref{itm:5}, we can find a line bundle $L_A$ extending $L : = \det(E)$.
Also, we have $[\ch_1(E)] = [\ch_1(L)]$ by the splitting principle, since the Chern character in de Rham cohomology is equal to the classical one on smooth varieties by \cite[Thm.~5.4]{TV2}.
Since $[\ch_1(L_A)] \in \Fil_{\Hd}^1\H^2_{\DR}(X_A)$, we conclude that $\widetilde{[\ch_1(E)]} = \widetilde{[\ch_1(L)]} \in \Fil_{\Hd}^1 \H_{\DR}^2(X_A/A)$ as desired.
This completes the proof.
\end{proof}

\subsection{Extension correspondences} \label{ss:ExtensionCorrespondence}

We provide the rigidified version of extension correspondences (\Cref{Prop:RidigifiedExtCorr}).

Consider the functor of commutative square
\[\Perf^{\Box} \colon \dSt^{\op} \xrightarrow{\Perf}\CAlg(\Cat) \xrightarrow{\Fun(\Delta^1\times \Delta^1,-)} \CAlg(\Cat).\]
Here the second functor is induced by the limit-preserving functor $\Fun(\Delta^1\times\Delta^1,-)\colon \Cat \to \Cat$, which is symmetric monoidal with the cartesian symmetric monoidal structure on $\Cat$.
More explicitly, for a derived stack $M$, the symmetric monoidal category $\Perf^{\Box}(M) \in \CAlg(\Cat)$ consists of commutative squares of perfect complexes on $M$,
\begin{equation}\label{Eq:Ext-1}
\xymatrix@-1pc{
E_1 \ar[r] \ar[d] & E_2 \ar[d] \\ E_4 \ar[r] & E_3.
}\end{equation}
Let $\uPerf^{\Box}$ be the monoid stack of commutative squares, induced by the limit-preserving functor 
\[\dSt^{\op} \xrightarrow{\Perf^{\Box}} \CAlg(\Cat) \xrightarrow{(-)^{\simeq}}\CAlg(\Grpd) \to \Alg(\Grpd),\]
under \Cref{Lem:BT-action-1}.

We will now define the functor of exact triangles,
\[\Exact \colon \dSt^{\op} \lra \Cat.\]
For a derived stack $M$, we define $\Exact(M) \subseteq \Perf^{\Box}(M)$ as the full subcategory consisting of commutative squares \eqref{Eq:Ext-1}, which are cartesian and $E_4 \simeq 0$.
Since $\Exact(M)\subseteq \Perf^{\Box}(M)$ is stable under pullbacks along morphisms in $\dSt$, we have a well-defined functor $\Exact$ induced from $\Perf^{\Box}$.
(For the precise construction, we consider $\Exact$ as the straightening of the subcategory of the unstraightening of $\Perf^{\Box}$, which is fiberwise $\Exact(M)$.)

\begin{definition}
The derived stack of \defterm{exact triangles}
\[\uExact \in \dSt\]
is the image of the limit-preserving functor $\dSt^{\op} \xrightarrow{\Exact} \Cat \xrightarrow{} \Grpd$ under $\dSt \simeq \Fun^{\rmR}(\dSt^{\op},\Grpd)$.

\end{definition}

Consider the canonical morphisms of monoid stacks
\[\uPic \to \uPerf \xrightarrow{\diag} \uPerf^{\Box} \xrightarrow{\ev_i} \uPerf,\]
where $\diag$ sends a perfect complex $E$ to the commutative square \eqref{Eq:Ext-1} with $E_i=E$, and
$\ev_i$ sends a commutative square \eqref{Eq:Ext-1} to $E_i$.
Since $\uPic\simeq B\GG_m$ as monoid stacks (\Cref{Lem:Pic=BGm}), 
\begin{itemize}
\item we have an induced $B\GG_m$-action on $\uPerf^{\Box}$, and
\item the evaluation maps $\ev_i \colon \uPerf^{\Box} \to \uPerf$ lift to $B\GG_m$-equivariant maps.
\end{itemize}

Note that $\uExact$ is closed under the $B\GG_m$-action on $\uPerf^{\Box}$.
More precisely, for a line bundle $L$ on a derived stack $M$ and a commutative square \eqref{Eq:Ext-1} of perfect complexes on $M$, which is cartesian and $E_4\simeq 0$, the induced commutative square
\[\xymatrix@-1pc{
E_1 \otimes L \ar[r] \ar[d] & E_2 \otimes L \ar[d] \\ E_4 \otimes L \ar[r] & E_3\otimes L
}\]
is also cartesian and $E_4\otimes L \simeq 0$.
Hence we have an induced $B\GG_m$-action on $\uExact$ such that $\uExact \to \uPerf^{\Box}$ is $B\GG_m$-equivariant.
In particular, we have a $B\GG_m$-equivariant correspondence
\begin{equation}\label{Eq:Ext-2}
\xymatrix@-1pc{
& \uExact \ar[ld]_{(\ev_1,\ev_3)} \ar[rd]^-{\ev_2} & \\ \uPerf \times \uPerf &&\uPerf
}\end{equation}
of derived stacks.

\begin{definition}
Let $p\colon X \to B$ be a morphism of derived stacks.
The derived stack of \defterm{exact triangles on the fibers of $p\colon X \to B$} is the mapping stack
\[\uExact(X/B)\coloneq \uMap_B(X,\uExact \times B) \in \dSt_B.\]
\end{definition}

Applying the functor $\uMap_B(X,(-)\times B)$ to the $B\GG_m$-equivariant correspondence \eqref{Eq:Ext-2},
we obtain a $BT$-equivariant correspondence
\begin{equation}\label{Eq:Ext-3}
\xymatrix@-1pc{
& \uExact(X/B) \ar[ld]_{(\ev_1,\ev_3)} \ar[rd]^-{\ev_2} & \\ \uPerf(X/B) \times \uPerf(X/B) &&\uPerf(X/B),
}\end{equation}
of derived stacks over $B$,
where we denoted $T\coloneq \GG_{m}\times B \in \Grp_B$.

Recall \cite[Cor.~6.5]{BD} that the correspondence \eqref{Eq:Ext-2} (without the $B\GG_m$-actions) has a canonical $2$-shifted Lagrangian correspondence structure.
By \cite[Thm.~2.10]{Cal} the correspondence \eqref{Eq:Ext-3} (without the $BT$-actions) has a canonical $(2-n)$-shifted Lagrangian correspondence structure when $X \to B$ has a Calabi-Yau structure of dimension $n$.

\begin{construction}\label{Const:RidigifiedExtCorr}
Let $p\colon X \to B$ be a smooth projective Calabi-Yau morphism of relative dimension $n$ between classical schemes of finite type.
Assume that $n\geq 3$ (or $n=2$ and $B=\pt$).
The \defterm{rigidified extension correspondence}
\[\xymatrix{
& \uExact(X/B)^{\rig} \ar[ld]_{(\ev_1,\ev_3)} \ar[rd]^-{\ev_2} & \\ \uPerf(X/B)/\!/BT \times \uPerf(X/B)/\!/BT &&\uPerf(X/B)/\!/BT
}\]
is the $(2-n)$-shifted Lagrangian correspondence, defined through the following steps:
\begin{enumerate}
\item [(S1)] By \Cref{Prop:BT-symp}, the Lagrangian correspondence \eqref{Eq:Ext-3} whose underlying correspondence has a $BT$-action, uniquely extends to a $BT$-equivariant Lagrangian correspondence;
we have
\[\uExact(X/B) \colon \uPerf(X/B)\times \uPerf(X/B) \dashrightarrow \uPerf(X/B) \textin \Symp^{BT}_{B,2-n}.\]
\item [(S2)] Applying the symplectic rigidification functor $\rig_{\Symp} \colon \Symp^{BT}_{B,2-n} \to \Symp_{B,2-n}$ in \Cref{Def:SympRig}, we obtain a Lagrangian correspondence,
\[\uExact(X/B)/\!/BT \colon \Big(\uPerf(X/B)\times \uPerf(X/B)\Big)/\!/BT \dashrightarrow \uPerf(X/B)/\!/BT\]
in $\Symp_{B,2-n}$,
where the $BT$-action on $\uPerf(X/B)\times \uPerf(X/B)$ is the diagonal action.
\item [(S3)] By the functoriality of Hamiltonian reductions in \Cref{Prop:Ind/Res-Functoriality} (see also \eqref{Eq:thirdisomorphismthm}), the short exact sequence $BT \xrightarrow{\Delta} BT \times BT \to BT$ of smooth group stacks gives us a canonical equivalence of $(2-n)$-shifted symplectic stacks,
\[\Big(\uPerf(X/B)\times \uPerf(X/B)\Big) /\!/(BT\times BT) \simeq \Bigg(\Big(\uPerf(X/B)\times \uPerf(X/B)\Big)/\!/BT \Bigg)/\!/BT.\]
Since the Hamiltonian reductions commute with products, we also have a canonical equivalence of $(2-n)$-shifted symplectic stacks,
\[ \qquad \big(\uPerf(X/B)/\!/BT \times \uPerf(X/B)/\!/BT \big) \simeq \Big(\uPerf(X/B)\times \uPerf(X/B)\Big) /\!/(BT\times BT).\]
\item [(S4)] Consider the Lagrangian correspondence of Hamiltonian reduction in \eqref{Eq:HamRed-LagCor} (or \eqref{Eq:HamRed-LagCor-funtorial}),
\begin{equation}\label{Eq:Ext-4}
\Bigg(\Big(\uPerf(X/B)\times \uPerf(X/B)\Big)/\!/BT \Bigg)/\!/BT \dashrightarrow \Big(\uPerf(X/B)\times \uPerf(X/B)\Big)/\!/BT .
\end{equation}

\item [(S5)] We define $\uExact(X/B)^{\rig}$ as the following composition of Lagrangian correspondences,
\begin{multline*}
\qquad \big(\uPerf(X/B)/\!/BT \times \uPerf(X/B)/\!/BT \big) \simeq \Bigg(\Big(\uPerf(X/B)\times \uPerf(X/B)\Big)/\!/BT \Bigg)/\!/BT \\
\xdashrightarrow[]{\eqref{Eq:Ext-4}} \Big(\uPerf(X/B)\times \uPerf(X/B)\Big)/\!/BT
\xdashrightarrow[]{\uExact(X/B)/\!/BT}\uPerf(X/B)/\!/BT.
\end{multline*}
\end{enumerate}
\end{construction}

\begin{proof}[Proof of \Cref{Prop:RidigifiedExtCorr}]
This follows from \Cref{Const:RidigifiedExtCorr} above.
\end{proof}

\appendix
\section{Correspondence categories}\label{Appendix:Corr}

In this section, we provide some technical results on correspondence categories that we use in this paper.
More precisely, we will provide the following two results:
\begin{enumerate}
\item (\Cref{Const:Corr-Funct}) The assignment $B\mapsto\Corr(\dSt_B)$ upgrades to a functor \[\Corr(\dSt) \lra \Cat.\]
\item (\Cref{Cor:Corr-Adj}) A morphism $B \xleftarrow{} A \xrightarrow{} C$ in $\Corr(\dSt)$ gives rise to an adjunction \[\xymatrix{
\Corr(\dSt_B)  \ar@<.4ex>[r] & \ar@<.4ex>[l] \Corr(\dSt_C). }\]
\end{enumerate}

\subsection{Review}

We first review the definition of the correspondence categories 
and their basic properties.
The original reference is \cite{Bar}.
We mostly follow the excellent presentation in \cite{HHLN}.

Let $\Lambda^2_2 \subseteq \Delta^2$ be the wide subcategory consisting of the morphisms $0 \to 2$ and $1 \to 2$.

\begin{definition}
The $\infty$-category of \defterm{adequate triples} is the subcategory
\[\Trip \subseteq \Fun(\Lambda^2_2,\Cat) \]
consisting of the following data:
\begin{itemize}
\item [(Obj)] An object is an $\infty$-category $\scC$ together with two wide subcategories $\scC_L,\scC_R \subseteq \scC$ satisfying the following property:
for a morphism $f\colon c \to d$ in $ \scC_R$ and a morphism $g\colon b \to d$ in $\scC_L$, there exists a cartesian square
\[\xymatrix@-.5pc{
a \ar[r]^-{\tf} \ar[d]_-{\tg}  & b \ar[d]^-{g} \\ c \ar[r]^-{f} & d
}\]
in $\scC$, such that $\tf$ lies in $\scC_R$ and $\tg$ lies in $\scC_L$.
\item [(Mor)] A morphism from $(\scC, \scC_L, \scC_R)$ to $(\scD,\scD_L,\scD_R)$ is a functor $F \colon \scC \to \scD$ such that $F(\scC_L) \subseteq \scD_L$, $F(\scC_R) \subseteq \scD_R$, and $F$ preserves pullbacks of morphisms in $\scC_R$ along morphisms in $\scC_L$.
\end{itemize}    
\end{definition}

Let $\Cat^{\fp} \subseteq \Cat$ be the subcategory consisting of $\infty$-categories with fiber products and $\infty$-functors preserving fiber products. 
Then we have a canonical $\infty$-functor
\begin{equation}\label{Eq:Corr-4}
\Cat^{\fp} \lra \Trip \colon \scC \mapsto (\scC, \scC,\scC),
\end{equation}
which preserves limits by \cite[Lem.~2.4]{HHLN}.    

\begin{example} 
By \cite[Ex.~2.9]{HHLN}, we have a functor of \defterm{twisted arrow categories}
\[\Tw \colon \Cat \lra \Trip.\]
For an $\infty$-category $\scC \in \Cat$, the twisted arrow category $\Tw(\scC) \in\Trip$ consists of the following data:
\begin{itemize}
\item [(D1)] The $\infty$-category $\Tw(\scC)$ is the unstraightening of the functor $\Map_{\scC} \colon \scC^{\op} \times \scC \to \Grpd$.
Hence
\begin{itemize}
\item an object in $\Tw(\scC)$ is a morphism $r\colon c \to b$ in $\scC$ and
\item a morphism $f$ from $r_1 \colon c_1 \to b_1$ to $r_2 \colon c_2 \to b_2$ is a commutative diagram
\[\xymatrix@-.5pc{
c_1 \ar[r]^-{f_c} \ar[d]_-{r_1}  & c_2 \ar[d]^-{r_2} \\ b_1  & b_2 \ar[l]_-{f_b}
}\]
in $\scC$, for some morphisms $f_c,f_d$.

\end{itemize}
\item [(D2)] The wide subcategory $\Tw(\scC)_L \subseteq \Tw(\scC)$ (resp. $\Tw(\scC)_R \subseteq \Tw(\scC)$) consists of morphisms $f\colon (c_1 \xrightarrow{r_1} b_1) \to (c_2 \xrightarrow{r_2} b_2)$ such that $f_c \colon c_1 \to c_2$ (resp. $f_b \colon b_2 \to b_1$) is an equivalence.
\end{itemize}
\end{example}

\begin{definition}\label{Def:Corr}
The functor of \defterm{correspondence categories}
\[\Corr \colon \Trip \lra \Cat\]
is the right adjoint of the functor $\Tw \colon \Cat \to \Trip$,
which exists by \cite[Thm.~2.18]{HHLN}.
\end{definition}

For an adequate triple $\scC \in \Trip$, the associated correspondence category $\Corr(\scC)$ can be described as follows:
Since we have a canonical equivalence of groupoids
\[\Map_{\Cat}(\Delta^n, \Corr(\scC)) \simeq \Map_{\Trip}(\Tw(\Delta^n),\scC),\]
an object in $\Corr(\scC)$ is an object in $\scC$, 
and a morphism from $c_1$ to $c_2$ in $\Corr(\scC)$ is a correspondence
\[\xymatrix@-1pc{& a\ar[ld]_s \ar[rd]^t & \\ c_1 && c_2}\]
in $\scC$, such that $s\colon a \to c_1$ lies in $\scC_L$ and $t \colon a \to c_2$ lies in $\scC_R$.

The following theorem of \cite[Thm.~12.2]{Bar} (more precisely, the correction in \cite[Thm.~3.1]{HHLN}) is extremely useful for studying the correspondence categories.

\begin{theorem}\label{Thm:Corr-cocart}
Let $p \colon \scC \to \scB$ be a morphism of adequate triples.
Assume the following:
\begin{enumerate}
\item [(A1)] The functor $p_L \colon \scC_L \to \scB_L$ is a cartesian fibration.
\item [(A2)] The functors $p_R \colon \scC_R \to \scB_R$ and $p|_{\scB_R} \colon \scC\times_{\scB} \scB_R \to \scB_R$ are cocartesian fibrations.
\item [(A3)] The inclusion $\scC_R \hookrightarrow \scC\times_{\scB}\scB_R$ preserves cocartesian morphisms over $\scB_R$.
\item [(A4)] Consider a commutative square
\[\xymatrix{
c_1 \ar[r]^-{c_{12}} \ar[d]_{c_{13}} \ar@{}[rd]|-{\sigma} & c_2 \ar[d]^{c_{24}} \\
c_3 \ar[r]^-{c_{34}} & c_4
}\]
in $\scC$,
such that 
$c_{12} \colon c_1 \to c_2$ lies in $\scC_R$, 
$c_{13} \colon c_1 \to c_3$ lies in $\scC_L$,
$c_{34} \colon c_3 \to c_4$ is a cocartesian morphism with respect to $p_R\colon \scC_R \to \scB_R$,
$p(c_{24}) \colon p(c_2) \to p(c_4)$ lies in $\scB_L$,
and $p(\sigma)$ is a cartesian square in $\scB$.
Then 
\begin{multline*}
\qquad\qquad\text{$\sigma$ is cartesian square in $\scC$ and $c_{24}\colon c_2 \to c_4$ lies in $\scC_L$ } \\
\iff \text{$c_{12}\colon c_1 \to c_2$ is a cocartesian morphism with respect to $p \colon \scC \to \scB$}.    
\end{multline*}
\end{enumerate}
Then the induced functor $\Corr(p) \colon \Corr(\scC) \to \Corr(\scB)$ is a cocartesian fibration.
\end{theorem}

\begin{example}\label{Ex:Unfurling}
Let $F\colon \scB \to \Cat$ be an $\infty$-functor.
Assume the following conditions:
\begin{enumerate}
\item [(A1)] For a morphism $f\colon b_1 \to b_2$ in $\scB$, the functor $f_! \colon F(b_1) \to F(b_2)$ has a right adjoint $f^*$.
\item [(A2)] The category $\scB$ has fiber products.
\item [(A3)] For a cartesian square
\[\xymatrix@-.5pc{
b_1 \ar[r]^-{\tf} \ar[d]_{\tg} \cart & b_2 \ar[d]^{g} \\
b_3 \ar[r]^-{f} & b_4
}\]
in $\scB$, the following Beck-Chevalley transformation is an equivalence,
\[\tf_! \circ \tg^* \xrightarrow{\simeq} g^* \circ f_! \textin \Fun (F(b_3) ,F(c_2)).\]
\end{enumerate}
The \defterm{unfurling} of the functor $F\colon \scB \to \Cat$ is a functor $\tF \colon \Corr(\scB) \lra \Cat$ 
\begin{itemize}
\item sending an object $b \in \Corr(\scB)$ to the category $F(b) \in \Cat$, 
\item sending a morphism $b_1 \xleftarrow{s} a \xrightarrow{t} b_2$ in $\Corr(\scB)$ to the functor $b_{2,!} \circ b_1^* \colon F(b_1) \to F(b_2)$,
\end{itemize}
defined as follows:
Let $p\colon \scC \to \scB$ be the cocartesian fibration associated to the functor $F \colon \scB \to \Cat$.
Let $\scC_{p\textnormal{-}\mathrm{cart}} \subseteq \scC$ be the wide subcategory of $p$-cartesian morphisms.
Then 
$(\scC,\scC_{p\textnormal{-}\mathrm{cart}}, \scC) \to (\scB,\scB,\scB)$ is a morphism in $\Trip$ by \cite[Prop.~2.6]{HHLN}.
We define $\tF \colon \Corr(\scB) \to \Cat$ as the straightening of
\[\widetilde{\scC}\coloneq\Corr(\scC,\scC_{p\textnormal{-}\mathrm{cart}}, \scC) \lra \Corr(\scB),\]
which is a cocartesian fibration by \Cref{Thm:Corr-cocart}.
\end{example}

\subsection{Adjunction} 

We provide canonical adjunctions between correspondence categories.

\begin{proposition}\label{Prop:Corr-Adj}
Let $F \colon \scC \to \scD$ be an $\infty$-functor in $\Cat^{\fp}$. Assume the followings:
\begin{enumerate}
\item [(A1)] There exists a right adjoint $F^{\rmR} \colon  \scD \to \scC$ of $F\colon \scC \to \scD$.
\item [(A2)] The functor $F\colon \scC \to \scD$ is conservative.
\item [(A3)] For any objects $c\in \scC$, $d\in \scD$, and a morphism $f\colon d \to Fc$ in $\scD$, the composition
\[F\left(c\times_{F^{\rmR}F c}F^{\rmR}d\right) \xrightarrow{F(\pr_2)} F F^{\rmR}d \xrightarrow{\mathrm{counit}} d \textin \scD\]
is an equivalence.
\end{enumerate}
Then there exists a canonical adjunction of $\infty$-categories,
\[\xymatrix{
\Corr(F) \colon \Corr(\scC)  \ar@<.4ex>[r] & \ar@<.4ex>[l] \Corr(\scD) \colon \Corr(F^{\rmR}). }\]
\end{proposition}

\begin{proof}
Let $p\colon \scE \to \Delta^1$ be the cocartesian fibration associated to the functor $F\colon \scC \to \scD$ (regarded as a functor $\Delta^1 \to \Cat$).
Since $F\colon \scC \to \scD$ has a right adjoint, the cocartesian fibration $p\colon \scE \to \Delta^1$ is also a cartesian fibration.
Moreover, since $F\colon  \scC \to \scD$ preserves fiber products, so is $p\colon \scE \to \Delta^1$, by \cite[Cor.~4.3.1.11]{LurHTT}.
We claim that the induced functor
\[\Corr(p)\colon \Corr(\scE) \lra \Corr(\Delta^1).\]
is a cocartesian fibration (and hence also a cartesian fibration).
Indeed, by \Cref{Thm:Corr-cocart}, it suffices to show that:
for any objects $c\in \scC$, $d\in \scD$, and morphisms $f\colon d \to Fc$ in $\scD$, $x \to c\times_{F^{\rmR}F c}F^{\rmR}d$ in $\scC$, 
\[x \xrightarrow{\simeq} c\times_{F^{\rmR}F c}F^{\rmR}d \textin \scC\iff Fx \xrightarrow{\simeq}d \textin \scD.\]
This follows from the assumption (A2) and (A3) (in \Cref{Prop:Corr-Adj}).

Then we have a cartesian and cocartesian fibration
\begin{equation}\label{Eq:Corr-Adj-1}
\Corr(\scE,\scE\times_{\Delta^1}(\Delta^1)^{\simeq},\scE)\lra \Corr(\Delta^1,(\Delta^1)^{\simeq},\Delta^1)\simeq \Delta^1\end{equation}
since it is a pullback of the above functor $\Corr(p) \colon \Corr(\scE) \to \Corr(\Delta^1)$.
We claim that the cocartesian (resp. cartesian) straightening of the above fibration is $\Corr(F)\colon \Corr(\scC) \to \Corr(\scD)$ (resp. $\Corr(F^{\rmR})\colon \Corr(\scD) \to \Corr(\scC)$).
Indeed, consider the morphisms
\[\xymatrix{\scC \times \Delta^1 \ar[r] \ar@/_1pc/[rr]_-{F\times \id} & \scE \ar[r] & \scD \times \Delta^1}\]
of cocartesian fibrations over $\Delta^1$ (preserving cocartesian morphisms over $\Delta^1$), induced by \[( \scC \xrightarrow{\id_{\scC}}\scC) \xrightarrow{(\id_{\scC},F)} (\scC \xrightarrow{F}\scD) \xrightarrow{(F,\id_{\scD})} (\scD \xrightarrow{\id_{\scD}}\scD) \textin \Fun(\Delta^1,\Cat).\]
Applying $\Corr(-)\times_{\Corr(\Delta^1)}\Delta^1$, we obtain morphisms
\[\xymatrix{\Corr(\scC)\times \Delta^1 \ar[r] \ar@/_1.5pc/[rr]_-{\Corr(F)\times \id} & \Corr(\scE, \scE\times_{\Delta^1}(\Delta^1)^{\simeq},\scE) \ar[r] & \Corr(\scD)\times \Delta^1}\]
of cocartesian fibrations over $\Delta^1$ (preserving cocartesian morphisms over $\Delta^1$).
Since $\Corr \colon \Trip \to \Cat$ preserves limits, the cocartesian straightening of \eqref{Eq:Corr-Adj-1} is $\Corr(F)$.
For the cartesian straightening, we consider the morphisms
\[\xymatrix{\scD \times \Delta^1 \ar[r] \ar@/_1pc/[rr]_-{F^{\rmR}\times \id} & \scE \ar[r] & \scC \times \Delta^1}\]
of cartesian fibrations over $\Delta^1$.
Then the claim follows from an analogous argument.
\end{proof}

\subsection{Functoriality} 

Let $\scC$ be an $\infty$-category.
We can form an $\infty$-functor
\begin{equation}\label{Eq:Corr-1}
\scC \lra \Cat
\end{equation}
\begin{itemize}
\item sending an object $c$ in $\scC$ to the slice category $\scC_{/c}$, and
\item sending a morphism $f\colon c_1 \to c_2$ in $\scC$ to the forgetful functor $f_! \colon \scC_{/c_1}  \to \scC_{/c_2}$.
\end{itemize}
Indeed, this functor $\scC \to \Cat$ is the straightening of the cocartesian fibration $\ev_1 \colon \Fun(\Delta^1,\scC) \to \scC$.

Now assume that $\scC$ has fiber products.
By the unfurling construction (\Cref{Ex:Unfurling}),
the above functor \eqref{Eq:Corr-1} lifts to an $\infty$-functor
\begin{equation}\label{Eq:Corr-2}
\Corr(\scC) \lra \Cat,
\end{equation}
\begin{itemize}
\item sending an object $c$ in $\scC$ to the slice category $\scC_{/c}$, and
\item sending a morphism $c_1 \xleftarrow{s} a \xrightarrow{t} c_2$ in $\Corr(\scC)$ to the composition $t_! \circ s^* \colon \scC_{/c_1}  \to \scC_{/c_2}$, where $s^*$ is the right adjoint of $s_!$.
\end{itemize}
Since the slice categories $\scC_{/c}$ have fiber products and the forgetful/pullback functors preserve them, the above functor \eqref{Eq:Corr-2} factors through the subcategory $\Cat^{\fp} \subseteq \Cat$,
and we have an $\infty$-functor
\begin{equation}\label{Eq:Corr-3}
\Corr(\scC) \lra \Cat^{\fp}.
\end{equation}

\begin{construction}\label{Const:Corr-Funct}
Let $\scC$ be an $\infty$-category with fiber products.
We define the $\infty$-functor $\Corr({\scC_{/-}})$ as the composition
\[\Corr({\scC_{/-}}) \colon\Corr(\scC) \xrightarrow{\eqref{Eq:Corr-3}} \Cat^{\fp} \xrightarrow{\eqref{Eq:Corr-4}} \Trip \xrightarrow{\Corr} \Cat.\]
\end{construction}

\begin{itemize}
\item On objects, the functor $\Corr({\scC_{/-}})$ sends an object $c$ in $\scC$ to the correspondence category $\Corr(\scC_{/c})$.
\item On morphisms, the functor $\Corr({\scC_{/-}})$ sends a morphism $c_1 \xleftarrow{s} a \xrightarrow{t} c_2$ in $\Corr(\scC)$ to the composition $\Corr(t_!) \circ \Corr(s^*) \colon \Corr(\scC_{/c_1})  \to \Corr( \scC_{/c_2})$.
\end{itemize}

\begin{remark}
We expect that the functor $\Corr(\scC_{/-}) \colon \Corr(\scC) \to \Cat$ in \Cref{Const:Corr-Funct} is equivalent to the straightening of the cocartesian fibration $\Corr(\ev_1)\colon \Corr(\Fun(\Delta^1,\scC)) \to \Corr(\scC)$, but the authors do not know how to prove it.
It is not difficult to show that the two functors are equivalent on the objects and morphisms level, but we don't know how to prove the equivalence as functors.
\end{remark}

\begin{corollary}\label{Cor:Corr-Adj}
Let $\scC$ be an $\infty$-category with fiber products.
Let
\[\xymatrix@-1pc{& a\ar[ld]_s \ar[rd]^t & \\ c_1 && c_2}\]
be a morphism in $\Corr(\scC)$.
Then there exists a canonical adjunction of $\infty$-categories,
\[\xymatrix{
\Corr(t_! \circ s^*) \colon \Corr(\scC_{/c_1})  \ar@<.4ex>[r] & \ar@<.4ex>[l] \Corr(\scC_{/c_2}) \colon \Corr(s_! \circ t^*). }\]
\end{corollary}

\begin{proof}
By \Cref{Prop:Corr-Adj}, the functor $\Corr(s^*) \colon \Corr(\scC_{/c_1}) \to \Corr(\scC_{/a})$ is a right adjoint of $\Corr(s_!) \colon \Corr(\scC_{/a}) \to \Corr(\scC_{/c_1})$.
Since we have a canonical equivalence $\Corr(-)^{\op}\simeq \Corr(-)$ by \cite[Lem.~2.14]{HHLN}, $\Corr(s^*)$ is also a left adjoint of $\Corr(s_!)$.
Hence $\Corr(t_!\circ s^*) \simeq \Corr(t_!)\circ \Corr(s^*)$ is a left adjoint of $\Corr(s_!)\circ \Corr(t^*) \simeq \Corr(s_! \circ t^*)$.
\end{proof}

\subsection{Weil restrictions}

We provide one technical lemma on Weil restrictions used in \S\ref{ss:invariance}.

Let $\scC$ be an $\infty$-category with fiber products.
We say that $\scC$ has \defterm{Weil restrictions} if for any morphism $f\colon c_1 \to c_2$ in $\scC$,
the pullback functor $f^* \colon \scC_{/c_2} \to \scC_{/c_1}$ has a right adjoint $f_* \colon \scC_{/c_1} \to \scC_{/c_2}$.

\begin{lemma}\label{Lem:Corr-Cart}
Let $\scC$ be an $\infty$-category with fiber products and Weil restrictions.
Let $f\colon c_1 \to c_2$ and $g \colon d \to c_1$ be morphisms in $\scC$.
Then there exists a canonical cartesian square of $\infty$-categories
\[\xymatrix{
\Corr(\scC_{/f_*d}) \ar[r] \ar[d] \cart & \Corr(\scC_{/d}) \ar[d] \\
\Corr(\scC_{/c_2}) \ar[r] & \Corr(\scC_{/c_1}).
}\]	
\end{lemma}

\begin{proof}
Consider the canonical commutative diagram
\[\xymatrix@-.5pc{
f_*d \cart \ar[d]_{\overline{g}} & \ar[l]_{\tf} f^*f_*d \ar[r]^-{e} \ar[d] & d \ar[d]^{g} \\
c_2 & \ar[l]_f c_1 \ar@{=}[r] & c_1
}\]
in $\scC$, induced by the adjunction $f^* \dashv f_*$.
Since the left square is cartesian, 
the above commutative diagram can be regarded as a commutative square
\[\xymatrix@-1pc{
f_*d  \ar@{-->}[r]^-{} \ar@{-->}[d] & d \ar@{-->}[d] \\
c_2 \ar@{-->}[r] & c_1
}\]
in $\Corr(\scC)$.
Applying the functor $\scC_{/(-)}\colon\Corr(\scC) \to \Cat$ in \eqref{Eq:Corr-2},
we obtain a commutative square 
\[\xymatrix{
\scC_{/f_*d}  \ar[r]^-{e_! \circ \tf^*} \ar[d]_{\overline{g}_!} & \scC_{/d} \ar[d]^{g_!} \\
\scC_{/c_2} \ar[r]^-{f^*} & \scC_{/c_1}
}\]
of $\infty$-categories.
By inspection using the adjunction $f^* \dashv f_*$, the above commutative square is cartesian.
Applying the limit-preserving functor $\Cat^{\fp} \hookrightarrow \Trip \xrightarrow{\Corr} \Cat$, we obtain the desired cartesian square.
\end{proof}


\begin{thebibliography}{MNOP2}
\bibitem[AC]{AC} M.~Anel, D.~Calaque, \textit{Shifted symplectic reduction of derived critical loci}, Adv.~Theor.~Math.~Phys.~26 (2022), no.~6, 1543--1583.
\bibitem[AKLPR]{AKLPR} D.~Aranha, A.~Khan, A.~Latyntsev, H.~Park, C.~Ravi, \textit{Virtual localization revisited}, Adv.~Math.~479 (2025), 110434.
\bibitem[AB]{AB} L.~Amorin, O.~Ben-Bassat, \textit{Perversely categorified Lagrangian correspondences}, Adv.~Theor.~Math.~Phys.~21 (2017), no.~2, 289–381.
\bibitem[BKP]{BKP} Y.~Bae, M.~Kool, H.~Park, \textit{Counting surfaces on Calabi-Yau $4$-folds I: foundations}, to appear in Geom.~Topol.
\bibitem[Bar]{Bar} C.~Barwick, \textit{Spectral Mackey functors and equivariant algebraic K-theory (I)}, Adv.~Math.~304 (2017), 646--727.
\bibitem[Ben]{Ben} O.~Ben-Bassat, \textit{Multiple derived Lagrangian intersections}, In Stacks and categories in geometry, topology, and algebra, Vol.~643.~Contemp.~Math.~Amer.~Math.~Soc., Providence, RI, (2015), 119--126.
\bibitem[BFN]{BZFN} D.~Ben-Zvi, J.~Francis, D.~Nadler, \textit{Integral transforms and Drinfeld centers in derived algebraic geometry}, J.~Amer.~Math.~Soc.~23 (2010), no.~4, 909--966.
\bibitem[BP]{BP} F.~Binda, M.~Porta, \textit{Descent problems for derived Azumaya algebras}, arXiv:2107.03914v2.
\bibitem[BJ]{BJ} D.~Borisov, D.~Joyce, \textit{Virtual fundamental classes for moduli spaces of sheaves on Calabi-Yau four-folds}, Geom.~Topol.~21 (2017), no.~6, 3231--3311.
\bibitem[BBDJS]{BBDJS} C.~Brav, V.~Bussi, D.~Dupont, D.~Joyce, B.~Szendr\"{o}i, \textit{Symmetries and stabilization for sheaves of vanishing cycles}, With an appendix by Jörg Schürmann, J.~Singul.~{11}, (2015), 85--151.
\bibitem[BD]{BD} C.~Brav, T.~Dyckerhoff, \textit{Relative Calabi--Yau structures II : shifted Lagrangians in the moduli of objects}, Selecta Math. (N.S.) 27 (2021), no.~4, Paper No.~63, 45~pp.
\bibitem[BF]{BuFl} R.-O.~Buchweitz, H.~Flenner, \textit{A semiregularity map for modules and applications to deformations}, Compositio Math.~137 (2003), no.~2, 135–210.
\bibitem[Cal1]{Cal} D.~Calaque, \textit{Lagrangian structures on mapping stacks and semi-classical TFTs}, In Stacks and categories in geometry, topology, and algebra, Contemp.~Math., 643, American Mathematical Society, Providence, RI, (2015), 1--23.
\bibitem[Cal2]{Cal2} D.~Calaque, \textit{Shifted cotangent stacks are shifted symplectic}, Ann.~Fac.~Sci.~Toulouse Math.~(6)~28 (2019), no.~1, 67--90.
\bibitem[CHS]{CHS} D.~Calaque, R.~Haugseng, C.~Scheimbauer, \textit{The AKSZ construction in derived algebraic geometry as an extended topological field theory}, Mem.~Amer.~Math.~Soc. 308 (2025), no.~1555, v+173 pp.
\bibitem[CPTVV]{CPTVV} D.~Calaque, T.~Panted, B.~T\"{o}en, M.~Vaqui\'{e}, G. Vezzosi, \textit{Shifted Poisson structures and deformation quantization}, J.~Topol.~10 (2017), no.~2, 483--584.
\bibitem[CS]{CalSaf} D.~Calaque, P.~Safronov, \textit{Shifted cotangent bundles, symplectic groupoids and deformation to the normal cone}, arXiv:2407.08622.
\bibitem[COT1]{COT1} Y.~Cao, G.~Oberdieck, Y.~Toda, \textit{Gopakumar-Vafa type invariants of holomorphic symplectic 4-folds}, Comm.~Math.~Phys.~405 (2024), no.~2, Paper No.~26, 79~pp.
\bibitem[COT2]{COT2} Y.~Cao, G.~Oberdieck, Y.~Toda, \textit{Stable pairs and Gopakumar-Vafa type invariants on holomorphic symplectic 4-folds}, Adv.~Math.~408 (2022) 108605, 44 pp. 
\bibitem[CZ]{CZ} Y.~Cao, G.~Zhao, \textit{Quasimaps to quivers with potentials}, arXiv:2306.01302.
\bibitem[Del]{Del1} P.~Deligne, \textit{Th\'eor\`eme de Lefschetz et crit\`eres de d\'eg\'en\'erescence de suites spectrales}, Publ.~Math.~IHES 35 (1968), 259--278.
\bibitem[DG]{DG} V.~Drinfeld, D.~Gaitsgory, \textit{One some finiteness questions for algebraic stacks}, Geom.~Funct.~Anal.~{23} (2013), no.~1, 149--294.
\bibitem[GJT]{GJT} J.~Gross, D.~Joyce, Y.~Tanaka, \textit{Universal structures in $\C$-linear enumerative invariant theories}, SIGMA Symmetry Integrability Geom.~Methods Appl.~18 (2022), Paper No.~068, 61 pp.
\bibitem[Gro]{Gro} A.~Grothendieck, \textit{On the de Rham cohomology of algebraic varieties}, Inst.~Hautes Études Sci.~Publ.~Math.~No.~29 (1966), 95–103.
\bibitem[GP]{GP} O.~Gwilliam, D.~Pavlov, \textit{Enhancing the filtered derived category}, J.~Pure Appl.~Algebra 222 (2018), no.~11, 3621–3674.
\bibitem[Hau]{Hau} R.~Haugseng, \textit{Iterated spans and classical topological field theories}, Math.~Z.~289 (2018), no.~3-4, 1427–1488.
\bibitem[HHLN]{HHLN} R.~Haugseng, F.~Hebestreit, S.~Linskens, J.~Nuiten, Joost, \textit{Two-variable fibrations, factorisation systems and-categories of spans}, Forum Math.~Sigma 11 (2023), Paper No.~e111, 70 pp.
\bibitem[HSS]{HSS} M.~Hoyois, S.~Scherotzke, N.~Sibilla, \textit{Higher traces, noncommutative motives, and the categorified Chern character}, Adv.~Math.~309 (2017), 97--154.
\bibitem[KL]{KL} Y.-H.~Kiem, J.~Li, \textit{Categorification of Donaldson-Thomas invariants via perverse sheaves}, arXiv:1212.6444.
\bibitem[KP1]{KP} Y.-H.~Kiem, H.~Park, \textit{Localizing virtual cycles for Donaldson--Thomas invariants of Calabi--Yau 4-folds}, J.~Algebraic Geom.~32 (2023), no.~4, 585–639.
\bibitem[KP2]{KP2} Y.-H.~Kiem, H.~Park, \textit{Cosection localization via shifted symplectic geometry}, arXiv:2504.19542.
\bibitem[KPS]{KPS} T.~Kinjo, H.~Park, P.~Safronov, \textit{Cohomological Hall algebras for 3-Calabi-Yau categories}, arXiv:2406.12838.
\bibitem[LM]{LM} G.~Laumon, L.~Moret-Bailly, \textit{Champs algébriques}, Ergeb.~Math.~Grenzgeb.~(3), 39, Springer-Verlag, Berlin, 2000, xii+208 pp.
\bibitem[Lur1]{LurHTT} J.~Lurie, \textit{Higher topos theory}, Ann.~of Math.~Stud.~170, Princeton University Press, Princeton, NJ, 2009, xviii+925 pp.
\bibitem[Lur2]{LurHA} J.~Lurie, \textit{Higher algebra}, 2017, url: https://www.math.ias.edu/~lurie/papers/HA.pdf.
\bibitem[Lur3]{LurSAG} J.~Lurie, \textit{Spectral algebraic geometry}, 2018, url: https://www.math.ias.edu/~lurie/papers/SAG-rootfile.pdf.
\bibitem[Mou]{Mou} T.~Moulinos, \textit{The geometry of filtrations}, Bull.~Lond.~Math.~Soc.~53 (2021), no.~5, 1486–1499.
\bibitem[OT]{OT} J.~Oh, R.~P.~Thomas, \textit{Counting sheaves on Calabi-Yau 4-folds,~I}, Duke Math.~J.~172 (2023), no.~7, 1333–1409.
\bibitem[PTVV]{PTVV} T.~Pantev, B.~T\"{o}en, M.~Vaqui\'{e}, G. Vezzosi, \textit{Shifted symplectic structures}, Publ.~Math.~Inst.~Hautes Etudes Sci.~117 (2013), 271--328.
\bibitem[Par1]{Park1} H.~Park, \textit{Virtual pullbacks in Donaldson-Thomas theory of Calabi-Yau $4$-folds}, arXiv:2110.03631.
\bibitem[Par2]{Park2} H.~Park, \textit{Shifted symplectic pushforwards}, arXiv:2406.19192.
\bibitem[PY]{PY} H.~Park, J.~You, \textit{An introduction to shifted symplectic structures}, Moduli spaces, virtual invariants and shifted symplectic structures (2025), 37--64.
\bibitem[Saf]{Saf1} P.~Safronov, \textit{Quasi-Hamiltonian reduction via classical Chern--Simons theory}, Adv.~Math.~287 (2016), 733--773.
\bibitem[STV]{STV} T.~Sch\"{u}rg, B.~T\"{o}en, G.~Vezzosi, \textit{Derived algebraic geometry, determinants of perfect complexes, and applications to obstruction theories for maps and complexes}, J.~Reine Angew.~Math.~702 (2015), 1--40.
\bibitem[Sta]{Sta} The Stacks project authors, \textit{The Stacks project}, url: https://stacks.math.columbia.edu/.
\bibitem[TV1]{TV1} B.~T\"{o}en, G.~Vezzosi, \textit{Homotopical algebraic geometry II : geometric stacks and applications}, Mem.~Amer.~Math.~Soc.~193 (2008), no.~902, x+224 pp.
\bibitem[TV2]{TV2} B.~T\"{o}en, G.~Vezzosi, \textit{Caractères de Chern, traces équivariantes et géométrie algébrique dérivée}, Selecta Math. (N.S.) 21 (2015), no. 2, 449–554.
\bibitem[Voi]{Voi} C.~Voisin, \textit{Hodge theory and complex algebraic geometry. II}, Translated from the French by Leila Schneps.
Cambridge Stud. Adv. Math., 77.
Cambridge University Press, Cambridge, 2003. x+351 pp.
\bibitem[You]{You} J.~You, \textit{Reduced pairs/sheaves correspondences}, in preparation.
\end{thebibliography}
\end{document}